\documentclass[12pt,twoside, reqno]{amsart}
\usepackage{amssymb}
\usepackage{amscd}
\usepackage{amsmath}
\usepackage{xypic}

\title[VMHS and semipositivity]{Variations of mixed Hodge structure and semipositivity theorems}
\author{Osamu Fujino}
\author{Taro Fujisawa} 
\date{2014/3/17, version 5.04}
\address{Department of Mathematics, Faculty of Science, 
Kyoto University, Kyoto 606-8502, Japan}
\email{fujino@math.kyoto-u.ac.jp}
\address{Tokyo Denki University, School of Engineering, 
Department of Mathematics, Tokyo, Japan}
\email{fujisawa@mail.dendai.ac.jp}
\newcommand{\red}[0]{{\operatorname{red}}}

\newcommand{\codim}[0]{{\operatorname{codim}}}
\newcommand{\Gr}[0]{{\operatorname{Gr}}}
\newcommand{\Coker}[0]{{\operatorname{Coker}}}

\newcommand{\xO}[0]{{\operatorname{\mathcal O}}}

\newcommand{\Ker}[0]{{\operatorname{Ker}}}
\newcommand{\Exc}[0]{{\operatorname{Exc}}}
\newcommand{\Supp}[0]{{\operatorname{Supp}}}

\newcommand{\im}[0]{{\operatorname{Im}}}

\newcommand{\RomI}{\uppercase\expandafter{\romannumeral 1}\ }
\newcommand{\RomII}{\uppercase\expandafter{\romannumeral 2}\ }
\newcommand{\RomIV}{\uppercase\expandafter{\romannumeral 4}}

\newcommand{\pv}{polarizable variation }
\newcommand{\gp}{graded polarizable }
\newcommand{\gpv}{graded polarizable variation }
\newcommand{\snc}{simple normal crossing }

\newcommand{\xC}{\operatorname{\mathcal{C}}}
\newcommand{\xF}{\operatorname{\mathcal{F}}}
\newcommand{\xG}{\operatorname{\mathcal{G}}}
\newcommand{\xP}{\operatorname{\mathcal{P}}}

\newcommand{\xV}{\operatorname{\mathcal{V}}}

\newcommand{\bC}{\mathbb C}

\newcommand{\bP}{\mathbb P}
\newcommand{\bQ}{\mathbb Q}
\newcommand{\bR}{\mathbb R}
\newcommand{\bV}{\mathbb V}

\newcommand{\css}{^{\bullet}}
\newcommand{\sso}{_{\bullet}}
\newcommand{\zo}[1]{{#1}^{\ast}}
\newcommand{\pd}{\Delta\!^{\ast}}
\newcommand{\gdm}{_{\rm Gdm}}

\DeclareMathOperator{\id}{id}

\DeclareMathOperator{\rec}{rec}
\DeclareMathOperator{\res}{Res}
\DeclareMathOperator{\shom}{{\mathcal H} {\it om}}

\newtheorem{thm}{Theorem}[section]
\newtheorem{lem}[thm]{Lemma}
\newtheorem{cor}[thm]{Corollary}
\newtheorem{prop}[thm]{Proposition}

\theoremstyle{definition}
\newtheorem{ex}[thm]{Example}
\newtheorem{defn}[thm]{Definition}
\newtheorem{rem}[thm]{Remark}

\newtheorem{step}{Step}
\newtheorem*{ack}{Acknowledgments}       
\newtheorem*{notation}{Notation}         
\newtheorem{notationnum}[thm]{Notation}       
\newtheorem{say}[thm]{}
\newcounter{newitem}[thm]
\renewcommand{\thenewitem}{\thethm.\arabic{newitem}}

\newenvironment{newitemize}{
\setcounter{newitem}{\value{equation}}
\begin{list}{}{%
\setlength{\topsep}{5pt}%
\setlength{\parsep}{5pt}%
\setlength{\itemsep}{0pt}%
\setlength{\leftmargin}{50pt}%
\setlength{\labelwidth}{50pt}%
}}
{\end{list}}

\newcommand{\itemno}{
\refstepcounter{newitem}
\item[{\rm (\thenewitem)}]
\setcounter{equation}{\value{newitem}}}

\newcounter{casesinthm}[thm]
\renewcommand{\thecasesinthm}{\rm (\roman{casesinthm})}

\newenvironment{subthm}{
\refstepcounter{casesinthm}
\thecasesinthm
\it}

\newenvironment{subdefn}{
\refstepcounter{casesinthm}
\thecasesinthm}

\numberwithin{equation}{thm}

\begin{document}
\bibliographystyle{amsalpha+}

\maketitle 

\begin{abstract}
We discuss the variations of mixed Hodge structure 
for cohomology with compact support of quasi-projective simple normal 
crossing pairs. We show that they are graded polarizable admissible variations of 
mixed Hodge structure. 
Then we prove a generalization of the Fujita--Kawamata 
semipositivity theorem. 
\end{abstract}

\tableofcontents
\section{Introduction}
Let $X$ be a simple normal crossing 
divisor on a smooth projective variety $M$ and 
let $B$ be a simple normal crossing 
divisor on $M$ such that $X+B$ is 
simple normal crossing on $M$ and that $X$ and $B$ have no common irreducible 
components. 
Then the pair $(X, D)$, 
where $D=B|_X$, 
is a typical example of simple normal crossing 
pairs. In this situation, a stratum of $(X, D)$ is an irreducible component 
of $T_{i_1}\cap \cdots \cap T_{i_k}\subset X$ for some 
$\{i_1, \cdots, i_k\}\subset I$, where $X+B=\sum _{i\in I}T_i$ is the irreducible decomposition of $X+B$. 
For the precise definition of simple normal crossing pairs, 
see Definition \ref{03} below.  
We note that 
simple normal crossing pairs frequently appear in the 
study of the log minimal model program for higher dimensional 
algebraic varieties with bad singularities. 
The first author has already investigated 
the mixed Hodge structures 
for $H^{\bullet}_c(X\setminus D, \mathbb Q)$ 
in \cite[Chapter 2]{book} to obtain various vanishing theorems (see also \cite{fuj-inj}). 
In this paper, we show that their variations are graded polarizable admissible 
variations of mixed Hodge structure.  
Then we prove a generalization of the Fujita--Kawamata 
semipositivity theorem. Our formulation of 
the Fujita--Kawamata semipositivity theorem is 
different from Kawamata's original one. 
However, it is more suited for our studies of 
simple normal crossing pairs. 

The following theorem is a corollary of 
Theorem \ref{main} and Theorem \ref{main2}, which are our main results of this paper 
(cf.~\cite[Theorem 5]{kawamata1}, \cite[Theorem 2.6]{ko2}, \cite[Theorem 1]{n}, 
\cite[Theorems 3.4 and 
3.9]{high}, \cite[Theorem 1.1]{kawamata}, 
and so on). 

\begin{thm}[Semipositivity theorem
(cf.~Theorem \ref{main} and Theorem \ref{main2})]\label{11}
Let $(X, D)$ be a simple normal crossing pair such that 
$D$ is reduced and let $f:X\longrightarrow Y$ be 
a projective surjective morphism onto a smooth 
complete algebraic variety $Y$.
Assume that every stratum of $(X,D)$ is dominant onto $Y$. 
Let $\Sigma$ be a simple normal crossing divisor on $Y$ such that 
every stratum of $(X, D)$ is smooth over $\zo{Y}=Y\setminus \Sigma$. 
Then $R^pf_*\omega_{X/Y}(D)$ is locally free for every $p$. 
We put $\zo{X}=f^{-1}(\zo{Y})$,
$\zo{D}=D|_{\zo{X}}$, and $d=\dim X-\dim Y$. 
We further assume that 
all the local monodromies on
$R^{d-i}(f|_{X^*\setminus D^*})_!\mathbb Q_{\zo{X} \setminus \zo{D}}$ 
around $\Sigma$ 
are unipotent. 
Then we obtain that
$R^if_*\omega_{X/Y}(D)$ is 
a semipositive locally free sheaf on $Y$. 
\end{thm}

We note the following definition. 

\begin{defn}[Semipositivity in the sense of Fujita--Kawamata]
A locally free sheaf $\mathcal E$ of finite rank on a complete algebraic variety 
$X$ is said to be semipositive (in the sense of Fujita--Kawamata) if and only if 
$\mathcal O_{\mathbb P_X(\mathcal E)}(1)$ is nef on 
$\mathbb P_{X}(\mathcal E)$. 
\end{defn}

In \cite{kawamata}, Kawamata obtained a weaker result 
similar to Theorem \ref{11} (see \cite[Theorem 1.1]{kawamata}). 
It is not surprising 
because both \cite{kawamata} and this paper grew from the same question 
raised by Valery Alexeev and Christopher Hacon. 
For details on Kawamata's approach, we recommend the reader 
to see \cite{ffs}, where we use the theory of mixed Hodge modules 
to give an alternative proof of Theorem \ref{11}. 
Note that \cite{ffs} was written after this paper was circulated. 

The semipositivity of $R^if_*\omega_{X/Y}(D)$ in Theorem \ref{11} 
follows from a purely Hodge theoretic semipositivity theorem:~Theorem \ref{semi-po}. 
In the proof of Theorem \ref{11}, we use the semipositivity 
of $(\Gr_F^a\mathcal V)^*$ 
in Theorem \ref{12}, where $(\Gr_F^a\mathcal V)^*=\mathcal Hom_{\mathcal O_Y}(\Gr_F^a\mathcal 
V, \mathcal O_Y)$. We do not need 
the semipositivity of $F^b\mathcal V$ in Theorem \ref{12} for Theorem \ref{11}. 
For more details, see the discussion in \ref{16ske} below.

\begin{thm}[Hodge theoretic semipositivity theorem (cf.~Theorem \ref{semi-po})]\label{12}
Let $X$ be a smooth complete complex algebraic variety,
$D$ a \snc divisor on $X$,
$\xV$ a locally free $\xO_X$-module of finite rank
equipped with a finite increasing filtration $W$
and a finite decreasing filtration $F$.
We assume the following$:$
\begin{enumerate}
\item
$F^a\xV=\xV$ and $F^{b+1}\xV=0$ for some $a < b$.
\item
$\Gr_F^p\Gr_m^W\xV$ is a locally free $\xO_X$-module of finite rank
for all $m,p$.
\item
For all $m$,
$\Gr_m^W\xV$ admits an integrable logarithmic connection
$\nabla_m$ with the nilpotent residue morphisms 
which satisfies the conditions
$$\nabla_m(F^p\Gr_m^W\xV) \subset \Omega_X^1(\log D) \otimes F^{p-1}\Gr_m^W\xV$$
for all $p$.
\item
The triple 
$(\Gr_m^W\xV, F\Gr_m^W\mathcal V, \nabla_m)|_{X \setminus D}$
underlies a polarizable variation of $\bR$-Hodge structure
of weight $m$ for every integer $m$.
\end{enumerate}
Then $(\Gr_F^a\xV)^*$ and $F^b\xV$ are semipositive.
\end{thm}

In this paper, we concentrate on the Hodge theoretic aspect of the Fujita--Kawamata 
semipositivity theorem (cf.~\cite{zucker1}, \cite{kawamata1}, 
\cite{ko2}, \cite{n}, \cite{high}, and \cite{ffs}). 
On the other hand, there are many results related to 
the Fujita--Kawamata semipositivity theorem from the analytic viewpoint 
(cf.~\cite{fujita}, \cite{bern}, \cite{bern-p}, \cite{mt}, and so on). 
Note that Griffiths's pioneering work on the variation of Hodge structure (cf.~\cite{grif}) 
is a starting point of the Fujita--Kawamata semipositivity theorem. 
For a related topic, see \cite{mochizuki}. 
Mochizuki's approach is completely different from 
ours and has more arithmetic geometrical flavors. 

As a special case of Theorem \ref{11}, we 
obtain the following theorem:~Theorem \ref{thm42}. 

\begin{thm}[{cf.~\cite[Theorem 5]{kawamata1}, 
\cite[Theorem 2.6]{ko2}, 
and \cite[Theorem 1]{n}}]\label{thm12}
Let $f:X\longrightarrow Y$ be a projective morphism 
between smooth complete 
algebraic varieties which satisfies the following conditions{\em{:}}
\begin{itemize}
\item[(i)] There is a Zariski open subset $\zo{Y}$ of $Y$ such that 
$\Sigma =Y \setminus \zo{Y}$ is a simple normal crossing 
divisor on $Y$. 
\item[(ii)] We put $\zo{X}=f^{-1}(\zo{Y})$. Then $f|_{X^*}$ is smooth.  
\item[(iii)] The local monodromies of $R^{d+i}(f|_{X^*})_{*}\mathbb C_{\zo{X}}$
around $\Sigma$ are unipotent,
where $d=\dim X-\dim Y$. 
\end{itemize}
Then $R^if_*\omega_{X/Y}$ is a semipositive locally free sheaf on $Y$. 
\end{thm}

We note that 
Theorem \ref{thm12} was first proved by Kawamata (cf.~\cite[Theorem 5]{kawamata1}) 
under the extra assumptions that 
$i=0$ and that $f$ has connected fibers. 
The above statement 
follows from \cite[Theorem 2.6]{ko2} or \cite[Theorem 1]{n} (see also \cite[Theorem 5.4]{fujino-rem}). 
We also note that, by Poincar\'e--Verdier duality, $R^{d+i}(f|_{X^*})_*\mathbb C_{X^*}$ 
is the dual local system of $R^{d-i}(f|_{X^*})_*\mathbb C_{X^*}$ in 
Theorem \ref{thm12}.
In \cite{ko2} and \cite{n}, the variation of Hodge 
structure on $R^{d+i}(f|_{X^*})_*\mathbb C_{X^*}$ is investigated
for the proof of Theorem \ref{thm12}. 
On the other hand, in this paper, 
we concentrate on the variation of Hodge 
structure on $R^{d-i}(f|_{X^*})_*\mathbb C_{X^*}$ 
for Theorem \ref{thm12}. 

The following example shows
that the assumption $(2)$ in Theorem \ref{12}
is indispensable. 
For related examples, see \cite[(3.15) and (3.16)]{sz}. 
In the proof of Theorem \ref{11},
the admissibility of the \gpv of $\bQ$-mixed Hodge structure
on $R^{d-i}(f|_{\zo{X} \setminus \zo{D}})_{!}\bQ_{\zo{X} \setminus \zo{D}}$,
which is proved in Theorem \ref{GPVMHS for snc pair}, 
assures us the existence of the extension of the Hodge filtration
satisfying the assumption $(2)$. 
Note that the notion of {\em{admissibility}} 
is due to Steenbrink--Zucker \cite{sz} and Kashiwara \cite{kashiwara}. 

\begin{ex}\label{not admissible} 
Let $\bV$ be a $2$-dimensional $\mathbb Q$-vector space with basis $\{e_1, e_2\}$. 
We give an increasing filtration $W$ on $\bV$
by $W_{-1}\bV=0$, $W_0\bV=W_1\bV=\mathbb Qe_1$, and 
$W_2\bV=\bV$. 
The constant sheaf on $\bP^1$ whose 
fibers are $\bV$ is denoted by the same letter $\bV$,
on which an increasing filtration $W$ is given as above.
We consider $\xV=\xO_{\mathbb P^1} \otimes \bV=\xO e_1 \oplus \xO e_2$ on $\mathbb P^1$.
We set a decreasing filtration $F$ on $\xV|_{\bC^{\ast}}$ by
\begin{equation*}
F^0(\xV|_{\bC^{\ast}})=\xV|_{\bC^{\ast}},\ 
F^1(\xV|_{\bC^{\ast}})=\xO_{\bC^{\ast}}(t^{-1}e_1+e_2),\ 
F^2(\xV|_{\bC^{\ast}})=0,
\end{equation*}
where $t$ is the coordinate function of $\bC \subset \mathbb P^1$.
We can easily check that $((\bV,W)|_{\bC^{\ast}},(\xV|_{\bC^{\ast}},F))$ is
a graded polarizable variation of $\mathbb Q$-mixed Hodge 
structure on $\mathbb C^*$.
In this case, we can not extend the Hodge filtration $F$ on $\xV|_{\bC^{\ast}}$
to the filtration $F$ on $\xV$
satisfying the assumption $(2)$ in Theorem \ref{12}.
In particular, the above variation of $\mathbb Q$-mixed Hodge structure is not admissible. 

We note that we can extend the Hodge filtration
$F$ on $\xV|_{\bC^{\ast}}$
to the filtration $F$ on $\xV$ such that 
$F^2\xV=0$, 
$F^1\xV \simeq \mathcal O_{\mathbb P^1}(-1)$,
and $F^0\xV=\xV$ with
$\Gr_F^0\xV \simeq \xO_{\mathbb P^1}(1)$. 
In this case, $F^1\xV$ and 
$(\Gr_F^0\xV)^*$ are not semipositive. 
This means that 
a naive generalization of the Fujita--Kawamata semipositivity theorem for graded 
polarizable 
variations of $\mathbb Q$-mixed Hodge structure is false. 
\end{ex}

As an application of Theorem \ref{11}, the first author proves 
a semipositivity theorem for families of singular 
varieties in \cite{fujino-semipositive}. It is a generalization of \cite[4.12.~Theorem]{ko-proj} 
and implies that the moduli functor of stable varieties is semipositive in the 
sense of Koll\'ar (see \cite[2.4.~Definition]{ko-proj}). Therefore, it 
will play crucial roles for the projectivity of the moduli spaces of 
higher-dimensional 
stable varieties. For details, see \cite{ko-proj}, \cite{fujino-slc}, and \cite{fujino-semipositive}. 

We give a sketch of the proof of Theorem \ref{11} for the 
reader's convenience. 

\begin{say}[{Sketch of the proof of Theorem \ref{11}}]\label{16ske} 
In Theorem \ref{11}, we see that 
the local system $R^{d-i}(f|_{\zo{X} \setminus \zo{D}})_!\mathbb Q_{\zo{X} \setminus \zo{D}}$
underlies 
an admissible variation of $\mathbb Q$-mixed Hodge structure
by Theorem \ref{GPVMHS for snc pair}.
Let $\mathcal V$ be the canonical extension of 
$(R^{d-i}(f|_{\zo{X} \setminus \zo{D}})_!\mathbb Q_{\zo{X} \setminus \zo{D}})
\otimes \mathcal O_{\zo{Y}}$. Then 
we can prove $R^{d-i}f_*\mathcal O_X(-D)\simeq \Gr _F^0\mathcal V$ where 
$F$ is the canonical extension of the Hodge filtration. 
Note that the {\em{admissibility}}
ensures the existence of the canonical extensions of the 
Hodge bundles (cf.~Proposition \ref{extension of Hodge filtration} 
and Remark \ref{upper-lower}). 
We also note that we use an explicit description of the canonical extension of the 
Hodge filtration in order to prove $R^{d-i}f_*\mathcal O_X(-D)\simeq \Gr_F^0\mathcal V$ when 
$Y$ is a curve. 
By Grothendieck duality, we obtain 
$R^if_*\omega_{X/Y}(D)\simeq (\Gr_F^0\mathcal V)^*$. 
Therefore, $R^if_*\omega_{X/Y}(D)$ is semipositive by Theorem \ref{12}. 
It is very important to note that the local system
$R^{d-i}(f|_{\zo{X} \setminus \zo{D}})_!\mathbb Q_{\zo{X} \setminus \zo{D}}$
is not necessarily the dual local system 
of $R^{d+i}(f|_{\zo{X} \setminus \zo{D}})_*\mathbb Q_{\zo{X} \setminus \zo{D}}$ because 
$X$ is not a smooth variety but a simple normal crossing variety. 
In the proof of Theorem \ref{11}, we use the recent developments of the theory of 
partial resolution of singularities for {\em{reducible}} varieties (see \cite{bierstone-milman} and \cite{bierstone-p}) 
for the reduction to simpler cases. 
\end{say}

We quickly explain the reason why we use mixed Hodge structures 
for cohomology with compact support. 

\begin{say}[Mixed Hodge structure for cohomology with compact support]\label{1616}
Let $X$ be a smooth projective variety 
and let $D$ be a simple normal crossing divisor on $X$. 
After Iitaka introduced the notion of logarithmic Kodaira dimension, 
$\mathcal O_X(K_X+D)$ plays important roles in the birational 
geometry, where $K_X$ is the canonical 
divisor of $X$. 
In the traditional birational 
geometry, $\mathcal O_X(K_X+D)$ is recognized 
to be $\Omega^{\dim X}_X(\log D)$. 
Therefore, the Hodge to de Rham spectral 
sequence 
$$
E^{p,q}_1=H^q(X, \Omega^p_X(\log D))\Longrightarrow 
H^{p+q}(X\setminus D, \mathbb C)
$$
arising from the mixed Hodge structures 
on $H^{\bullet}(X\setminus D, \mathbb C)$ is useful. 
The first author recognizes 
$\mathcal O_X(K_X+D)$ as 
$$\mathcal Hom_{\mathcal O_X}(\mathcal O_X(-D), \mathcal O_X(K_X))$$ 
or $$R\mathcal Hom_{\mathcal O_X}(\mathcal O_X(-D), 
\omega^{\bullet}_X)[-\dim X]$$ where 
$\omega^{\bullet}_X=\mathcal O_X(K_X)[\dim X]$ is the dualizing 
complex of $X$. 
Furthermore, $\mathcal O_X(-D)$ can be 
interpreted as the $0$-th 
term of the complex 
$$
\Omega^{\bullet}_X(\log D)\otimes \mathcal O_X(-D). 
$$ 
By this observation, we can use the following 
Hodge to de Rham spectral 
sequence 
$$
E^{p,q}_1=H^q(X, \Omega^p_X(\log D)\otimes \mathcal O_X(-D))\Longrightarrow 
H^{p+q}_c(X\setminus D, \mathbb C)
$$ 
arising from the mixed Hodge structures on 
cohomology groups $H^{\bullet}_c(X\setminus D, \mathbb C)$ of $X\setminus D$ 
with compact support 
and obtain various powerful vanishing theorems. 
For details and many applications, see \cite[Chapter 2]{book}, 
\cite{fujino-on}, 
\cite{fujino-quasi}, 
\cite[Section 5]{fujino-fund}, \cite{fujino-surface}, \cite{fujino-slc}, 
and \cite{fuj-inj}. 
Therefore, it is very natural to consider the variations of 
such mixed Hodge structures. 
\end{say}

We summarize the contents of this paper. 
Section \ref{sec2} is a preliminary section. 
Section \ref{generalities on VMHS} collects some generalities on 
variations of mixed Hodge structure. 
In Section \ref{VMHS of geometric origin},
we discuss the variations of mixed Hodge structure for 
simple normal crossing pairs. 
We show that 
they are graded polarizable 
and admissible. 
Theorem \ref{GPVMHS for snc pair} is the main result of 
Section \ref{VMHS of geometric origin}. 
In Section \ref{sub42}, we discuss a purely Hodge theoretic 
aspect of the Fujita--Kawamata semipositivity theorem. 
Our formulation is different from 
Kawamata's original one but 
is suited for our results in Section \ref{sec4}. 
In Section \ref{sec-quasi-proj}, we discuss 
some generalizations of vanishing and torsion-free theorems 
for quasi-projective simple normal crossing pairs. 
They are necessary for the arguments in Section \ref{sec4}. 
Section \ref{sec4} is the main part of this paper. 
Here, we characterize higher direct images of 
log canonical divisors by using canonical 
extensions of Hodge bundles (cf.~Theorem \ref{main} and Theorem \ref{main2}). 
It is a generalization of the results 
by Yujiro Kawamata, Noboru Nakayama, J\'anos Koll\'ar, Morihiko Saito, and Osamu Fujino. 
In Section \ref{sec-final}, we treat some examples, which 
help us understand the Fujita--Kawamata semipositivity theorem, 
Viehweg's weak positivity theorem, and so on, in details. 

Let us recall basic definitions and 
notation. 

\begin{notation} 
For a proper morphism $f:X\longrightarrow Y$, the {\em{exceptional 
locus}}, which is denoted by $\Exc(f)$, 
is the locus where 
$f$ is not an isomorphism. 
$\mathbb R$ (resp.~$\mathbb Q$) denotes 
the set of real (resp.~rational) 
numbers. 
$\mathbb Z$ denotes the set of integers.  
\end{notation}

\begin{say}[Divisors, $\mathbb Q$-divisors, and $\mathbb R$-divisors]
For an $\mathbb R$-Weil divisor 
$D=\sum _{j=1}^r d_j D_j$ such that 
$D_i$ is a prime divisor for 
every $i$ and that $D_i\ne D_j$ for $i\ne j$, we define 
the {\em{round-up}} 
$\lceil D\rceil =\sum _{j=1}^{r} 
\lceil d_j\rceil D_j$ 
(resp.~the {\em{round-down}} $\lfloor D\rfloor 
=\sum _{j=1}^{r} \lfloor d_j \rfloor D_j$), 
where for any real number $x$, 
$\lceil x\rceil$ (resp.~$\lfloor x\rfloor$) is 
the integer defined by $x\leq \lceil x\rceil <x+1$ 
(resp.~$x-1<\lfloor x\rfloor \leq x$). 
The {\em{fractional part}} $\{D\}$ of $D$ denotes $D-\lfloor D\rfloor$. 
We call $D$ a {\em{boundary}} (resp.~{\em{subboundary}}) 
$\mathbb R$-divisor 
if $0\leq d_j\leq 1$ (resp.~$d_j\leq 1$) for every $j$. 
{\em{$\mathbb Q$-linear equivalence}} (resp.~{\em{$\mathbb R$-linear 
equivalence}})\index{$\sim_{\mathbb Q}$, 
$\mathbb Q$-linear equivalence} 
of two $\mathbb Q$-divisors (resp.~{\em{$\mathbb R$-divisors}}) 
$B_1$ and $B_2$ is denoted by $B_1\sim _{\mathbb Q}B_2$ (resp.~$B_1
\sim_{\mathbb R}B_2$). 
\end{say}

\begin{say}[Singularities of pairs]
Let $X$ be a normal variety and let $\Delta$ be an 
effective $\mathbb R$-divisor on $X$ 
such that $K_X+\Delta$ is $\mathbb R$-Cartier. 
Let $f:X\longrightarrow Y$ be 
a resolution such that $\Exc(f)\cup f^{-1}_*\Delta$ 
has a simple normal crossing support, 
where $f^{-1}_*\Delta$ is 
the strict transform of $\Delta$ on $Y$. We can 
write 
$$K_Y=f^*(K_X+\Delta)+\sum _i a_i E_i. 
$$
We say that $(X, \Delta)$ 
is {\em{log canonical}} ({\em{lc}}, for short) if $a_i\geq -1$ for every $i$. 
We usually write $a_i= a(E_i, X, \Delta)$
and call it the {\em{discrepancy coefficient}} of 
$E$ with respect to $(X, \Delta)$. 

If $(X, \Delta)$ is log canonical and 
there exist a resolution $f:Y\longrightarrow X$ and a divisor $E$ on $Y$ such 
that $a(E, X, \Delta)=-1$, then $f(E)$ is called a 
{\em{log canonical center}} ({\em{lc center}}, for short) with respect to $(X, \Delta)$. 
\end{say}

It is very important to understand the following example. 

\begin{say}[A basic example]\label{110basic}
Let $X$ be a smooth variety and let $\Delta$ be a reduced simple normal crossing 
divisor on $X$. 
Then the pair $(X, \Delta)$ is log canonical. 
Let $\Delta=\sum _{i\in I}\Delta_i$ be the irreducible decomposition of $\Delta$. 
Then a subvariety $W$ of $X$ is a log canonical center 
with respect to $(X, \Delta)$ if and only if 
$W$ is an irreducible component of $\Delta_{i_1}\cap \cdots \cap \Delta_{i_k}$ for some 
$\{i_1, \cdots, i_k\}\subset I$. 
\end{say}

We will work over $\mathbb C$, 
the field of complex numbers, throughout this paper. 

\begin{ack}
The first author was partially supported by The Inamori 
Foundation and by the Grant-in-Aid for Young Scientists (A) 
$\sharp$24684002 from JSPS. 
He would like to thank Professors Takeshi Abe,  
Hiraku Kawanoue, Kenji Matsuki, 
and Shigefumi Mori for discussions. 
He also thanks Professors 
Valery Alexeev and Christopher Hacon for 
introducing him this problem. 
He thanks Professor Gregory James Pearlstein and 
Professor Fouad El Zein for answering his questions and giving him some useful comments. 
He thanks Professor Kazuya Kato for fruitful discussions. 
The authors would like to thank Professors 
Kazuya Kato, Chikara Nakayama, Sampei Usui, and Morihiko Saito 
for discussions and 
comments. 
\end{ack}

\section{Preliminaries}\label{sec2}

Let us recall the definition of {\em{simple normal crossing 
pairs}}. 

\begin{defn}[Simple normal crossing pairs]\label{03} 
We say that the pair $(X, D)$ is {\em{simple normal crossing}} at 
a point $a\in X$ if $X$ has a Zariski open neighborhood $U$ of $a$ that can be embedded in a smooth 
variety 
$Y$, 
where $Y$ has regular system of parameters $(x_1, \cdots, x_p, y_1, \cdots, y_r)$ at 
$a=0$ in which $U$ is defined by a monomial equation 
$$
x_1\cdots x_p=0
$$ 
and $$
D=\sum _{i=1}^r \alpha_i(y_i=0)|_U, \quad  \alpha_i\in \mathbb R. 
$$ 
We say that $(X, D)$ is a {\em{simple normal crossing pair}} if it is simple 
normal crossing at every point of $X$. 
When $D$ is the zero divisor for a simple normal crossing 
pair $(X, D)$, 
$X$ is called a {\em{simple normal crossing variety}}.
If $(X,D)$ is a simple normal crossing pair, then $X$ has only Gorenstein singularities. 
Thus, it has an invertible dualizing sheaf $\omega_X$. 
Therefore, we can define the {\em{canonical divisor $K_X$}} such that 
$\omega_X\simeq \mathcal O_X(K_X)$. 
It is a Cartier divisor on $X$ and is well-defined up to linear equivalence. 

We say that a simple normal crossing 
pair $(X, D)$ is {\em{embedded}} 
if there exists a closed embedding 
$\iota:X\hookrightarrow M$, where 
$M$ is a smooth variety 
of dimension $\dim X+1$. 

Let $X$ be a simple normal crossing variety and 
let $D$ be a Cartier divisor on $X$. 
If $(X, D)$ is a simple normal crossing pair and 
$D$ is reduced, 
then $D$ is called a {\em{simple normal crossing divisor}} on $X$. 
\end{defn}

We note that a simple normal crossing pair is called a {\em{semi-snc pair}} 
in \cite[Definition 1.10]{kollar-book}. 

\begin{defn}[Strata and permissibility]\label{stra} 
Let $X$ be a simple normal crossing variety and let $X=\bigcup _{i\in I}X_i$ be the 
irreducible decomposition of $X$. 
A {\em{stratum}} of $X$ is an irreducible component of $X_{i_1}\cap \cdots \cap X_{i_k}$ for some 
$\{i_1, \cdots, i_k\}\subset I$. 
A Cartier divisor $B$ on $X$ is {\em{permissible}} if $B$ contains no strata of $X$ in its support. 
A finite $\mathbb Q$-linear (resp.~$\mathbb R$-linear) combination of permissible 
Cartier divisors is called a {\em{permissible $\mathbb Q$-divisor}} (resp.~{\em{$\mathbb R$-divisor}}) on 
$X$. 

Let $(X, D)$ be a simple normal crossing pair such that 
$D$ is a boundary $\mathbb R$-divisor 
on $X$. 
Let $\nu:X^\nu \longrightarrow X$ be the normalization. 
We define $\Theta$ by the formula 
$$
K_{X^\nu}+\Theta=\nu^*(K_X+D). 
$$ 
Then a {\em{stratum}} of $(X, D)$ is an irreducible component of $X$ or the $\nu$-image 
of a log canonical center of $(X^\nu, \Theta)$. 
We note that $(X^\nu, \Theta)$ is log canonical (cf.~\ref{110basic}).  
When $D=0$, 
this definition is compatible with the aforementioned case. 
An $\mathbb R$-Cartier $\mathbb R$-divisor $B$ on $X$ is 
{\em{permissible with respect to $(X, D)$}} if $B$ contains no strata of $(X, D)$ in its support. 
If $B$ is a permissible $\mathbb R$-Cartier $\mathbb R$-divisor with respect to $(X, D)$, then 
we can easily check that 
$$
B=\sum _i b_i B_i 
$$ 
where $B_i$ is a permissible Cartier divisor with respect to $(X, D)$ and 
$b_i\in \mathbb R$ for every $i$. 
\end{defn}

The reader will find that it is very useful 
to introduce the notion of {\em{globally embedded simple normal crossing 
pairs}} for the 
proof of vanishing and torsion-free theorems (cf.~\cite[Chapter 2]{book}). 

\begin{defn}[Globally embedded simple normal crossing 
pairs]\label{gsnc0} 
Let $X$ be a simple normal crossing divisor 
on a smooth 
variety $M$ and let $B$ be an $\mathbb R$-divisor 
on $M$ such that 
$\Supp (B+X)$ is a simple normal crossing divisor and that 
$B$ and $X$ have no common irreducible components. 
We put $D=B|_X$ and consider the pair $(X, D)$. 
We call $(X, D)$ a {\em{globally embedded simple normal 
crossing pair}}. 
In this case, it is obvious that 
$(X, D)$ is an embedded simple normal crossing pair. 
\end{defn}

In Section \ref{sec-quasi-proj}, 
we will discuss some vanishing and torsion-free theorems for 
{\em{quasi-projective}} simple normal crossing pairs, which 
will play crucial roles in Section \ref{sec4}. 
See also \cite{fujino-vanishing}, \cite{fujino-slc}, and \cite{fuj-inj}. 

Finally, let us recall the definition of {\em{semi divisorial log terminal 
pairs}} in the sense of Koll\'ar (see \cite[Definition 5.19]{kollar-book} 
and \cite[Definition 4.1]{fujino-vanishing}). 

\begin{defn}[Semi divisorial log terminal pairs]\label{def-sdlt}
Let $X$ be an equidimensional variety which satisfies Serre's $S_2$ 
condition and is normal crossing in codimension one. 
Let $\Delta$ be a boundary $\mathbb R$-divisor on $X$ whose 
support does not contain any irreducible components of the conductor of 
$X$. Assume that $K_X+\Delta$ is $\mathbb R$-Cartier. 
The pair $(X, \Delta)$ is {\em{semi divisorial log terminal}} 
if $a(E, X, \Delta)>-1$ for every 
exceptional divisor $E$ over $X$ such that 
$(X, \Delta)$ is not a simple normal crossing 
pair at the generic point of $c_X(E)$, where 
$c_X(E)$ is the 
center of $E$ on $X$. 
\end{defn}

\begin{rem}
The definition of semi divisorial log terminal pairs in \cite[Definition 1.1]{fujino-abun} 
is different from Definition \ref{def-sdlt}. 
\end{rem}

For the details of semi divisorial log terminal pairs, see 
\cite[Section 5.4]{kollar-book} and \cite[Section 4]{fujino-vanishing}. 

\section{Generalities on variation of mixed Hodge structure}
\label{generalities on VMHS} 

\begin{say}
Let $X$ be a complex analytic variety.
For a point $x \in X$, 
$\bC(x)(\simeq \bC)$ denotes 
the residue field at the point $x$.
For a morphism $\varphi: \xF \longrightarrow \xG$ of $\xO_X$-modules 
the morphism
\begin{equation*}
\varphi \otimes \id:
\xF \otimes \bC(x) \longrightarrow \xG \otimes \bC(x)
\end{equation*} 
is denoted by $\varphi(x)$
for a point $x \in X$.
\end{say}

\begin{rem}
\label{equivalent conditions for strict compatibility}
For a complex $K$ equipped with a finite decreasing filtration $F$
and for an integer $q$,
the four conditions
\begin{newitemize}
\itemno
\label{strict compatibility for (K,F)}
$d:K^q \longrightarrow K^{q+1}$ is strictly compatible
with the filtration $F$,
\itemno
the canonical morphism
$H^{q+1}(F^pK) \longrightarrow H^{q+1}(K)$
is injective for all $p$,
\itemno
the canonical morphism
$H^{q+1}(F^{p+1}K) \longrightarrow H^{q+1}(F^pK)$
is injective for all $p$,
\itemno
the canonical morphism
$H^q(F^pK) \longrightarrow H^q(\Gr_F^pK)$
is surjective for all $p$
\end{newitemize}
are equivalent.
Therefore the strict compatibility in \eqref{strict compatibility for (K,F)}
makes sense in the filtered derived category.
\end{rem}

On a complex variety $X$,
a complex of $\xO_X$-modules $K$ is called
a perfect complex
if, locally on $X$, it is isomorphic in the derived category
to a bounded complex
consisting of free $\xO_X$-modules of finite rank
(see e.g.~\cite[8.3.6.3]{FGAE}). 
The following definition
is an analogue of the notion of perfect complex.

\begin{defn}
\label{def of filtered perfect complex}
Let $X$ be a complex variety.
A complex of $\xO_X$-modules $K$
equipped with a finite decreasing filtration $F$ 
is called a filtered perfect 
if $\Gr_F^pK$ 
is a perfect complex for all $p$. 
\end{defn}

\begin{lem}
\label{toy base change}
Let $X$ be a complex manifold.

\begin{subthm}
\label{upper semi-continuity}
For a perfect complex $K$ on $X$,
the function
\begin{equation*}
X \ni x \mapsto \dim H^q(K \otimes \bC(x))
\end{equation*}
is upper semi-continuous for all $q$.
\end{subthm}

\begin{subthm}
\label{local freeness implies base change}
Let $K$ be a perfect complex on $X$.
If there exists an integer $q_0$
such that $H^q(K)$ is locally free of finite rank for all $q \ge q_0$,
then the canonical morphism
\begin{equation*}
H^q(K) \otimes \xF
\longrightarrow
H^q(K \otimes^L \xF)
\end{equation*}
is an isomorphism for any $\xO_X$-module $\xF$ and for all $q \ge q_0$.
\end{subthm}

\begin{subthm}
\label{equivalence of constancy of dim and local freeness}
Fix an integer $q$.
For a perfect complex $K$ on $X$,
the following two conditions are equivalent:
\begin{newitemize}
\itemno
The function
\begin{equation*}
X \ni x \mapsto \dim H^q(K \otimes^L \bC(x))
\end{equation*}
is locally constant.
\itemno
The sheaf $H^q(K)$ is locally free of finite rank
and the canonical morphism
\begin{equation*}
H^q(K) \otimes \xF
\longrightarrow
H^q(K \otimes^L \xF)
\end{equation*}
is an isomorphism for any $\xO_X$-module $\xF$.
\end{newitemize}
Moreover, if these equivalent conditions are satisfied,
then the canonical morphism
\begin{equation*}
H^{q-1}(K) \otimes \xF
\longrightarrow
H^{q-1}(K \otimes^L \xF)
\end{equation*}
is an isomorphism for any $\xO_X$-module $\xF$.
\end{subthm}

\begin{subthm}
\label{strictness from residue field to stalk}
Let $(K,F)$ be a filtered perfect complex on $X$.
Assume that
the function
\begin{equation*}
X \ni x \mapsto \dim H^q(K \otimes^L \bC(x))
\end{equation*}
is locally constant.
If the morphisms
\begin{align*}
&d(x):(K \otimes^L \bC(x))^{q-1} \longrightarrow (K \otimes^L \bC(x))^q \\
&d(x):(K \otimes^L \bC(x))^q \longrightarrow (K \otimes^L \bC(x))^{q+1}
\end{align*}
are strictly compatible with the filtration $F(K \otimes^L \bC(x))$
for every $x \in X$,
then $H^q(\Gr_F^pK)$
is locally free of finite rank,
the canonical morphism
\begin{equation}
\label{the canonical morphism in question}
H^q(\Gr_F^pK) \otimes \bC(x)
\simeq
H^q(\Gr_F^p(K \otimes^L \bC(x)))
\end{equation}
is an isomorphism for all $p$
and for every $x \in X$,
and $d:K^q \longrightarrow K^{q+1}$
is strictly compatible with the filtration $F$.
\end{subthm}
\end{lem}
\begin{proof}
We can easily obtain
\ref{upper semi-continuity},
\ref{local freeness implies base change} and
\ref{equivalence of constancy of dim and local freeness}
by the arguments in \cite[Chapter 5]{mumfordAV}.

The strict compatibility conditions
in \ref{strictness from residue field to stalk}
imply the exactness of the sequence
\begin{equation*}
\begin{CD}
0 @>>> H^q(F^{p+1}(K \otimes^L \bC(x)))
@>>> H^q(F^p(K \otimes^L \bC(x))) \\
@. @>>> H^q(\Gr_F^pK \otimes^L \bC(x)) @>>> 0
\end{CD}
\end{equation*}
for all $p$ and for every $x \in X$.
Thus we obtain the equality
\begin{equation*}
\sum_{p}\dim H^q(\Gr_F^pK \otimes^L \bC(x))
= \dim H^q(K \otimes^L \bC(x))
\end{equation*}
for every $x$,
which implies that $\dim H^q(\Gr_F^pK \otimes^L \bC(x))$
is locally constant with respect to $x \in X$.
Applying \ref{equivalence of constancy of dim and local freeness},
$H^q(\Gr_F^pK)$ is locally free
and \eqref{the canonical morphism in question}
is an isomorphism for all $p$ and for any $x \in X$.
By using the isomorphisms
\eqref{the canonical morphism in question} for all $p$,
we can easily see
the surjectivity of the canonical morphism
\begin{equation*}
H^q(F^pK) \otimes \bC(x)
\longrightarrow
H^q(\Gr_F^pK) \otimes \bC(x)
\end{equation*}
for any $x \in X$,
and then the canonical morphism
\begin{equation*}
H^q(F^pK) \longrightarrow H^q(\Gr_F^pK)
\end{equation*}
is surjective for every $p$.
Thus the morphism $d:K^q \longrightarrow K^{q+1}$
is strictly compatible with the filtration $F$
by Remark \ref{equivalent conditions for strict compatibility}.
\end{proof}

\begin{defn}
Let $X$ be a complex manifold.
A pre-variation of $\bQ$-Hodge structure of weight $m$ on $X$
is a triple $V=(\bV,(\xV,F),\alpha)$ such that
\begin{itemize}
\item
$\bV$ is a local system of finite dimensional $\bQ$-vector space on $X$,
\item
$\xV$ is an $\xO_X$-module
and $F$ is a finite decreasing filtration on $\xV$,
\item
$\alpha: \bV \longrightarrow \xV$
is a morphism of $\bQ$-sheaves,
\end{itemize}
satisfying the conditions
\begin{newitemize}
\itemno
\label{isom induced by alpha}
$\alpha$ induces an isomorphism
$\xO_X \otimes \bV \simeq \xV$ of $\xO_X$-modules,
\itemno
$\Gr_F^p\xV$ is a locally free $\xO_X$-module of finite rank for every $p$,
\itemno
$(\bV_x, F(\xV(x)))$ is a Hodge structure of weight $m$
for every $x \in X$,
where we identify $\bV_x \otimes \bC$ with $\xV(x)$ by the isomorphism $\alpha(x)$.
\end{newitemize}
We denote $(\bV_x,F(\xV(x)))$ by $V(x)$ for $x \in X$.

We identify $\xO_X \otimes \bV$ with $\xV$ by the isomorphism
in \eqref{isom induced by alpha}
if there is no danger of confusion.
Under this identification,
we write $V=(\bV,F)$ for a pre-variation of $\bQ$-Hodge structure.

We define the notion of a morphism of pre-variations
in the trivial way.
\end{defn}

\begin{rem}
\label{Griffiths transversality}
A variation of $\bQ$-Hodge structure of weight $m$ on $X$
is nothing but a pre-variation $V=(\bV,F)$ of $\bQ$-Hodge structure of 
weight $m$ 
such that the canonical integrable connection
$\nabla$ on $\xV=\xO_X \otimes \bV$
satisfies the Griffiths transversality
\begin{equation}
\label{Griffiths transversality:eq}
\nabla(F^p) \subset \Omega_X^1 \otimes F^{p-1}
\end{equation}
for every $p$.
A morphism of variations of $\bQ$-Hodge structure 
is a morphism of underlying pre-variations of $\bQ$-Hodge structure.
\end{rem}

\begin{rem}
\begin{subdefn}
Let $V_1=(\bV_1,F)$ and $V_2=(\bV_2,F)$ be
pre-variations of $\bQ$-Hodge structure of weight $m_1$ and $m_2$ respectively.
Then the local systems $\bV_1 \otimes \bV_2$
and $\shom(\bV_1, \bV_2)$ underlie pre-variations of $\bQ$-Hodge structure 
of weight $m_1+m_2$ and $m_2-m_1$ respectively. 
These pre-variations of $\bQ$-Hodge structure 
are denoted by $V_1 \otimes V_2$ and $\shom(V_1,V_2)$ respectively.
\end{subdefn}

\begin{subdefn}
For an integer $n$,
$\bQ_X(n)$ denotes the pre-variation of $\bQ$-Hodge structure of Tate as usual.
This is, in fact, a variation of $\bQ$-Hodge structure of weight $-2n$ on $X$.
For a pre-variation $V$ of $\bQ$-Hodge structure of weight $m$,
$V(n)=V \otimes \bQ_X(n)$ is a pre-variation of $\bQ$-Hodge structure of weight $m-2n$,
which is called the Tate twist of $V$ as usual.
\end{subdefn}
\end{rem}

\begin{defn}
Let $X$ be a complex manifold
and $V=(\bV, F)$ a pre-variation of $\bQ$-Hodge structure of weight $m$ on $X$.
A polarization on $V$ is a morphism of pre-variations of $\bQ$-Hodge structure 
\begin{equation*}
V \otimes V \longrightarrow \bQ_X(-m)
\end{equation*}
which induces a polarization on $V(x)$ for every point $x \in X$.
A pre-variation of $\bQ$-Hodge structure of weight $m$
is said to be polarizable,
if there exists a polarization on it.
A morphism of polarizable pre-variations of $\bQ$-Hodge structure 
is a morphism of the underlying pre-variations of $\bQ$-Hodge structure.
\end{defn}

\begin{defn}
Let $X$ be a complex manifold.

\begin{subdefn}
A pre-variation of $\bQ$-mixed Hodge structure on $X$
is a triple $V=((\bV, W),(\xV,W,F),\alpha)$ consisting of
\begin{itemize}
\item
a local system of finite dimensional $\bQ$-vector space $\bV$,
equipped with a finite increasing filtration $W$ by local subsystems,
\item
an $\xO_X$-module $\xV$
equipped with a finite increasing filtration $W$
and a finite decreasing filtration $F$,
\item
a morphism of $\bQ$-sheaves
$\alpha: \bV \longrightarrow \xV$
preserving the filtration $W$
\end{itemize}
such that the triple
$\Gr_m^WV=(\Gr_m^W\bV,(\Gr_m^W\xV,F),\Gr_m^W\alpha)$
is a pre-variation of $\bQ$-Hodge structure of weight $m$
for every $m$.

We identify $(\xO_X \otimes \bV, W)$ and $(\xV,W)$ by the isomorphism
induced by $\alpha$ as before,
if there is no danger of confusion.
Under this identification,
we use the notation $V=(\bV,W,F)$
for a pre-variation of $\bQ$-mixed Hodge structure.
\end{subdefn}

\begin{subdefn}
A pre-variation $V=(\bV,W,F)$ of $\bQ$-mixed Hodge structure on $X$
is called graded polarizable,
if $\Gr_m^WV$ is a polarizable pre-variation of $\bQ$-Hodge structure
for every $m$.
\end{subdefn}

\begin{subdefn}
We define the notion of a morphism of pre-variations of $\bQ$-mixed Hodge 
structure by the trivial way.
A morphism of polarizable pre-variations of $\bQ$-mixed Hodge structure 
is a morphism of the underlying pre-variations.
\end{subdefn}
\end{defn}

Now, let us recall the definition of {\em{graded polarizable variation of 
$\mathbb Q$-mixed Hodge structure}} (GPVMHS, for short).
See, for example, \cite[\S 3]{sz}, 
\cite[Part I, Section 1]{ssu2}, \cite[Section 7]{bz}, 
\cite[Definitions 14.44 and 14.45]{ps}, and so on.

\begin{defn}[GPVMHS]\label{vmhs}
Let $X$ be a complex manifold.

\begin{subdefn}
A pre-variation $V=(\bV,W,F)$ of $\bQ$-mixed Hodge structure on $X$
is said to be a variation of $\bQ$-mixed Hodge structure,
if the canonical integrable connection
$\nabla$ on $\xV \simeq \xO_X \otimes \bV$
satisfies the Griffiths transversality
\eqref{Griffiths transversality:eq}
in Remark \ref{Griffiths transversality}.
\end{subdefn}

\begin{subdefn}
A variation of $\bQ$-mixed Hodge structure
is called graded polarizable,
if the underlying pre-variation is graded polarizable.
\end{subdefn}

\begin{subdefn}
A morphism of (graded polarizable) variations of $\bQ$-mixed Hodge structure 
is a morphism of the underlying pre-variations.
\end{subdefn}
\end{defn}

The following definition of the {\em{admissibility}}
is given by Steenbrink--Zucker \cite[(3.13) Properties]{sz}
in the one-dimensional case
and by Kashiwara
\cite[1.8, 1.9]{kashiwara} in the general case. 
See also \cite[Definition 14.49]{ps}. 

\begin{defn}[{Admissibility (cf.~\cite[1.8, 1.9]{kashiwara})}]\label{admissibility}
\begin{subdefn}
\label{pre-admissibility on Delta*}
A variation of $\mathbb Q$-mixed Hodge structure
$V=(\mathbb{V}, W, F)$ over $\pd=\Delta \setminus \{0\}$, where 
$\Delta=\{ t\in \mathbb C\, |\, |t|<1\}$,  
is said to be {\em{pre-admissible}} if it satisfies: 
\begin{newitemize}
\itemno
\label{quasi-unipotency}
The monodromy around the origin is quasi-unipotent.
\itemno
\label{extendability of Hodge filtration}
Let $\widetilde{\xV}$ and $W_k\widetilde{\xV}$
be the upper canonical extensions of $\xV=\xO_{\pd} \otimes \mathbb{V}$
and of $\xO_{\pd} \otimes W_k\mathbb{V}$
in the sense of Deligne \cite[Remarques 5.5 (i)]{Deligne} 
(see also \cite[Section 2]{ko2} and Remark \ref{upper-lower}).
Then the filtration $F$ on $\xV$
extends to the filtration $F$
on $\widetilde{\xV}$ such that
$\Gr_{F}^p\Gr_k^{W}\widetilde{\xV}$
is locally free $\xO_{\Delta}$-modules of finite rank for each $k,p$.
\itemno
\label{existence of the relative monodromy weight filtration}
The logarithm of the unipotent part of the monodromy
admits a weight filtration relative to $W$. 
\end{newitemize}
\end{subdefn}

\begin{subdefn}
Let $X$ be a complex variety
and $U$ a nonsingular Zariski open subset of $X$.
A variation of $\bQ$-mixed Hodge structure $V$
on $U$ is said to be {\em{admissible}} (with respect to $X$)
if for every morphism $i: \Delta \longrightarrow X$
with $i(\pd) \subset U$,
the variation $i\!^*V$ on $\pd$ is pre-admissible.
\end{subdefn}
\end{defn}

We can define an $\mathbb R$-mixed Hodge structure, a variation of $\mathbb R$-mixed 
Hodge structure, and so on, analogously. 

We frequently use the following lemma in Section \ref{sub42},
which is a special case of \cite[Proposition 1.11.3]{kashiwara} 
(see also Remark \ref{upper-lower} below).

\begin{prop}[{cf.~\cite{kashiwara}}]
\label{extension of Hodge filtration}
Let $X$ be a complex manifold,
$U$ the complement of a normal crossing divisor on $X$
and $V=(\mathbb{V},W,F)$
a variation of $\bR$-mixed Hodge structure on $U$. 
The upper canonical extensions of
$\xV=\xO_U \otimes \mathbb{V}$ and of $W_k\xV=\xO_U \otimes W_k\mathbb{V}$ 
are denoted 
by $\widetilde{\xV}$ and by $W_k\widetilde{\xV}$ respectively.
If $V$ is admissible on $U$
with respect to $X$,
then the filtration $F$ on $\xV$
extends to a finite filtration $F$ on $\widetilde{\xV}$ by subbundles
such that $\Gr_F^p\Gr_k^W\widetilde{\xV}$ is a locally free
$\xO_X$-module of finite rank for all $k,p$.
\end{prop}

We give an elementary but useful remark on the quasi-unipotency of 
monodromy. 

\begin{rem}[Quasi-unipotency]\label{q-uni}
If the local system $\mathbb V$ has a $\mathbb Z$-structure, that is, 
there is a local system $\mathbb V\!_{\mathbb Z}$ on $X$ of $\mathbb Z$-modules 
of finite rank such that 
$\mathbb V=\mathbb V\!_{\mathbb Z}\otimes \mathbb Q$, 
in Definition \ref{admissibility}, 
then the quasi-unipotency automatically follows from 
Borel's theorem (cf.~\cite[(4.5) Lemma (Borel)]{schmid}). 
\end{rem}

The following lemma states
the fundamental results on pre-variations of $\bQ$-Hodge structure.

\begin{lem}
\label{lemma for pre-variations}
Let $X$ be a complex manifold.

\begin{subthm}
\label{abelian category of pre-variations}
The category of the pre-variations of $\bQ$-Hodge structure 
of weight $m$ on $X$
is an abelian category for every $m$.
\end{subthm}

\begin{subthm}
\label{a property of morphism of pre-variations}
Let $V_1$ and $V_2$ be pre-variations of $\bQ$-Hodge structure 
of weight $m_1$ and $m_2$ respectively,
and $\varphi: V_1 \longrightarrow V_2$ a morphism of pre-variations.
If $m_1 > m_2$, then $\varphi=0$.
\end{subthm}

\begin{subthm}
\label{strictness of morphism of pre-variations}
Let $\varphi:V_1 \longrightarrow V_2$ be a morphism
of pre-variations $V_1=(\bV_1,F)$ and $V_2=(\bV_2,F)$
of $\bQ$-Hodge structure of weight $m$ on $X$.
Then the induced morphism
$\varphi \otimes \id: \bV_1 \otimes \xO_X \longrightarrow \bV_2 \otimes \xO_X$
is strictly compatible with the filtration $F$.
\end{subthm}

\begin{subthm}
\label{exactness of the functor (x)}
The functor from the category 
of the pre-variations of $\bQ$-Hodge structure of weight $m$
to the category of the $\bQ$-Hodge structures of weight $m$
which assigns $V$ to $V(x)$ is an exact functor for every $x \in X$.
\end{subthm}

\begin{subthm}
\label{abelian category of polarizable variations}
The category of the polarizable variations
of $\bQ$-Hodge structure of weight $m$ on $X$
is an abelian category for every $m$.
\end{subthm}
\end{lem}
\begin{proof}
The statements
\ref{abelian category of pre-variations},
\ref{strictness of morphism of pre-variations}
and \ref{exactness of the functor (x)}
are easy consequences of
Lemma \ref{toy base change}
\ref{strictness from residue field to stalk},
and \ref{a property of morphism of pre-variations}
is easily proved by the corresponding result
for $\bQ$-Hodge structures.
So we prove
\ref{abelian category of polarizable variations} now.

Let $V_1=(\bV_1,F)$ and $V_2=(\bV_2,F)$ be
polarizable pre-variations of $\bQ$-Hodge structure 
of weight $m$ on $X$,
and $\varphi: V_1 \longrightarrow V_2$ a morphism.
We fix polarizations on $V_1$ and $V_2$ respectively.
Taking (i) into the account,
it is sufficient to prove that $\Ker(\varphi)$
and $\Coker(\varphi)$ are polarizable.
The case of $\Ker(\varphi)$ is trivial.
Then we discuss the case of $\Coker(\varphi)$.

The morphism $\varphi$ induces a morphism
\begin{equation*}
\varphi^*: \shom(V_2,\bQ_X(-m)) \longrightarrow \shom(V_1,\bQ_X(-m))
\end{equation*}
which is clearly a morphism of pre-variations of $\bQ$-Hodge structure 
of weight $m$.
On the other hand,
the polarizations on $V_1$ and $V_2$
induce the identifications
\begin{equation*}
V_1 \simeq \shom(V_1,\bQ_X(-m)), \quad
V_2 \simeq \shom(V_2,\bQ_X(-m))
\end{equation*}
which are isomorphisms of pre-variations of $\bQ$-Hodge structure.
By these identifications the morphism $\varphi^*$ above
can be considered as a morphism of pre-variations $V_2 \longrightarrow V_1$,
which is denoted 
by the same symbol $\varphi^*$ by abuse of the language.
Then the inclusion $\Ker(\varphi^*) \hookrightarrow V_2$
induces an isomorphism $\Ker(\varphi^*) \simeq \Coker(\varphi)$
of pre-variations.
Therefore we obtain a polarization expected.
\end{proof}

Here we give a brief remark
on the dual of a variation of $\bQ$-mixed Hodge structure.

\begin{rem}
\label{dual of VMHS}
Let $V=((\bV,W),(\xV,W,F),\alpha)$ be
a pre-variation of $\bQ$-mixed Hodge structure
on a complex manifold $X$.
On the dual local system $\bV^{\ast}=\shom_{\bQ}(\bV, \bQ)$,
\begin{equation*}
W_m\bV^{\ast}=(\bV/W_{-m-1})^{\ast} \subset \bV^{\ast}
\end{equation*}
defines an increasing filtration $W$.
Similarly, on $\xV^{\ast}=\shom_{\xO_X}(\xV,\xO_X)$
\begin{equation*}
W_m\xV^{\ast}=(\xV/W_{-m-1})^{\ast}, \quad
F^p\xV^{\ast}=(\xV/F^{1-p})^{\ast}
\end{equation*}
define increasing and decreasing filtrations on $\xV^{\ast}$.
We have
\begin{equation*}
\Gr_F^p\Gr_m^W\xV^{\ast} \simeq (\Gr_F^{-p}\Gr_{-m}^W\xV)^{\ast}
\end{equation*}
for every $m,p$ by definition.
Equipped with an isomorphism $\xO_X \otimes \bV^{\ast} \simeq \xV^{\ast}$,
it turns out that
$((\bV^{\ast},W),(\xV^{\ast},W,F))$
is a pre-variation of $\bQ$-mixed Hodge structure on $X$. 
It is denoted by $V^{\ast}$ for short
and is called the dual of $V$.
It is easy to see that
$V^{\ast}$ is graded polarizable 
or a variation of $\bQ$-mixed Hodge structure 
if $V$ is so.
For the case where $X$ is a Zariski open subset of another variety,
$V^{\ast}$ is admissible if $V$ is so.
\end{rem}

We close this section by the lemma 
concerning with the relative monodromy weight filtration
for a filtered $\bQ$-mixed Hodge complex. 
For the definition, 
see, for example,  \cite[6.1.4 D\'efinition]{elzein2}.

\begin{rem}
Now we remark the notation on the shift of increasing filtrations.
We set
\begin{equation*}
W[m]_k=W_{k-m}
\end{equation*}
as in \cite{DeligneII}, \cite{elzein2}.
Our notation is different from the one
in \cite{cks}.
\end{rem}

\begin{lem}
\label{lemma on relative monodromy filtration}
Let
$((A_{\bQ},W^f,W), (A_{\bC},W^f,W,F), \alpha)$
be a filtered $\bQ$-mixed Hodge complex
such that the spectral sequence $E_r^{p,q}(A_{\bC},W^f)$
degenerates at $E_2$-terms,
and $\nu: A_{\bC} \longrightarrow A_{\bC}$
a morphism of complexes
preserving the filtration $W^f$
and satisfying the condition
$\nu(W_mA_{\bC}) \subset W_{m-2}A_{\bC}$
for every $m$.
If the filtration $W[-m]$ on $H^n(\Gr_m^{W^f}A_{\bC})$
is the monodromy weight filtration of the endomorphism
$H^n(\Gr_m^{W^f}\nu)$
for all $m,n$,
then the filtration $W$ on $H^n(A_{\bC})$
is the relative weight monodromy filtration
of the endomorphism $H^n(\nu)$ with respect to the filtration $W^f$
for all $n$.
\end{lem}
\begin{proof}
The assumption implies that the morphism
$H^{p+q}(\Gr_{-p}^{W^f}\nu)^k$ induces an isomorphism
\begin{equation*}
\Gr_{q+k}^{W[p+q]}E_1^{p,q}(A_{\bC},W^f)
\longrightarrow
\Gr_{q-k}^{W[p+q]}E_1^{p,q}(A_{\bC},W^f)
\end{equation*}
for all $p,q$ and for $k \ge 0$
because of the isomorphism
$E_1^{p,q}(A_{\bC},W^f) \simeq H^{p+q}(\Gr_{-p}^{W^f}A_{\bC})$.
On the other hand,
the $E_2$-degeneracy for the filtration $W^f$
gives us the isomorphism
\begin{equation*}
E_2^{p,q}(A_{\bC},W^f)
\simeq
\Gr_{-p}^{W^f}H^{p+q}(A_{\bC})
\end{equation*}
for all $p,q$,
under which
the filtration $W_{\rec}$ on the left hand side
coincides with the filtration $W$ on the right hand side
by \cite[6.1.8 Th\'eor\`eme]{elzein2}.
Since the morphism $d_1$ of $E_1$-terms
induces a morphism of mixed Hodge structures
\begin{equation*}
d_1:
(E_1^{p,q}(A_{\bC},W^f),W[p+q],F)
\longrightarrow
(E_1^{p+1,q}(A_{\bC},W^f),W[p+q+1],F)
\end{equation*}
for all $p,q$
by \cite[6.1.8 Th\'eor\`eme]{elzein2} again,
the morphism
$(H^{p+q}(\nu))^k$ induces an isomorphism
\begin{equation*}
\Gr_{q+k}^{W[p+q]}\Gr_{-p}^{W^f}H^{p+q}(A_{\bC})
\longrightarrow
\Gr_{q-k}^{W[p+q]}\Gr_{-p}^{W^f}H^{p+q}(A_{\bC})
\end{equation*}
for $p,q$ and for $k \ge 0$.
Then we can easily check the conclusion.
\end{proof}

\section{Variation of mixed Hodge structure of geometric origin}
\label{VMHS of geometric origin} 

In this section, we discuss variations of 
mixed Hodge structure arising from mixed Hodge structures 
on cohomology with compact support for \snc pairs.
We will check that those variations of mixed Hodge structure are  
{\em{graded polarizable}} and {\em{admissible}}
(see Theorem \ref{GPVMHS for snc pair}).
These properties 
will play crucial roles in the subsequent sections.  

\begin{say}
\label{Godment resolution} 
For a morphism of topological spaces
$f:X \longrightarrow Y$,
we always use the Godment resolution
to compute the higher direct image $Rf_{\ast}$
of abelian sheaves on $X$.
This means that $Rf_{\ast}\xF$ is the genuine {\itshape complex}
$f_{\ast}\xC\gdm^{\bullet}\xF$ for an abelian sheaf $\xF$ on $X$,
where $\xC\gdm^{\bullet}$ stands for the Godment resolution
as in Peters--Steenbrink \cite[B.2.1]{ps}.
If $\xF$ carries a filtration $F$,
$Rf_{\ast}\xF$ is the genuine {\itshape filtered complex}.
For a morphism of sheaves $\varphi: \xF \longrightarrow \xG$ on $X$,
the morphism
\begin{equation*}
Rf_{\ast}(\varphi):Rf_{\ast}\xF \longrightarrow Rf_{\ast}\xG
\end{equation*}
is the genuine {\itshape morphism of complexes}
defined by using the Godment resolution.
We use the same notation
for complexes of abelian sheaves, $\bQ$-sheaves, $\bC$-sheaves,
$\xO_X$-modules and so on.
\end{say}

\begin{say}
Let $f:X\sso \longrightarrow Y$
be an augmented semi-simplicial topological space.
The morphism $X_p \longrightarrow Y$ induced by $f$
is denoted by $f_p$ for every $p$.
For an abelian sheaf $\xF\css$ on $X\sso$.
\begin{equation*}
Rf_{\bullet\ast}\xF\css
=\{Rf_{p\ast}\xF^p\}_{p \ge 0}
\end{equation*}
defines a co-semi-simplicial {\itshape complexes} on $Y$
by what we mentioned in \ref{Godment resolution}.
Then we define
\begin{equation*}
Rf_{\ast}\xF\css=
sRf_{\bullet\ast}\xF\css
\end{equation*}
as in \cite[(5.2.6.1)]{DeligneIII}.
More precisely,
$Rf_{\ast}\xF\css$ is the single complex
associated to the double complex
\begin{equation*}
\begin{CD}
@. \vdots @. \vdots \\
@. @VVV @VVV \\
\cdots @>>> (Rf_{p\ast}\xF^p)^q @>{\delta}>> (Rf_{p+1\ast}\xF^{p+1})^q @>>> \cdots \\
@. @V{(-1)^pd}VV @VV{(-1)^{p+1}d}V \\
\cdots @>>> (Rf_{p\ast}\xF^p)^{q+1} @>{\delta}>> (Rf_{p+1\ast}\xF^{p+1})^{q+1} @>>> \cdots \\
@. @VVV @VVV \\
@. \vdots @. \vdots
\end{CD}
\end{equation*}
where $\delta$ in the horizontal lines denotes the \v{C}ech type morphism
and $d$ in the vertical lines denotes the differential
of the {\itshape complexes} $Rf_{p\ast}\xF^p$ for every $p$.
The increasing filtration $L$ on $Rf_{\ast}\xF\css$ is defined by
\begin{equation}
\label{filtration L}
L_m(Rf_{\ast}\xF\css)^n
=\bigoplus_{p \ge -m}(Rf_{p\ast}\xF^p)^{n-p}
\end{equation}
for every $m,n$
(cf.~Deligne \cite[(5.1.9.3)]{DeligneIII}).
Thus we have
\begin{equation}
\label{gr of L}
\Gr_m^LRf_{\ast}\xF\css
=Rf_{-m\ast}\xF^{-m}[m]
\end{equation}
for every $m$.
For the case where $\xF\css$ admits an increasing filtration $W$,
we set
\begin{align}
&W_m(Rf_{\ast}\xF\css)^n
=\bigoplus_pW_m(Rf_{p\ast}\xF^p)^{n-p}
\label{filtration W} \\
&\delta(W,L)_m(Rf_{\ast}\xF\css)^n
=\bigoplus_pW_{m+p}(Rf_{p\ast}\xF^p)^{n-p}
\label{filtration delta}
\end{align}
for every $n,p$.
Then we have
\begin{equation*}
\Gr_m^{\delta(W,L)}Rf_{\ast}\xF
=\bigoplus_p\Gr_{m+p}^WRf_{p\ast}\xF^p[-p]
\end{equation*}
for every $m$.
The case of decreasing filtration $F$ on $\xF$
is transformed to the case of increasing filtration
by setting $W_m\xF=F^{-m}\xF$.
We use the same convention for complexes of 
abelian sheaves, $\bQ$-sheaves and so on.
\end{say}

\begin{say}
\label{the case of smooth projective morphism}
Let $X$ and $Y$ be complex manifolds and
$f: X \longrightarrow Y$ a smooth projective morphism. 
The de Rham complexes of $X$ and $Y$
are denoted by $\Omega_X$ and $\Omega_Y$ respectively,
and the relative de Rham complex for the morphism $f$
is denoted by $\Omega_{X/Y}$. 
Moreover, $F$ denotes the stupid filtration on $\Omega_X$ 
and $\Omega_{X/Y}$.
The inclusion $\bQ_X \longrightarrow \xO_X$
induces a morphism of complexes $\bQ_X \longrightarrow \Omega_{X/Y}$.
Then we obtain a morphism 
\begin{equation*}
R^if_*\bQ_X \longrightarrow R^if_*\Omega_{X/Y}
\end{equation*}
for every $i$, 
which is denoted by $\alpha_{X/Y}$ simply.
Then
\begin{equation*} 
(R^if_*\bQ_X, (R^if_*\Omega_{X/Y},F),\alpha_{X/Y})
\end{equation*}
is a polarizable variation of $\bQ$-Hodge 
structure of weight $i$ on $Y$.
Here we recall the proof of the Griffiths transversality
following \cite{katzoda}.

A finite decreasing filtration $G$ on $\Omega_X$ is defined by
\begin{equation}
\label{filtration G on Omega}
G^p\Omega_X
=\im(f^{-1}\Omega^p_Y \otimes_{f^{-1}\xO_Y} \Omega_X[-p]
\longrightarrow \Omega_X)
\end{equation}
for all $p$.
Then we have isomorphisms
\begin{equation*}
\Gr_G^p\Omega_X \simeq f^{-1}\Omega^p_Y \otimes_{f^{-1}\xO_Y} \Omega_{X/Y}[-p]
\end{equation*}
for all $p$,
which induces isomorphisms
\begin{equation*}
E_1^{p,q}(Rf_{\ast}\Omega_X,G) \simeq \Omega_Y^p \otimes R^qf_{\ast}\Omega_{X/Y}
\end{equation*}
for every $p,q$.
Thus the morphisms of the $E_1$-terms give us
\begin{equation*}
\nabla:
\Omega_Y^p \otimes R^qf_{\ast}\Omega_{X/Y}
\longrightarrow
\Omega_Y^{p+1} \otimes R^qf_{\ast}\Omega_{X/Y}
\end{equation*}
for all $p,q$.
It is easy to see that
\begin{equation*}
\nabla:
R^qf_{\ast}\Omega_{X/Y}
\longrightarrow
\Omega_Y^1 \otimes R^qf_{\ast}\Omega_{X/Y}
\end{equation*}
satisfies the Griffiths transversality.

On the other hand,
we consider the complexes $\Omega_Y$ and $f^{-1}\Omega_Y$.
The stupid filtration on $\Omega_Y$ is denoted by $G$ for a while.
We have
\begin{equation*}
\begin{split}
\Gr_G^p f^{-1}\Omega_Y
&\simeq
f^{-1}\Omega_Y^p[-p] \\
&=f^{-1}\Omega_Y^p \otimes_{\bQ} \bQ_X[-p] \\
&=f^{-1}\Omega_Y^p \otimes_{f^{-1}\xO_Y}f^{-1}\xO_Y[-p]
\end{split}
\end{equation*}
for every $p$.
Therefore the $E_1$-terms of the associated spectral sequence is identified with
\begin{equation*}
E_1^{p,q}(Rf_{\ast}f^{-1}\Omega_Y,G)=
\Omega_Y^p \otimes_{\bQ} R^qf_{\ast}\bQ_X
\end{equation*}
under which the morphisms of $E_1$-terms
are identified with
\begin{equation*}
d \otimes \id:
\Omega_Y^p \otimes_{\bQ} R^qf_{\ast}\bQ_X
\longrightarrow
\Omega_Y^{p+1} \otimes_{\bQ} R^qf_{\ast}\bQ_X
\end{equation*}
for every $p,q$.
On the other hand,
the canonical morphism
\begin{equation*}
f^{-1}\Omega_Y \longrightarrow \Omega_X
\end{equation*}
is a filtered quasi-isomorphism
by the relative Poincar\'e lemma.
Thus we obtain a commutative diagram
\begin{equation*}
\begin{CD}
\xO_Y \otimes_{\bQ} R^qf_{\ast}\bQ_X
@>{\simeq}>>
R^qf_{\ast}\Omega_{X/Y} \\
@V{d \otimes \id}VV @VV{\nabla}V \\
\Omega_Y^1 \otimes_{\bQ} R^qf_{\ast}\bQ_X
@>>{\simeq}> \Omega_Y^1 \otimes R^qf_{\ast}\Omega_{X/Y}
\end{CD}
\end{equation*}
which shows that $d \otimes \id$
on $\xO_Y \otimes_{\bQ} R^if_{\ast}\bQ_X$
satisfies the Griffiths transversality. 
\end{say}

\begin{notationnum}
\label{notation for a semi-simplicial variety} 
A semi-simplicial variety $X\sso$ is said to be strict
if there exists a non-negative integer $p_0$
such that $X_p=\emptyset$ for all $p \ge p_0$.

For an augmented semi-simplicial variety
$f:X\sso \longrightarrow Y$,
we say $f$ is smooth, projective etc.,
if $f_p: X_p \longrightarrow Y$
is smooth, projective etc.~for all $p$.
\end{notationnum}

\begin{lem} 
\label{lemma for cohomology of semi-simplicial varieties}
Let $f:X\sso \longrightarrow Y$
be a smooth projective augmented strict semi-simplicial variety.
Moreover, we assume that $Y$ is smooth.
Then $R^if_{\ast}\bQ_{X\sso}$ underlies
a graded polarizable variation of $\bQ$-mixed Hodge structure on $Y$
for all $i$.
\end{lem}
\begin{proof}
The morphism
\begin{equation*}
R^if_{\ast}(\alpha_{X\sso/Y}):
R^if_{\ast}\bQ_{X\sso}
\longrightarrow
R^if_{\ast}\Omega_{X\sso/Y}
\end{equation*}
is induced by the canonical morphism
$\alpha_{X\sso/Y}: \bQ_{X\sso} \longrightarrow \Omega_{X\sso/Y}$.
We can easily see that
$R^if_{\ast}(\alpha_{X\sso/Y})$
induces the isomorphism
\begin{equation}
\label{isomorphism from Q to O}
R^if_{\ast}\bQ_{X\sso/Y} \otimes \xO_Y
\longrightarrow
R^if_{\ast}\Omega_{X\sso/Y}
\end{equation}
by the relative Poincar\'e lemma.

The filtration $L$ on $R^if_{\ast}\bQ_{X\sso}$
and $R^if_{\ast}\Omega_{X\sso/Y}$ is defined by
\eqref{filtration L}.
Moreover, the stupid filtration $F$ on $\Omega_{X\sso/Y}$
induces the filtration $F$ on $R^if_{\ast}\Omega_{X\sso/Y}$
by the same way as \eqref{filtration W}.
It is sufficient to prove that
\begin{equation*}
((R^if_{\ast}\bQ_{X\sso},L[i]),
(R^if_{\ast}\Omega_{X\sso/Y},L[i],F),
R^if_{\ast}(\alpha_{X\sso/Y}))
\end{equation*}
is a \gpv of $\bQ$-mixed Hodge structure on $Y$.

To this end, we consider the data
\begin{equation*}
K=
((Rf_{\ast}\bQ_{X\sso},L),
(Rf_{\ast}\Omega_{X\sso/Y},L,F),
Rf_{\ast}(\alpha_{X\sso/Y}))
\end{equation*}
and the spectral sequence associated to the filtration $L$.
By \eqref{gr of L}, we have
\begin{equation*}
E_1^{p,q}(K,L)=
(R^qf_{p\ast}\bQ_{X_p},(R^qf_{p\ast}\Omega_{X_p/Y},F), R^qf_{p\ast}(\alpha_{X_p/Y})),
\end{equation*}
which is a polarizable variation of $\bQ$-Hodge structure of weight $q$
for every $p,q$.
Moreover, the morphism
\begin{equation*}
d_1:E_1^{p,q}(K,L) \longrightarrow E_1^{p+1,q}(K,L)
\end{equation*}
is a morphism of variations of $\bQ$-Hodge structure.
Therefore
$E_2^{p,q}(K,L)$
is a polarizable variation of $\bQ$-Hodge structure
of weight $q$ for every $p,q$
and $F_{\rec}=F_{d}=F_{d^{\ast}}$ on $E_2^{p,q}(K,L)$
by the lemma on two filtrations in \cite{DeligneIII}.
Then the morphism
\begin{equation*}
d_2:E_2^{p,q}(K,L) \longrightarrow E_2^{p+2,q-1}(K,L)
\end{equation*}
between $E_2$-terms is a morphism of variations 
of $\bQ$-Hodge structure 
of weight $q$ and of weight $q-1$ respectively,
which implies that $d_2=0$ (see Lemma \ref{lemma 
for pre-variations} \ref{a property of morphism of pre-variations}). 
Therefore the spectral sequence $E_r^{p,q}(K,L)$
degenerates at $E_2$-terms
and $F=F_{\rec}=F_{d}=F_{d^{\ast}}$
on $E_{\infty}^{p,q}(K,L)=\Gr_{-p}^LH^{p+q}(K)$
by the lemma on two filtrations again.
Thus it turns out that
\begin{equation*}
\Gr_m^{L[i]}H^i(K)
=(\Gr_m^{L[i]}R^if_{\ast}\bQ_{X\sso},
(\Gr_m^{L[i]}R^if_{\ast}\Omega_{X\sso/Y},F),
\Gr_m^{L[i]}R^if_{\ast}(\alpha_{X\sso/Y}))
\end{equation*}
is a polarizable pre-variation
of $\bQ$-Hodge structure of weight $m$ on $Y$ for every $i,m$.

What remains to prove is that
the morphism
\begin{equation}
\label{morphism d otimes id}
d \otimes \id:
\xO_Y \otimes R^if_{\ast}\bQ_{X\sso}
\longrightarrow
\Omega^1_Y \otimes R^if_{\ast}\bQ_{X\sso}
\end{equation}
satisfies the Griffiths transversality
under the identification \eqref{isomorphism from Q to O}.

Now we consider
$Rf_{\ast}f_{\bullet}^{-1}\Omega_Y$ and $Rf_{\ast}\Omega_{X\sso}$, 
where $f_{\bullet}^{-1}\Omega_Y$ denotes
the complex $\{f_p^{-1}\Omega_Y\}_{p \ge 0}$ on $X\sso$.
The filtration $G$ on $\Omega_Y$ and $\Omega_{X\sso}$
induces the filtration $G$ on
$Rf_{\ast}f_{\bullet}^{-1}\Omega_Y$
and $Rf_{\ast}\Omega_{X\sso}$.
Moreover, the filtration $F$ on $\Omega_{X\sso}$ induces
the filtration $F$ on $Rf_{\ast}\Omega_{X\sso}$.
The canonical morphism
\begin{equation*}
\gamma:f_{\bullet}^{-1}\Omega_Y \longrightarrow \Omega_{X\sso}
\end{equation*}
which is a filtered quasi-isomorphism with respect to the filtration $G$
by the relative Poincar\'e lemma,
induces the filtered quasi-isomorphism
\begin{equation*}
Rf_{\ast}(\gamma):
Rf_{\ast}f_{\bullet}^{-1}\Omega_Y
\longrightarrow
Rf_{\ast}\Omega_{X\sso}
\end{equation*}
with respect to $G$.
Now we consider the spectral sequences associated to $G$,
and the morphism of the spectral sequences
induced by $Rf_{\ast}(\gamma)$.

We have
\begin{equation}
\label{iso for OmegaY}
\Gr_G^pRf_{\ast}f_{\bullet}^{-1}\Omega_Y
\simeq
Rf_{\ast}f_{\bullet}^{-1}\Omega^p_Y[-p]
\simeq
\Omega^p_Y \otimes Rf_{\ast}\bQ_{X\sso}[-p]
\end{equation}
and 
\begin{equation}
\label{iso for OmegaX}
\begin{split}
\Gr_G^pRf_{\ast}\Omega_{X\sso}
&\simeq
Rf_{\ast}(f_{\bullet}^{-1}\Omega^p_Y \otimes \Omega_{X\sso/Y}[-p]) \\
&\simeq
\Omega^p_Y \otimes Rf_{\ast}\Omega_{X\sso/Y}[-p]
\end{split}
\end{equation}
for every $p$,
such that the diagram
\begin{equation}
\label{commutativity of alpha and gamma}
\begin{CD}
\Gr_G^pRf_{\ast}f_{\bullet}^{-1}\Omega_Y
@>{\simeq}>>
\Omega^p_Y \otimes Rf_{\ast}\bQ_{X\sso}[-p] \\
@V{\Gr_G^pRf_{\ast}(\gamma)}VV @VV{\id \otimes Rf_{\ast}(\alpha_{X\sso/Y})[-p]}V \\
\Gr_G^pRf_{\ast}\Omega_{X\sso}
@>>{\simeq}>
\Omega^p_Y \otimes Rf_{\ast}\Omega_{X\sso/Y}[-p]
\end{CD}
\end{equation}
is commutative.

The morphism
\begin{equation}
\label{regular connection nabla}
\nabla:
R^if_{\ast}\Omega_{X\sso/Y}
\longrightarrow
\Omega^1_Y \otimes R^if_{\ast}\Omega_{X\sso/Y}
\end{equation}
is induced by the morphism of $E_1$-terms
\begin{equation*}
d_1:
E_1^{0,i}(Rf_{\ast}\Omega_{X\sso},G)
\longrightarrow
E_1^{1,i}(Rf_{\ast}\Omega_{X\sso},G)
\end{equation*}
via the identification
\eqref{iso for OmegaX}.
On the other hand, the morphism of $E_1$-terms
\begin{equation*}
d_1:
E_1^{0,i}(Rf_{\ast}f_{\bullet}^{-1}\Omega_Y,G)
\longrightarrow
E_1^{1,i}(Rf_{\ast}f_{\bullet}^{-1}\Omega_Y,G)
\end{equation*}
is identified with
\begin{equation*}
d \otimes \id:
\xO_Y \otimes R^if_{\ast}\bQ_{X\sso}
\longrightarrow
\Omega^1_Y \otimes R^if_{\ast}\bQ_{X\sso}
\end{equation*}
by the isomorphism
\eqref{iso for OmegaY}.
By the commutativity of the diagram 
\eqref{commutativity of alpha and gamma},
the morphisms $\nabla$ and $d \otimes \id$
are identified under the isomorphism
\eqref{isomorphism from Q to O}.
Because $\nabla$ satisfies the Griffiths transversality,
so does $d \otimes \id$.
\end{proof}

\begin{rem}
\label{remark on E1-degeneration}
The spectral sequence
$E_r^{p,q}(Rf_{\ast}\Omega_{X\sso/Y},F)$
degenerates at $E_1$-terms
by the lemma on two filtrations
\cite[Proposition (7.2.8)]{DeligneIII}.
\end{rem}

\begin{rem} 
\label{functoriality or VMHS for semi-simplicial variety}
The construction above is functorial
and compatible with the pull-back by the morphism $Y' \longrightarrow Y$.
\end{rem}

\begin{lem}
\label{lemma for relative cohomology}
Let $f:X\sso \longrightarrow Y$
and $g:Z\sso \longrightarrow Y$ be
smooth projective augmented strict semi-simplicial varieties
and $\varphi:Z\sso \longrightarrow X\sso$
a morphism of semi-simplicial varieties
compatible with the augmentations
$X\sso \longrightarrow Y$ and $Z\sso \longrightarrow Y$.
The morphism $\varphi$ induces
a morphism of {\itshape complexes}
\begin{equation*}
\varphi^{-1}:
Rf_{\ast}\bQ_{X\sso}
\longrightarrow
Rg_{\ast}\bQ_{Z\sso}
\end{equation*}
by using the Godment resolutions
as we mentioned in $\ref{Godment resolution}$.
The cone of the morphism $\varphi^{-1}$
is denoted by $C(\varphi^{-1})$.
Then $H^i(C(\varphi^{-1}))$ underlies
a \gpv of $\bQ$-mixed Hodge structure for every $i$.
\end{lem}
\begin{proof}
A filtration $L$ on $C(\varphi^{-1})$ is defined by
\begin{equation*}
L_mC(\varphi^{-1})^n
=L_{m-1}(Rf_{\ast}\bQ_{X\sso})^{n+1}
\oplus L_m(Rg_{\ast}\bQ_{Z\sso})^n
\end{equation*}
where $L$ on the right hand sides
denotes the filtrations defined
in the proof of
Lemma \ref{lemma for cohomology of semi-simplicial varieties}.

The morphism $\varphi:Z\sso \longrightarrow X\sso$
induces another morphisms
of {\itshape complexes}
\begin{equation*}
\varphi^{\ast}:
Rf_{\ast}\Omega_{X\sso/Y}
\longrightarrow
Rg_{\ast}\Omega_{Z\sso/Y}
\end{equation*}
which makes the diagram of {\itshape complexes}
\begin{equation}
\label{commutativity of Q and O for (X,Z)}
\begin{CD}
Rf_{\ast}\bQ_{X\sso}
@>{\varphi^{-1}}>>
Rg_{\ast}\bQ_{Z\sso} \\
@V{Rf_{\ast}(\alpha_{X\sso/Y})}VV
@VV{Rg_{\ast}(\alpha_{Z\sso/Y})}V \\
Rf_{\ast}\Omega_{X\sso/Y}
@>>{\varphi^{\ast}}>
Rg_{\ast}\Omega_{Z\sso/Y}
\end{CD}
\end{equation}
commutes,
where $\alpha_{X\sso/Y}$ and $\alpha_{Z\sso/Y}$
are the canonical morphism as in the proof of
Lemma \ref{lemma for cohomology of semi-simplicial varieties}.
Now we consider the mixed cone of the morphism $\varphi^{\ast}$
(see e.g. \cite[3.4]{ps}),
that is,
the cone $C(\varphi^{\ast})$
equipped with the filtrations
\begin{equation}
\begin{split}
&L_mC(\varphi^{\ast})^n
=L_{m-1}(Rf_{\ast}\Omega_{X\sso/Y})^{n+1}
\oplus L_m(Rg_{\ast}\Omega_{Z\sso/Y})^n \\
&F^pC(\varphi^{\ast})^n
=F^p(Rf_{\ast}\Omega_{X\sso/Y})^{n+1}
\oplus F^p(Rg_{\ast}\Omega_{Z\sso/Y})^n
\end{split}
\label{filtrations L and F on the cone}
\end{equation}
where $L$ and $F$ on the right hand sides
denote the filtrations defined
in the proof of
Lemma \ref{lemma for cohomology of semi-simplicial varieties}.
Then the commutative diagram
\eqref{commutativity of Q and O for (X,Z)}
induces a morphism of filtered complexes
$\alpha:(C(\varphi^{-1}),L) \longrightarrow (C(\varphi^{\ast}),L)$
which induces a filtered quasi-isomorphism
$(C(\varphi^{-1}),L) \otimes \xO_Y \longrightarrow (C(\varphi^{\ast}),L)$.
Then we have
\begin{equation*}
\begin{split}
\Gr_m^LC(\varphi^{-1})
&=\Gr_{m-1}^LRf_{\ast}\bQ_{X\sso}[1]
\oplus \Gr_m^LRg_{\ast}\bQ_{Z\sso} \\
&=
R(f_{-m+1})_{\ast}\bQ_{X_{-m}}[m]
\oplus R(g_{-m})_{\ast}\bQ_{Z_{-m}}[m]
\end{split}
\end{equation*}
and
\begin{equation*}
\begin{split}
(\Gr_m^L&C(\varphi^{\ast}),F) \\
&=
(\Gr_{m-1}^LRf_{\ast}\Omega_{X\sso/Y}[1],F)
\oplus
(\Gr_m^LRg_{\ast}\Omega_{Z\sso/Y},F) \\
&=
(R(f_{-m+1})_{\ast}\Omega_{X_{-m+1}/Y}[m],F)
\oplus
(R(g_{-m})_{\ast}\Omega_{Z_{-m}/Y}[m],F)
\end{split}
\end{equation*}
for every $m$.
Therefore the data
\begin{equation*}
(E_1^{p,q}(C(\varphi^{-1}),L),
(E_1^{p,q}(C(\varphi^{\ast}),L),F_{\rec}),
E_1^{p,q}(\alpha))
\end{equation*}
is a \pv of $\bQ$-Hodge structure of weight $q$.
Then the same argument in the proof of
Lemma \ref{lemma for cohomology of semi-simplicial varieties}
implies that the spectral sequences
$E_r^{p,q}(C(\varphi^{-1}),L)$ and
$E_r^{p,q}(C(\varphi^{\ast}),L)$ degenerate
at $E_2$-terms,
the spectral sequence $E_r^{p,q}(C(\varphi^{\ast}),F)$
degenerates at $E_1$-terms
and that the data
\begin{equation*}
((H^i(C(\varphi^{-1})),L[i]),(H^i(C(\varphi^{\ast})),L[i],F),H^i(\alpha))
\end{equation*}
is a \gp pre-variation of $\bQ$-mixed Hodge structure on $Y$
for every $i$.

What remains to prove is the Griffiths transversality.
We consider the complexes
$\Omega_{X\sso}$, $\Omega_{Z\sso}$,
$g\sso^{-1}\Omega_Y$ and $f\sso^{-1}\Omega_Y$
with the decreasing filtration $G$ as in the proof of
Lemma \ref{lemma for cohomology of semi-simplicial varieties}.
We have the commutative diagram
\begin{equation*}
\begin{CD}
Rf_{\ast}f^{-1}\sso\Omega_Y
@>>>
Rg_{\ast}g^{-1}\sso\Omega_Y \\
@VVV @VVV \\
Rf_{\ast}\Omega_{X\sso}
@>>>
Rg_{\ast}\Omega_{Z\sso}
\end{CD}
\end{equation*}
where the vertical arrows are filtered quasi-isomorphism
with respect to the filtration $G$.
The top horizontal arrow is denoted by $\psi^{-1}$
and the bottom by $\psi^{\ast}$
for a while.
Considering the cones
$C(\psi^{-1})$ and $C(\psi^{\ast})$
with the filtration $G$
defined by the same way as $F$ 
in \eqref{filtrations L and F on the cone}, 
we obtain a commutative diagram of quasi-isomorphisms
\begin{equation*}
\begin{CD}
C(\varphi^{-1})[-p] \otimes \Omega^p_Y
@>>>
\Gr_G^pC(\psi^{-1}) \\
@V{\alpha \otimes \id}VV @VVV \\
C(\varphi^{\ast})[-p] \otimes \Omega^p_Y
@>>>
\Gr_G^pC(\psi^{\ast})
\end{CD}
\end{equation*}
for every $p$.
Then we can check the Griffiths transversality
by the same way as in the proof of
Lemma \ref{lemma for cohomology of semi-simplicial varieties}.
\end{proof}

\begin{say} 
\label{review on Steenbrink's results}
Now we review Steenbrink's results
in \cite{steenbrink1}, \cite{steenbrink2}
and fix the notation for the later use.

Let $X$ be a smooth complex variety
and $f:X \longrightarrow \Delta$
a projective surjective morphism.
We set $\zo{X}=f^{-1}(\pd)$,
$E=f^{-1}(0)$. 
The coordinate function on $\Delta$ is denoted by $t$. 
We assume that $E_{\red}$ is a \snc divisor on $X$
and $f:\zo{X} \longrightarrow \pd$ is a smooth morphism.
Moreover, we assume that
$R^if_{\ast}\bQ_{\zo{X}}$ are of unipotent monodromy
for all $i$ for simplicity.

A finite decreasing filtration $G$ on $\Omega_X(\log E)$ is defined by
\begin{equation*}
\begin{split}
&G^0\Omega_X(\log E)=\Omega_X(\log E) \\
&G^1\Omega_X(\log E)
=
\im(f^{-1}\Omega^1_{\Delta}(\log 0) \otimes \Omega_X(\log E)[-1]
\longrightarrow \Omega_X(\log E)) \\
&G^2\Omega_X(\log E)=0
\end{split}
\end{equation*}
as in \eqref{filtration G on Omega}.
Then the morphism
\begin{equation*}
\nabla:
R^if_{\ast}\Omega_{X/\Delta}(\log E)
\longrightarrow
\Omega^1_{\Delta}(\log 0) \otimes R^if_{\ast}\Omega_{X/\Delta}(\log E)
\end{equation*}
is obtained as the morphism of $E_1$-terms
of the spectral sequence
$E_r^{p,q}(Rf_{\ast}\Omega_{X/\Delta}(\log E),G)$.
The restriction $\nabla|_{\pd}$ is identified with
$d \otimes \id$ on $\xO_{\pd} \otimes R^if_{\ast}\bQ_{\zo{X}}$
via the isomorphisms
\begin{equation*}
R^if_{\ast}\Omega_{X/\Delta}(\log E)|_{\pd}
\simeq
R^if_{\ast}\Omega_{\zo{X}/\pd}
\simeq
\xO_{\pd} \otimes R^if_{\ast}\bQ_{\zo{X}}
\end{equation*}
by definition.

Steenbrink proved that
$R^if_{\ast}\Omega_{X/\Delta}(\log E)$
is a locally free coherent $\xO_{\Delta}$-module
and $\res_0(\nabla)$ is nilpotent.
Therefore
$R^if_{\ast}\Omega_{X/\Delta}(\log E)$ is the canonical extension
of $\xO_{\pd} \otimes R^if_{\ast}\bQ_{\zo{X}}$
for every $i$.
Once we know the local freeness of $R^if_{\ast}\Omega_{X/\Delta}(\log E)$,
the canonical morphism
\begin{equation}
\label{base change iso for relative log de Rham complex}
R^if_{\ast}\Omega_{X/\Delta}(\log E) \otimes \bC(0)
\overset{\simeq}{\longrightarrow}
H^i(E,\Omega_{X/\Delta}(\log E) \otimes \xO_E)
\end{equation}
is an isomorphism for every $i$.

The filtration $G$ on $\Omega_X(\log E)$ induces
a filtration on $\Omega_X(\log E) \otimes \xO_E$,
which is denoted by $G$ again.
Then we have
\begin{equation}
\label{identification for G1}
\Gr_G^1\Omega_X(\log E) \otimes \xO_E
\simeq (\Omega_{X/\Delta}(\log E) \otimes \xO_E)[-1]
\end{equation}
because we have the identification
\begin{equation*}
G^1\Omega_X(\log E)
=d\log t \wedge \Omega_X(\log E)[-1]
\simeq \Omega_{X/\Delta}(\log E)[-1]
\end{equation*}
where $d\log t=dt/t$.
Therefore we have
\begin{equation*}
\begin{split}
&E_1^{0,i}(R\Gamma(E, \Omega_X(\log E) \otimes \xO_E),G)
\simeq
H^i(E, \Omega_{X/\Delta}(\log E) \otimes \xO_E) \\
&E_1^{1,i}(R\Gamma(E,\Omega_X(\log E) \otimes \xO_E),G)
\simeq
H^i(E, \Omega_{X/\Delta}(\log E) \otimes \xO_E)
\end{split}
\end{equation*} 
for every $i$.
Then the morphism of $E_1$-terms
\begin{equation*}
H^i(E,\Omega_{X/\Delta}(\log E) \otimes \xO_E)
\longrightarrow
H^i(E,\Omega_{X/\Delta}(\log E) \otimes \xO_E)
\end{equation*}
coincides with $\res_0(\nabla)$
under the identification
\eqref{base change iso for relative log de Rham complex}.

In \cite{steenbrink1}, 
Steenbrink constructed a cohomological $\bQ$-mixed Hodge complex,
denoted by
\begin{equation*}
A_{X/\Delta}
=((A^{\bQ}_{X/\Delta},W), (A^{\bC}_{X/\Delta},W,F),\alpha_{X/\Delta})
\end{equation*}
in this article,
which admits a filtered quasi-isomorphism
\begin{equation}
\label{morphism theta}
\theta_{X/\Delta}:
(\Omega_{X/\Delta}(\log E) \otimes \xO_{E_{\red}},F)
\longrightarrow
(A^{\bC}_{X/\Delta},F)
\end{equation}
where the filtration $F$ on the left hand side denotes the filtration
induced by the stupid filtration on $\Omega_{X/\Delta}(\log E)$.
Because the canonical morphism
\begin{equation*}
H^i(E,\Omega_{X/\Delta}(\log E) \otimes \xO_E)
\longrightarrow
H^i(E_{\red},\Omega_{X/\Delta}(\log E) \otimes \xO_{E_{\red}})
\end{equation*}
is an isomorphism for every $i$
by the unipotent monodromy condition,
we have the isomorphisms
\begin{equation}
\label{isomorphism induced by theta}
\begin{split}
R^if_{\ast}\Omega_{X/\Delta}(\log E) \otimes \bC(0)
&\overset{\simeq}{\longrightarrow}
H^i(E,\Omega_{X/\Delta}(\log E) \otimes \xO_E) \\
&\overset{\simeq}{\longrightarrow}
H^i(E_{\red},\Omega_{X/\Delta}(\log E) \otimes \xO_{E_{\red}}) \\
&\overset{\simeq}{\longrightarrow}
H^i(E_{\red},A^{\bC}_{X/\Delta})
\end{split}
\end{equation}
for every $i$.

Here we just recall the definition of $A^{\bC}_{X/\Delta}$
in \cite{steenbrink1}. 
The filtration $W_X(E)$ denotes the increasing filtration on $\Omega_X(\log E)$
defined by the order of poles along $E$ as usual. 
The complex $A^{\bC}_{X/\Delta}$ is the single complex
associated to the double complex
$((A^{\bC}_{X/\Delta})^{p,q},d',d'')$ given by
\begin{equation*}
(A^{\bC}_{X/\Delta})^{p,q}=\Omega^{p+q+1}_X(\log E)/W_X(E)_q
\end{equation*}
with the differentials
\begin{equation*}
\begin{split}
&d'=-d:(A^{\bC}_{X/\Delta})^{p,q}
\longrightarrow (A^{\bC}_{X/\Delta})^{p+1,q} \\
&d''=-d\log t \wedge:
(A^{\bC}_{X/\Delta})^{p,q} \longrightarrow (A^{\bC}_{X/\Delta})^{p,q+1}
\end{split}
\end{equation*}
where $d$ is the morphism induced from the differential of $\Omega_X(\log E)$
and where $d\log t \wedge$ denotes the morphism
given by the wedge product with $d\log t=dt/t$. 
For the definitions of $W$ and $F$ on $A^{\bC}_{X/\Delta}$, 
see \cite[4.17]{steenbrink1}. 
The morphism given by
\begin{equation*}
\Omega_X^p(\log E) \ni \omega
\mapsto
d\log t \wedge \omega \in (A^{\bC}_{X/\Delta})^{p,0}
\end{equation*}
induces the morphism \eqref{morphism theta}
\begin{equation*}
\theta_{X/\Delta}:
\Omega_{X/\Delta}(\log E) \otimes \xO_{E_{\red}}
\longrightarrow
A^{\bC}_{X/\Delta}
\end{equation*}
which turns out to be a filtered quasi-isomorphism
with respect to the filtrations $F$ on the both sides
(see \cite[Lemma (4.15)]{steenbrink1}).
The composite
\begin{equation*}
\Omega_X(\log E) \otimes \xO_E
\longrightarrow
\Omega_X(\log E) \otimes \xO_{E_{\red}}
\overset{\theta_{X/\Delta}}{\longrightarrow}
A^{\bC}_{X/\Delta}
\end{equation*}
is also denoted by $\theta_{X/\Delta}$ by abuse of the notation.

On the other hand,
the projection
\begin{equation*}
\Omega^p_X(\log E)
\longrightarrow
\Omega^p_X(\log E)/W_X(E)_0 =(A^{\bC}_{X/\Delta})^{p-1,0}
\subset (A^{\bC}_{X/\Delta}[-1])^p
\end{equation*}
induces the morphism
\begin{equation*}
\pi_{X/\Delta}:
\Omega^p_X(\log E) \otimes \xO_E
\longrightarrow
(A^{\bC}_{X/\Delta}[-1])^p
\end{equation*}
for every $p$.
Moreover, the projection
\begin{equation*}
\begin{split}
(A^{\bC}_{X/\Delta})^{p,q}=\Omega^{p+q+1}_X(\log &D)/W_X(E)_q \\
&\longrightarrow
\Omega^{p+q+1}_X(\log E)/W_X(E)_{q+1}=(A^{\bC}_{X/\Delta})^{p-1,q+1}
\end{split}
\end{equation*}
induces a morphism of bifiltered complexes
\begin{equation*}
(A^{\bC}_{X/\Delta},W,F)
\longrightarrow
(A^{\bC}_{X/\Delta},W[-2],F[-1])
\end{equation*}
denoted by $\nu_{X/\Delta}$ 
(see \cite[(4.22), Proposition (4.23)]{steenbrink1}). 
Then we have
\begin{equation}
\label{relation between pi and nu}
d\pi_{X/\Delta}=\pi_{X/\Delta}d+\nu_{X/\Delta}\theta_{X/\Delta}:
\Omega^p_X(\log E)\otimes \mathcal O_E
\longrightarrow
(A^{\bC}_{X/\Delta})^p
\end{equation}
for every $p$,
where $d$ on the left hand side
is the differential of $A^{\bC}_{X/\Delta}[-1]$.

We set
\begin{equation*}
N_{X/\Delta}=H^i(\nu_{X/\Delta}):
H^i(E_{\red},A^{\bC}_{X/\Delta})
\longrightarrow
H^i(E_{\red},A^{\bC}_{X/\Delta})
\end{equation*}
for every $i$.
It is proved that
the morphism
\begin{equation}
\label{coincidence of W and monodromy weight filtration}
N_{X/\Delta}^k:
\Gr_k^WH^i(E_{\red},A^{\bC}_{X/\Delta})
\longrightarrow
\Gr_{-k}^WH^i(E_{\red},A^{\bC}_{X/\Delta})
\end{equation}
is an isomorphism for every $k \ge 0$
(see Steenbrink \cite{steenbrink1}, El Zein \cite{elzein1},
Saito \cite{morihiko}, Guillen--Navarro Aznar \cite{gn},
Usui \cite{usui2}).

The complex $B_{X/\Delta}$ is defined by
\begin{equation*}
B^p_{X/\Delta}
=(A^{\bC}_{X/\Delta})^{p-1} \oplus (A^{\bC}_{X/\Delta})^p
\end{equation*}
with the differential
\begin{equation*}
d(x,y)=(-dx-\nu_{X/\Delta}(y),dy)
\end{equation*}
for $x \in (A^{\bC}_{X/\Delta})^{p-1}$
and $y \in (A^{\bC}_{X/\Delta})^p$,
where $d$ denotes the differential of the complex $A^{\bC}_{X/\Delta}$.
We define a filtration $G$ on $B_{X/\Delta}$ by
\begin{equation*}
\begin{split}
&G^0B_{X/\Delta}=B_{X/\Delta} \\
&G^1B_{X/\Delta}=A^{\bC}_{X/\Delta}[-1] \\
&G^2B_{X/\Delta}=0
\end{split}
\end{equation*}
where $A^{\bC}_{X/\Delta}[-1]$ is regarded as a subcomplex of $B_{X/\Delta}$
by the inclusion
$(A^{\bC}_{X/\Delta})^{p-1} \longrightarrow B^p_{X/\Delta}$
for every $p$.

A morphism
\begin{equation*}
\Omega^p_X(\log E) \otimes \xO_E \ni \omega
\mapsto
(\pi_{X/\Delta}(\omega), \theta_{X/\Delta}(\omega))
\in (A^{\bC}_{X/\Delta})^{p-1} \oplus (A^{\bC}_{X/\Delta})^p
\end{equation*}
defines a morphism of complexes
\begin{equation}
\label{morphism eta}
\eta_{X/\Delta}:\Omega_X(\log E) \otimes \xO_E \longrightarrow B_{X/\Delta}
\end{equation}
by \eqref{relation between pi and nu}.
It is easy to check that the morphism $\eta_{X/\Delta}$
preserves the filtration $G$ on the both sides.
Note that the diagrams
\begin{equation*}
\begin{CD}
\Gr_G^0\Omega_X(\log E) \otimes \xO_E
@>{\Gr_G^0\eta_{X/\Delta}}>>
\Gr_G^0B_{X/\Delta} \\
@| @| \\
\Omega_{X/\Delta}(\log E) \otimes \xO_E
@>>{\theta_{X/\Delta}}>
A^{\bC}_{X/\Delta}
\end{CD}
\end{equation*}
and
\begin{equation*}
\begin{CD}
\Gr_G^1\Omega_X(\log E) \otimes \xO_E
@>{\Gr_G^1\eta_{X/\Delta}}>> \Gr_G^1B_{X/\Delta} \\
@| @| \\
\Omega_{X/\Delta}(\log E) \otimes \xO_E[-1]
@>>{\theta_{X/\Delta}[-1]}> A^{\bC}_{X/\Delta}[-1]
\end{CD}
\end{equation*}
are commutative.
Considering the morphisms between $E_1$-terms
induced by $\eta_{X/\Delta}$,
we have
\begin{equation*}
\res_0(\nabla)=-N_{X/\Delta}
\end{equation*}
on $H^i(E_{\red},A^{\bC}_{X/\Delta})$
under the isomorphism \eqref{isomorphism induced by theta}.

We remark that the construction above is functorial. 
For the rational structure $A^{\bQ}_{X/\Delta}$,
we can use the construction by Steenbrink--Zucker \cite{sz}.

In \cite{steenbrink2},
the local freeness of $R^if_{\ast}\Omega^p_{X/\Delta}(\log E)$ is proved.
\end{say}

\begin{lem} 
\label{GPVMHS for semi-simplicial variety}
Let $f:X\sso \longrightarrow Y$
be a projective surjective augmented strict semi-simplicial variety
to a smooth algebraic variety $Y$.
Moreover, assume that
$X_p$ is smooth for every $p$.
Then there exists a Zariski open dense subset $\zo{Y}$ of $Y$
such that $(R^if_{\ast}\bQ_{X\sso})|_{\zo{Y}}$
underlies an admissible \gpv of $\bQ$-mixed Hodge structure 
for every $i$.
\end{lem}
\begin{proof}
There exists a non-empty Zariski open subset $\zo{Y}$ of $Y$
such that the morphism
$f:X\sso \longrightarrow Y$
is smooth over $\zo{Y}$.
We set $\zo{X}\sso=f^{-1}(\zo{Y})$
and denote the induced morphisms
$\zo{X}\sso \longrightarrow \zo{Y}$
by the same letter $f$ by abuse of the language.
Then
$(R^if_{\ast}\bQ_{X\sso})|_{\zo{Y}}$
underlies a \gpv of $\bQ$-mixed Hodge structure 
by Lemma \ref{lemma for cohomology of semi-simplicial varieties}.

It is sufficient to prove
that the \gpv of $\bQ$-mixed Hodge structure above is admissible. 
Because the \gpv of $\bQ$-mixed Hodge structure constructed
in Lemma \ref{lemma for cohomology of semi-simplicial varieties}
commutes with the base change
as in Remark \ref{functoriality or VMHS for semi-simplicial variety},
we may assume $Y=\Delta$ and $\zo{Y}=\Delta\!^*$.

Our variation of graded polarizable $\mathbb Q$-mixed Hodge structure 
has a $\mathbb Z$-structure. 
Therefore, the quasi-unipotency of the monodromy around the origin
is obvious by Remark \ref{q-uni}.
Thus we have the property
\eqref{quasi-unipotency}
in Definition \ref{admissibility} \ref{pre-admissibility on Delta*}. 
Once we know the quasi-unipotency of the monodromy,
Lemma 1.9.1 in \cite{kashiwara}
allows us to assume
that $f:X\sso \longrightarrow \Delta$
is of unipotent monodromy.
Moreover we may assume that
$f^{-1}(0)_{\red}$
is a \snc divisor
on the smooth semi-simplicial complex variety $X\sso$.
We set
$E\sso=f_{\bullet}^{-1}(0)$.
Note that $E\sso$ and $E_{\bullet\red}=\{(E_p)_{\red}\}_{p \ge 0}$
are strict semi-simplicial subvarieties of $X\sso$.

Consider the bifiltered complex
\begin{equation*}
(Rf_{\ast}\Omega_{X\sso/\Delta}(\log E\sso),L,F)
\end{equation*}
where $L$ is defined in \eqref{filtration L}
and $F$ as in \eqref{filtration W}.
We trivially have
\begin{equation*}
(Rf_{\ast}\Omega_{X\sso/\Delta}(\log E\sso),L,F)|_{\pd}
=(Rf_{\ast}\Omega_{\zo{X}\sso/\pd},L,F)
\end{equation*}
by definition.

\begin{step}
In this first step,
we prove the local freeness of several coherent $\xO_{\Delta}$-modules.

Consider the spectral sequence
\begin{equation*}
(E_r^{p,q}(Rf_{\ast}\Omega_{X\sso/\Delta}(\log E\sso),L),F_{\rec})
\end{equation*}
associated to the filtration $L$
on the complex $Rf_{\ast}\Omega_{X\sso/\Delta}(\log E\sso)$.

By \eqref{gr of L}, we have
\begin{equation}
\label{E1-terms of K(log 0)}
E_1^{p,q}(Rf_{\ast}\Omega_{X\sso/\Delta}(\log E\sso),L)
\simeq
R^qf_{p\ast}\Omega_{X_p/\Delta}(\log E_p)
\end{equation}
which is the canonical extension of
\begin{equation*}
\begin{split}
E_1^{p,q}(Rf_{\ast}\Omega_{\zo{X}\sso/\pd},L)
&\simeq
R^qf_{p\ast}\Omega_{\zo{X}_p/\pd} \\
&\simeq
\xO_{\pd} \otimes R^qf_{p\ast}\bQ_{\zo{X}_p}
\simeq
\xO_{\pd} \otimes E_1^{p,q}(Rf_{\ast}\bQ_{\zo{X}\sso},L)
\end{split}
\end{equation*}
by \cite{steenbrink1}.
Because taking the canonical extension is an exact functor
by \cite[Proposition 5.2 (d)]{Deligne},
$E_2^{p,q}(Rf_{\ast}\Omega_{X\sso/\Delta}(\log E\sso),L)$
is the canonical extension of
$E_2^{p,q}(Rf_{\ast}\Omega_{\zo{X}\sso/\pd},L)
\simeq \xO_{\pd} \otimes_{\bQ} E_2^{p,q}(Rf_{\ast}\bQ_{\zo{X}\sso},L)$.
Therefore
$E_2^{p,q}(Rf_{\ast}\Omega_{X\sso/\Delta}(\log E\sso),L)$
is a locally free coherent $\xO_{\Delta}$-module.
Once we know the local freeness of $E_2$-terms,
the spectral sequence
$E_r^{p,q}(Rf_{\ast}\Omega_{X\sso/\Delta}(\log E\sso),L)$
degenerates at $E_2$-terms
because its restriction on $\pd$ degenerates at $E_2$-terms.
Thus we obtain that
\begin{equation*}
\begin{split}
E_2^{p,q}(Rf_{\ast}\Omega_{X\sso/\Delta}(\log E\sso),L)
&\simeq
E_{\infty}^{p,q}(Rf_{\ast}\Omega_{X\sso/\Delta}(\log E\sso),L) \\
&\simeq
\Gr_{-p}^LR^{p+q}f_{\ast}\Omega_{X\sso/\Delta}(\log E\sso)
\end{split}
\end{equation*}
is locally free for every $p,q$.
In particular,
$R^if_{\ast}\Omega_{X\sso/\Delta}(\log E\sso)$
is a locally free coherent $\xO_{\Delta}$-module
with the property
\begin{equation*}
R^if_{\ast}\Omega_{X\sso/\Delta}(\log E\sso)|_{\pd}
\simeq
R^if_{\ast}\Omega_{\zo{X}\sso/\Delta}
\simeq
\xO_{\pd} \otimes R^if_{\ast}\bQ_{\zo{X}\sso}
\end{equation*}
for every $i$.

The morphism of $E_1$-terms
\begin{equation}
\label{morphism d1}
d_1:
R^qf_{p\ast}\Omega_{X_p/\Delta}(\log E_p)
\longrightarrow
R^qf_{p+1\ast}\Omega_{X_{p+1}/\Delta}(\log E_{p+1})
\end{equation}
preserves the filtration $F$
because $F_{\rec}=F_d=F_{d^{\ast}}$
on the $E_1$-terms in general
and because $F_{\rec}$
on $E_1^{p,q}(Rf_{\ast}\Omega_{X\sso/\Delta}(\log E\sso),L)$
coincides with $F$
on $R^qf_{p\ast}\Omega_{X_p/\Delta}(\log E_p)$
under the isomorphism \eqref{E1-terms of K(log 0)}.
On the other hand,
the filtration $F$ on $R^qf_{p\ast}\Omega_{X_p/\Delta}(\log E_p)$
coincides with the filtration
obtained by the nilpotent orbit theorem in \cite{schmid}
because
\begin{equation*}
\Gr_F^rR^qf_{p\ast}\Omega_{X_p/\Delta}(\log E_p) \\
\simeq
R^{q-r}f_{p\ast}\Omega^r_{X_p/\Delta}(\log E_p)
\end{equation*}
is locally free for every $r$
(see Corollary \ref{corollary for the extensions of a filtrartion} below).
By the $SL_2$-orbit theorem in \cite{schmid},
\begin{equation*}
(R^qf_{p\ast}\Omega_{X_p/\Delta}(\log E_p) \otimes \bC(0),W,F)
\end{equation*}
underlies a $\bQ$-mixed Hodge structure for every $p,q$,
where $W$ denotes the monodromy weight filtration.
On the other hand,
the morphism
\begin{equation*}
d_1(0):
R^qf_{p\ast}\Omega_{X_p/\Delta}(\log E_p) \otimes \bC(0)
\longrightarrow
R^qf_{p+1\ast}\Omega_{X_{p+1}/\Delta}(\log E_{p+1}) \otimes \bC(0)
\end{equation*}
induced by the morphism $d_1$ in \eqref{morphism d1}
underlies a morphism of $\bQ$-mixed Hodge structures
because the restriction of $d_1$ on $\pd$
preserves the $\bQ$-structures
$R^qf_{p\ast}\bQ_{\zo{X}_p}$ and $R^qf_{p+1\ast}\bQ_{\zo{X}_{p+1}}$.
Therefore $d_1(0)$ is strictly compatible with the filtrations $F$
on $R^qf_{p\ast}\Omega_{X_p/\Delta}(\log E_p) \otimes \bC(0)$
and $R^qf_{p+1\ast}\Omega_{X_{p+1}/\Delta}(\log E_{p+1}) \otimes \bC(0)$.
In other words,
the morphism
\begin{equation*}
\begin{split}
d_1(0):
E_1^{p,q}(Rf_{\ast}\Omega_{X\sso/\Delta}&(\log E\sso),L) \otimes \bC(0) \\
&\longrightarrow
E_1^{p+1,q}(Rf_{\ast}\Omega_{X\sso/\Delta}(\log E\sso),L) \otimes \bC(0)
\end{split}
\end{equation*}
is strictly compatible with the filtrations $F_{\rec}$
on the both sides.

Applying \ref{strictness from residue field to stalk}
in Lemma \ref{toy base change} to the complex
\begin{equation*}
(E_1^{\bullet,q}(Rf_{\ast}\Omega_{X\sso/\Delta}(\log E\sso),L),F_{\rec})
\end{equation*}
we conclude that
\begin{equation*}
\Gr_{F_{\rec}}^rE_2^{p,q}(Rf_{\ast}\Omega_{X\sso/\Delta}(\log E\sso),L)
\end{equation*}
is locally free coherent $\xO_{\Delta}$-module for every $p,q,r$
and that $F_{\rec}=F_d=F_{d^{\ast}}$ on
$E_2^{p,q}(Rf_{\ast}\Omega_{X\sso/\Delta}(\log E\sso),L)$.
Therefore
\begin{equation*}
\Gr_{F_{\rec}}^rE_2^{p,q}(Rf_{\ast}\Omega_{X\sso/\Delta}(\log E\sso),L)
\simeq
\Gr_F^r\Gr_{-p}^LR^{p+q}f_{\ast}\Omega_{X\sso/\Delta}(\log E\sso)
\end{equation*}
is locally free for every $p,q,r$. 
Moreover, the spectral sequence
\begin{equation*}
E_r^{p,q}(Rf_{\ast}\Omega_{X\sso/\Delta}(\log E\sso),F)
\end{equation*}
degenerates at $E_1$-terms
by the lemma on two filtrations as before. 
\end{step}

\begin{step}
The canonical morphism
\begin{equation*}
\Omega_{X\sso/\Delta}(\log E\sso)
\longrightarrow
\Omega_{X\sso/\Delta}(\log E\sso) \otimes \xO_{E\sso}
\end{equation*}
induces the morphism of complexes
\begin{equation}
\label{base change morphism for X bullet}
Rf_{\ast}\Omega_{X\sso/\Delta}(\log E\sso) \otimes \bC(0)
\longrightarrow
R\Gamma
(E\sso,
\Omega_{X\sso/\Delta}(\log E\sso) \otimes \xO_{E\sso})
\end{equation}
preserving the filtration $L$ on the both sides.
Then the morphism of the spectral sequences induces the morphism
\begin{equation}
\label{morphism of ss from Omega to Omega otimes OD}
\begin{split}
E_r^{p,q}(Rf_{\ast}\Omega_{X\sso/\Delta}&(\log E\sso), L) \otimes \bC(0) \\
&\longrightarrow
E_r^{p,q}(R\Gamma
(E\sso,
\Omega_{X\sso/\Delta}(\log E\sso) \otimes \xO_{E\sso}),L)
\end{split}
\end{equation}
for every $p,q,r$.
For the case of $r=1$,
the morphism above coincides with the canonical morphism
\begin{equation*}
R^qf_{p\ast}\Omega_{X_p/\Delta}(\log E_p) \otimes \bC(0)
\longrightarrow
H^q(E_p,\Omega_{X_p/\Delta}(\log E_p) \otimes \xO_{E_p})
\end{equation*}
which is an isomorphism for every $p,q$ as mentioned
in \ref{review on Steenbrink's results}.
Therefore the morphism
\begin{equation*}
\begin{split}
H^p(E_1^{\bullet,q}
(Rf_{\ast}\Omega_{X\sso/\Delta}&(\log E\sso),L) \otimes \bC(0)) \\
&\longrightarrow
E_2^{p,q}(R\Gamma
(E\sso,
\Omega_{X\sso/\Delta}(\log E\sso) \otimes \xO_{E\sso}),L)
\end{split}
\end{equation*}
is an isomorphism for every $p,q$.
Moreover, the canonical morphism
\begin{equation*}
\begin{split}
E_2^{p,q}(Rf_{\ast}\Omega_{X\sso/\Delta}&(\log E\sso), L) \otimes \bC(0) \\
&\longrightarrow
H^p(E_1^{\bullet,q}
(Rf_{\ast}\Omega_{X\sso/\Delta}(\log E\sso),L) \otimes \bC(0))
\end{split}
\end{equation*}
is an isomorphism for every $p,q$,
because
$E_2^{p,q}(Rf_{\ast}\Omega_{X\sso/\Delta}(\log E\sso), L)$
is locally free for every $p,q$ as proved in Step 1.
Thus we know that
the morphism
\eqref{morphism of ss from Omega to Omega otimes OD}
is an isomorphism for every $p,q$ and for $r=2$.
Therefore the $E_2$-degeneracy of
$E_r^{p,q}(Rf_{\ast}\Omega_{X\sso/\Delta}(\log E\sso), L)$
implies the $E_2$-degeneracy of
$E_r^{p,q}(R\Gamma(E\sso,\Omega_{X\sso/\Delta}(\log E\sso)
\otimes \xO_{E\sso}),L)$.
Moreover, we have the canonical isomorphism
\begin{equation*}
\Gr_m^LR^if_{\ast}\Omega_{X\sso/\Delta}(\log E\sso) \otimes \bC(0)
\longrightarrow
\Gr_m^LH^i(E\sso,\Omega_{X\sso/\Delta}(\log E\sso) \otimes \xO_{E\sso})
\end{equation*}
for every $i,m$.
In particular, the canonical morphism
\begin{equation}
\label{the canonical morphism from X bullet to D bullet}
R^if_{\ast}\Omega_{X\sso/\Delta}(\log E\sso) \otimes \bC(0)
\longrightarrow
H^i(E\sso,\Omega_{X\sso/\Delta}(\log E\sso) \otimes \xO_{E\sso})
\end{equation}
is an isomorphism for every $i$.
\end{step}

\begin{step}
Considering the filtered complex
\begin{equation*}
(Rf_{\ast}\Omega_{X\sso}(\log E\sso), G)
\end{equation*}
we obtain the log integrable connection
\begin{equation*}
\nabla:
R^if_{\ast}\Omega_{X\sso/\Delta}(\log E\sso)
\longrightarrow
\Omega^1_{\Delta}(\log 0)
\otimes
R^if_{\ast}\Omega_{X\sso/\Delta}(\log E\sso)
\end{equation*}
as the morphism of $E_1$-terms of the spectral sequence.
It is clear that
$\nabla|_{\pd}$ coincides with the connection \eqref{regular connection nabla}.

On the other hand,
we consider
\begin{equation*}
(R\Gamma(E\sso,\Omega_{X\sso}(\log E\sso) \otimes \xO_{E\sso}),G)
\end{equation*}
with the identification
\begin{equation*}
G^1\Omega_{X\sso}(\log E\sso) \otimes \xO_{E\sso}
\simeq
(\Omega_{X\sso/\Delta}(\log E\sso) \otimes \xO_{E\sso})[-1]
\end{equation*}
as in \eqref{identification for G1}. 
The same procedure in
\ref{review on Steenbrink's results}
shows the fact that the morphism of $E_1$-terms
\begin{equation}
\label{morphism of E1-terms for Omega otimes OD}
\begin{split}
H^i(E\sso,\Omega_{X\sso/\Delta}&(\log E\sso) \otimes \xO_{E\sso}) \\
&\longrightarrow
H^i(E\sso,\Omega_{X\sso/\Delta}(\log E\sso) \otimes \xO_{E\sso})
\end{split}
\end{equation}
coincides with $\res_0(\nabla)$
via the isomorphism
\eqref{the canonical morphism from X bullet to D bullet}. 
\end{step}

\begin{step}
The data $A_{X\sso/\Delta}$ gives us an object on
the semi-simplicial variety
$E_{\bullet\red}$
because Steenbrink's construction in \ref{review on Steenbrink's results}
is functorial as mentioned there.
Then the data
\begin{equation}
\label{fmhc A}
(R\Gamma(E_{\bullet\red}, A_{X\sso/\Delta}),L,\delta(W,L),F)
\end{equation}
is obtained.
We set
\begin{equation*}
A_{\bC}=R\Gamma(E_{\bullet\red}, A^{\bC}_{X\sso/\Delta})
\end{equation*}
for simplicity.
The morphism
\begin{equation*}
\theta_{X\sso/\Delta}:
\Omega_{X\sso/\Delta}(\log E\sso) \otimes \xO_{E\sso}
\longrightarrow
A^{\bC}_{X\sso/\Delta}
\end{equation*}
induces the morphism
\begin{equation}
\label{morphism RGamma(theta)}
\theta:
R\Gamma(E\sso,\Omega_{X\sso/\Delta}(\log E\sso) \otimes \xO_{E\sso})
\longrightarrow
A_{\bC}
\end{equation}
which preserves the filtrations $L$ and $F$.
Because
\begin{equation*}
H^i(\Gr_m^L\theta):
H^i(\Gr_m^L
R\Gamma(E\sso,\Omega_{X\sso/\Delta}(\log E\sso) \otimes \xO_{E\sso}))
\longrightarrow
H^i(\Gr_m^LA_{\bC})
\end{equation*}
coincides with the isomorphism
\begin{equation*}
H^i(E_{-m},\Omega_{X_{-m}/\Delta}(\log E_{-m}) \otimes \xO_{E_{-m}}))
\longrightarrow
H^i((E_{-m})_{\red}, A^{\bC}_{X_{-m}/\Delta})
\end{equation*}
in \eqref{isomorphism induced by theta}
for every $i,m$,
the morphism $\theta$ is a filtered quasi-isomorphism with respect to $L$.
In particular,
$\theta$ is a quasi-isomorphism,
that is,
\begin{equation}
\label{isomorphism Hi(theta)}
H^i(\theta):
H^i(E\sso,\Omega_{X\sso/\Delta}(\log E\sso) \otimes \xO_{E\sso})
\longrightarrow
H^i(A_{\bC})
\end{equation}
is an isomorphism for every $i$.

Moreover, the morphism
\begin{equation*}
\nu_{X\sso/\Delta}:
(A^{\bC}_{X\sso/\Delta},W,F)
\longrightarrow
(A^{\bC}_{X\sso/\Delta},W[-2],F[-1])
\end{equation*}
induces the morphism
\begin{equation}
\label{morphism nu}
(A_{\bC},L,\delta(W,L),F)
\longrightarrow
(A_{\bC},L,\delta(W,L)[-2],F[-1])
\end{equation}
which we simply denote by $\nu$.
Because of the property
\begin{equation*}
\nu(\delta(W,L)_mA_{\bC}) \subset \delta(W,L)_{m-2}A_{\bC}
\end{equation*}
$\nu$ is a nilpotent endomorphism.
We set
\begin{equation*}
N=H^i(\nu):
H^i(A_{\bC}) \longrightarrow H^i(A_{\bC})
\end{equation*}
for every $i$.

On the other hand,
we obtain
\begin{equation*}
(R\Gamma(E_{\bullet\red},B_{X\sso/\Delta}),G)
\end{equation*}
from $B_{X\sso/\Delta}$ with the filtration $G$.
By definition, we have
\begin{equation*}
R\Gamma(E_{\bullet\red},B_{X\sso/\Delta})^n
=A_{\bC}^{n-1} \oplus A_{\bC}^n
\end{equation*}
and
\begin{equation*}
d(x,y)=(-dx-\nu(y),dy)
\end{equation*}
for $x \in A_{\bC}^{n-1}$ and $y \in A_{\bC}^n$,
where $d$ is the differential of the complex $A_{\bC}$.
Moreover,
the filtration $G$
on $R\Gamma(E_{\bullet\red},B_{X\sso/\Delta})$ satisfies
\begin{equation*}
G^1R\Gamma(E_{\bullet\red},B_{X\sso/\Delta})
=A_{\bC}[-1]
\end{equation*}
as in \ref{review on Steenbrink's results}.
The morphism
\begin{equation}
\label{morphism eta sso}
\eta_{X\sso/\Delta}:
\Omega_{X\sso}(\log E\sso) \otimes \xO_{E\sso}
\longrightarrow
B_{X\sso/\Delta}
\end{equation}
induced by \eqref{morphism eta}
gives us the morphism
\begin{equation*}
\eta:
R\Gamma(E\sso,\Omega_{X\sso}(\log E\sso) \otimes \xO_{E\sso})
\longrightarrow
R\Gamma(E_{\bullet\red},B_{X\sso/\Delta})
\end{equation*}
preserving the filtration $G$.
Note that the diagrams
\begin{equation*}
\begin{CD}
\Gr_G^0R\Gamma(E\sso,\Omega_{X\sso}(\log E\sso) \otimes \xO_{E\sso})
@>{\Gr_G^0\eta}>>
\Gr_G^0B_{X\sso/\Delta} \\
@| @| \\
R\Gamma(E\sso,\Omega_{X\sso/\Delta}(\log E\sso) \otimes \xO_{E\sso})
@>>{\theta}>
A_{\bC}
\end{CD}
\end{equation*}
and
\begin{equation*}
\begin{CD}
\Gr_G^1R\Gamma(E\sso,\Omega_{X\sso}(\log E\sso) \otimes \xO_{E\sso})
@>{\Gr_G^1\eta}>>
\Gr_G^1B_{X\sso/\Delta} \\
@| @| \\
R\Gamma(E\sso,\Omega_{X\sso/\Delta}(\log E\sso) \otimes \xO_{E\sso})[-1]
@>>{\theta[-1]}>
A_{\bC}[-1]
\end{CD}
\end{equation*}
are commutative by definition.
Thus the morphism of $E_1$-terms
\eqref{morphism of E1-terms for Omega otimes OD}
coincides with $-N$
under the isomorphism
\eqref{isomorphism Hi(theta)}.
Therefore $\res_0(\nabla)$ is identified with $-N$
via the isomorphisms
\eqref{the canonical morphism from X bullet to D bullet}
and \eqref{isomorphism Hi(theta)}.
Because $\nu$ is nilpotent,
the morphism $N$ is nilpotent,
and then so is $\res_0(\nabla)$.
Thus we conclude that
$R^if_{\ast}\Omega_{X\sso/\Delta}(\log E\sso)$
is the canonical extension of
$R^if_{\ast}\Omega_{\zo{X}\sso/\pd}
\simeq \xO_{\pd} \otimes R^if_{\ast}\bQ_{\zo{X}\sso}$.
\end{step}

\begin{step}
We can easily see that
the data \eqref{fmhc A}
is a $\bQ$-mixed Hodge complex filtered by $L$
(for the definition of filtered $\bQ$-mixed Hodge complex,
see e.g.~\cite[6.1.4 D\'efinition]{elzein2}).
Moreover, the spectral sequence
associated to the filtration $L$
degenerates at $E_2$-terms
because
$\theta$ in 
\eqref{morphism RGamma(theta)} is a filtered quasi-isomorphism
and because
the spectral sequence for
$(R\Gamma(E\sso,\Omega_{X\sso/\Delta}(\log E\sso) \otimes \xO_{E\sso}),L)$
degenerates at $E_2$-terms.
The filtration induced by $\delta(W,L)$ on
\begin{equation*}
\Gr_m^LA_{\bC}
\simeq
R\Gamma((E_{-m})_{\red},A^{\bC}_{X_{-m}/\Delta})
\end{equation*}
coincides with $W[m]$
and the morphism
$\Gr_m^L\nu$ coincides with $\nu_{X_{-m}/\Delta}$
for every $m$.
Hence the filtration $\delta(W,L)[-m]$
on $H^i(\Gr_m^LA_{\bC})$ is the monodromy weight filtration
by the isomorphism
\eqref{coincidence of W and monodromy weight filtration}.
Therefore $\delta(W,L)$ on $H^i(A_{\bC})$
is the monodromy weight filtration of $N=H^i(\nu)$
relative to the filtration $L$
by Lemma \ref{lemma on relative monodromy filtration}.
\end{step}

Thus the condition
\eqref{extendability of Hodge filtration}
is obtained by the local freeness in Step 1
and by the fact that
$R^if_{\bullet\ast}\Omega_{X\sso/\Delta}(\log E\sso)$
is the canonical extension of
$R^if_{\bullet\ast}\Omega_{\zo{X}\sso/\pd}
\simeq \xO_{\pd} \otimes R^if_{\bullet\ast}\bQ_{\zo{X}\sso}$
for every $i$ in Step 3.
Moreover,
the condition
\eqref{existence of the relative monodromy weight filtration}
is proved
by the existence of the monodromy weight filtration of $N$
relative to $L$ in Step 5
and by the fact that $N$ coincides with $-\res_0(\nabla)$ in Step 3.
\end{proof}

\begin{rem} 
Let $(V,W)$ be a finite dimensional $\bQ$-vector space
equipped with a finite increasing filtration $W$
and $N$ a nilpotent endomorphism of $V$
preserving the filtration $W$.
On the $\bC$-vector space $V_{\bC}=\bC \otimes V$,
the filtration $W$ and the nilpotent endomorphism $N_{\bC}=\id \otimes N$
are induced by the trivial way.
Then the existence of the monodromy weight filtration
of $N$ relative to $W$ on $V$
is equivalent to the existence of the monodromy weight filtration of $N_{\bC}$
relative to $W$ on $V_{\bC}$.
We can check this equivalence 
by using Theorem (2.20) in \cite{sz}.
\end{rem}

\begin{lem}[GPVMHS for relative cohomology]
\label{admissibility for relative cohomology}
Let $f:X\sso \longrightarrow Y$
and $g:Z\sso \longrightarrow Y$ be
projective augmented strict semi-simplicial varieties
and $\varphi:Z\sso \longrightarrow X\sso$
a morphism of semi-simplicial varieties
compatible with the augmentations
$X\sso \longrightarrow Y$ and $Z\sso \longrightarrow Y$.
The cone of the canonical morphism
$\varphi^{-1}: Rf_{\ast}\bQ_{X\sso} \longrightarrow Rg_{\ast}\bQ_{Z\sso}$
is denoted by $C(\varphi^{-1})$
as in Lemma $\ref{lemma for relative cohomology}$.
Take the open subset $\zo{Y}$ such that
$f:X\sso \longrightarrow Y$ and $g:Z\sso \longrightarrow Y$
are smooth over $\zo{Y}$.
Then $H^i(C(\varphi^{-1}))|_{\zo{Y}}$ underlies
an admissible \gpv of $\bQ$-mixed Hodge structure for every $i$.
\end{lem}
\begin{proof}
By Lemma \ref{lemma for relative cohomology},
$H^i(C(\varphi^{-1}))|_{\zo{Y}}$ is a \gpv
of $\bQ$-mixed Hodge structure.
We will prove the admissibility of it.

As in Lemma \ref{GPVMHS for semi-simplicial variety},
we may assume the following:
\begin{itemize}
\item
$Y=\Delta$, $\zo{Y}=\Delta\!^*$.
\item
$f_p:X_p \longrightarrow \Delta$
and $g_q:Z_q \longrightarrow \Delta$
are of unipotent monodromy
for all $p,q$.
\item
$f^{-1}(0)_{\red}$
and $g^{-1}(0)_{\red}$
are simple normal crossing divisors
on $X\sso$ and $Z\sso$ respectively.
\end{itemize}
We set $E\sso=f^{-1}(0)$
and $F\sso=g^{-1}(0)$.
The morphism $\varphi:Z\sso \longrightarrow X\sso$
induces the morphism of {\itshape complexes}
\begin{equation*}
\varphi^{\ast}:
Rf_{\ast}\Omega_{X\sso/\Delta}(\log E\sso)
\longrightarrow
Rg_{\ast}\Omega_{Z\sso/\Delta}(\log F\sso)
\end{equation*}
as in the proof of Lemma \ref{lemma for relative cohomology}.
Then we consider
$C(\varphi^{\ast})$
equipped with filtrations $L$ and $F$
defined by the same way as
\eqref{filtrations L and F on the cone}
from the filtrations $L$ and $F$
on the complexes
$Rf_{\ast}\Omega_{X\sso/\Delta}(\log E\sso)$
and $Rg_{\ast}\Omega_{Z\sso/\Delta}(\log F\sso)$.
Then $C(\varphi^{\ast})|_{\Delta\!^*}$
induces the mixed Hodge structure on $H^i(C(\varphi^{-1}))$
as in the proof of Lemma \ref{lemma for relative cohomology}.

Because we have
\begin{equation*}
\begin{split}
(E_1^{p,q}(&C(\varphi^{\ast}),L),F) \\
&=
(R^q(f_{p+1})_{\ast}\Omega_{X_{p+1}/\Delta}(\log E_p),F)
\oplus
(R^q(g_p)_{\ast}\Omega_{Z_p/\Delta}(\log F_p),F)
\end{split}
\end{equation*}
as in the proof of Lemma \ref{lemma for relative cohomology},
the same argument as Step 1 of the proof of
Lemma \ref{GPVMHS for semi-simplicial variety}
shows that the spectral sequence
$E_r^{p,q}(C(\varphi^{\ast}),L)$ degenerates at $E_2$-terms
and that
\begin{equation*}
\Gr_{F_{\rec}}^rE_2^{p,q}(C(\varphi^{\ast}),L)
\simeq
\Gr_F^r\Gr_{-p}^LH^{p+q}(C(\varphi^{\ast}))
\end{equation*}
are locally free coherent $\xO_{\Delta}$-modules
for all $p,q,r$.
Moreover, the spectral sequence
$E_r^{p,q}(C(\varphi^{\ast}),F)$ degenerates at $E_1$-terms
by the lemma on two filtrations as usual.

Let $A^{\bC}_{X\sso/\Delta}$
and $A^{\bC}_{Z\sso/\Delta}$ be the complexes
defined in Step 4
of the proof of Lemma \ref{GPVMHS for semi-simplicial variety}.
The morphism $\varphi$ induces a morphism
of trifiltered {\itshape complexes}
\begin{equation*}
\begin{split}
(R\Gamma(E_{\bullet\red},A^{\bC}_{X\sso/\Delta}),&L,\delta(W,L),F) \\
&\longrightarrow
(R\Gamma(F_{\bullet\red},A^{\bC}_{Z\sso/\Delta}),L,\delta(W,L),F)
\end{split}
\end{equation*}
by using the Godment resolution
as in \ref{Godment resolution}.
This morphism of {\itshape complexes}
is denoted by $\psi$ for a while.
On the complex $C(\psi)$,
the filtrations $L$ and $F$
are defined by the same way
as in \eqref{filtrations L and F on the cone} 
and the filtration $\delta(W,L)$
by the same way as $L$.
We can easily check that $(C(\psi),L,\delta(W,L),F)$
underlies a filtered $\bQ$-mixed Hodge complex.
The composites of the morphisms
\eqref{base change morphism for X bullet}
and \eqref{morphism RGamma(theta)}
for $X\sso$ and $Z\sso$
fit in the commutative diagram
\begin{equation*}
\begin{CD}
Rf_{\ast}\Omega_{X\sso/\Delta}(\log E\sso) \otimes \bC(0)
@>>> R\Gamma(E_{\bullet\red},A^{\bC}_{X\sso/\Delta}) \\
@V{\varphi^{\ast} \otimes \id}VV @VV{\psi}V \\
Rg_{\ast}\Omega_{Z\sso/\Delta}(\log F\sso) \otimes \bC(0)
@>>>
R\Gamma(F_{\bullet\red}, A^{\bC}_{Z\sso/\Delta})
\end{CD}
\end{equation*}
from which the morphism of complexes
$C(\varphi^{\ast}) \otimes \bC(0) \longrightarrow C(\psi)$
preserving the filtrations $L$ and $F$
is obtained.
This morphism induces an isomorphism
$E_1^{p,q}(C(\varphi^{\ast}),L) \otimes \bC(0)
\longrightarrow E_1^{p,q}(C(\psi),L)$
because the morphism between $E_1$-terms
coincides with the direct sum
of the isomorphisms
\eqref{the canonical morphism from X bullet to D bullet}
for $X\sso$ and $Z\sso$.
Then the canonical morphism
$E_2^{p,q}(C(\varphi^{\ast}),L) \otimes \bC(0)
\longrightarrow E_2^{p,q}(C(\psi),L)$
is an isomorphism
because
the local freeness of $E_2^{p,q}(C(\varphi^{\ast}),L)$
implies that the canonical morphism
\begin{equation*}
E_2^{p,q}(C(\varphi^{\ast}),L) \otimes \bC(0)
\longrightarrow
H^p(E_1^{\bullet,q}(C(\varphi^{\ast}),L) \otimes \bC(0))
\end{equation*}
is an isomorphism.
Therefore the spectral sequence
$E_r^{p,q}(C(\psi),L)$ degenerates at $E_2$-terms
and the canonical morphism
\begin{equation*}
H^i(C(\varphi^{\ast})) \otimes \bC(0) \longrightarrow H^i(C(\psi))
\end{equation*}
is an isomorphism for every $i$,
under which the filtration $L$ on the both sides coincides.

The morphisms
\eqref{morphism nu} for $X\sso$ and $Z\sso$
induce the morphism of complexes
\begin{equation*}
(C(\psi),L,\delta(W,L),F)
\longrightarrow
(C(\psi),L,\delta(W,L)[-2],F[-1])
\end{equation*}
which is denoted by $\nu$ again.

By using the mapping cone of the canonical morphism
\begin{equation*}
(Rf_{\ast}\Omega_{X\sso}(\log E\sso),G)
\longrightarrow
(Rg_{\ast}\Omega_{Z\sso}(\log F\sso),G)
\end{equation*}
with the decreasing filtrations $G$
defined by the same way
as $F$ in \eqref{filtrations L and F on the cone},
we obtain the log integrable connection
$\nabla$ on $H^i(C(\varphi^{\ast}))$ for every $i$
by the same way as in the proof of
Lemma \ref{lemma for relative cohomology}.
By definition,
the restriction of this $\nabla$ on $\Delta\!^{\ast}$
coincides with the original $\nabla$
on $H^i(C(\varphi^{\ast}))|_{\Delta\!^{\ast}}$.
Similarly, the morphism $\varphi:Z\sso \longrightarrow X\sso$
induces a morphism of filtered complexes
$(B_{X\sso/\Delta},G) \longrightarrow (B_{Z\sso/\Delta},G)$
such that the diagram 
\begin{equation*}
\begin{CD}
\Omega_{X\sso}(\log E\sso)
@>>>
\Omega_{Z\sso}(\log F\sso) \\
@VVV @VVV \\
B_{X\sso/\Delta} @>>> B_{Z\sso/\Delta}
\end{CD}
\end{equation*}
is commutative.
By considering the cone of the morphism
\begin{equation*}
R\Gamma(E_{\sso\red},B_{X\sso/\Delta})
\longrightarrow
R\Gamma(F_{\sso\red}, B_{Z\sso/\Delta})
\end{equation*}
with the decreasing filtration $G$,
the residue
$\res_0(\nabla)$ on $H^i(C(\varphi^{\ast})) \otimes \bC(0)$
is identified with $-H^i(\nu)$
for every $i$.
Because the morphism $\nu$ is trivially nilpotent,
we conclude that $\res_0(\nabla)$ is nilpotent.
Therefore $H^i(C(\varphi^{\ast}))$ is the canonical extension
of $H^i(C(\varphi^{\ast}))|_{\Delta\!^{\ast}}$.
Thus the condition \eqref{extendability of Hodge filtration}
is satisfied by $H^i(C(\varphi^{\ast}))$ for every $i$.
Moreover,
we can easily see that the filtration $\delta(W,L)[-m]$
on $H^i(\Gr_m^LC(\psi))$
is the monodromy weight filtration
of $\Gr_m^LH^i(\nu)$ for every $i, m$.
Therefore Lemma \ref{lemma on relative monodromy filtration}
implies that the filtration $\delta(W,L)$
is the monodromy weight filtration
of $H^i(\nu)$ relative to $L$
on $H^i(C(\psi)) \simeq H^i(C(\varphi^{\ast})) \otimes \bC(0)$.
\end{proof}

\begin{thm}[GPVMHS for cohomology with compact support]
\label{GPVMHS for compact support cohomology}
Let $f:X \longrightarrow Y$ be a projective surjective morphism
from a complex variety $X$
onto a smooth complex variety $Y$
and $Z$ a closed subset of $X$.
Then there exists a Zariski open dense subset $\zo{Y}$ of $Y$
such that
$(R^i(f|_{X \setminus Z})_!\bQ_{X \setminus Z})|_{\zo{Y}}$
underlies an admissible \gpv of $\bQ$-mixed Hodge structure for every $i$.
\end{thm}
\begin{proof}
The open immersion $X \setminus Z \longrightarrow X$
and the closed immersion $Z \longrightarrow X$
are denoted by $\iota$ and $j$ respectively.
We set $g=fj:Z \longrightarrow Y$.
Take cubical hyperresolutions
$\varepsilon_Z:Z\sso \longrightarrow Z$
and $\varepsilon_X:X\sso \longrightarrow X$
which fits in the commutative diagram
\begin{equation*}
\begin{CD}
Z\sso @>{\varphi}>> X\sso \\
@V{\varepsilon_Z}VV @VV{\varepsilon_X}V \\
Z @>>{j}> X
\end{CD}
\end{equation*}
for some morphism $\varphi:Z\sso \longrightarrow X\sso$
of cubical varieties.
The cone of the canonical morphism
$\varphi^{-1}:R(f\varepsilon_X)_{\ast}\bQ_{X\sso}
\longrightarrow R(g\varepsilon_Z)_{\ast}\bQ_{Z\sso}$
is denoted by $C(\varphi^{-1})$
as in Lemma \ref{admissibility for relative cohomology}.
Then the composite of the canonical morphisms
\begin{equation*}
R(f|_{X \setminus Z})_!\bQ_{X \setminus Z}
\simeq
Rf_{\ast}\iota_!\bQ_{X \setminus Z}
\longrightarrow
Rf_{\ast}\bQ_X
\longrightarrow
R(f\varepsilon_X)_{\ast}\bQ_{X\sso}
\end{equation*}
induces the quasi-isomorphism
$R(f|_{X \setminus Z})_!\bQ_{X \setminus Z} \longrightarrow C(\varphi^{-1})[-1]$
from which we obtain the conclusion
by considering the filtration $L[-1]$ on $C(\varphi^{-1})[-1]$.
\end{proof}

\begin{say} 
\label{preliminaries for snc pair} 
Let $(X,D)$ be a \snc pair with $D$ reduced.
The irreducible decompositions of $X$ and $D$ are given by
\begin{equation*}
X=\bigcup_{i \in I}X_i, \quad
D=\bigcup_{\lambda \in \Lambda}D_{\lambda}
\end{equation*}
respectively.
Fixing orders $<$ on $\Lambda$ and $I$,
we set
\begin{equation*}
D_k \cap X_l
=\coprod_{\substack{\lambda_0< \lambda_1 < \cdots < \lambda_k \\
               i_0 < i_1 < \cdots < i_l}}
D_{\lambda_0} \cap D_{\lambda_1} \cap \cdots \cap D_{\lambda_k}
\cap X_{i_0} \cap X_{i_1} \cap \cdots \cap X_{i_l}
\end{equation*}
for $k,l \ge 0$. 
Here we use the convention 
\begin{align*}
&D_k=D_k \cap X_{-1}
=\coprod_{\lambda_0 < \lambda_1 < \cdots < \lambda_k}
D_{\lambda_0} \cap D_{\lambda_1} \cap \cdots \cap D_{\lambda_k} \\
&X_l=D_{-1} \cap X_l
=\coprod_{i_0 < i_1 < \cdots < i_l}
X_{i_0} \cap X_{i_1} \cap \cdots \cap X_{i_l}
\end{align*}
for $k,l \ge 0$. 
Moreover, we set
\begin{equation*}
(D \cap X)_n
=\coprod_{k+l+1=n}D_k \cap X_l
\end{equation*}
for $n \ge 0$.
Thus we obtain 
projective augmented strict semi-simplicial varieties
$\varepsilon_X: (D \cap X)\sso \longrightarrow X$,
$\varepsilon_D: D_{\bullet} \longrightarrow D$
and the morphism of semi-simplicial varieties
$\varphi: D_{\bullet} \longrightarrow (D \cap X)\sso$
compatible with the augmentation $\varepsilon_X$ and $\varepsilon_D$.
We remark that $D_k \cap X_l$ are smooth for all $k,l$
by the definition of \snc pair.
Therefore $\varepsilon_D: D_{\bullet} \longrightarrow D$
is a hyperresolution of $D$.
Now we will see that $\varepsilon_X:(D \cap X)\sso \longrightarrow X$
is also a hyperresolution of $X$.
It is sufficient to prove that $\varepsilon_X$ is of cohomological descent.
The cone of the canonical morphism
\begin{equation*}
\bQ_X
\longrightarrow
(\varepsilon_X)_{\ast}\bQ_{(D \cap X)\sso}
=R(\varepsilon_X)_{\ast}\bQ_{(D \cap X)\sso}
\end{equation*}
is the single complex associated to the double complex
\begin{equation*}
\begin{CD}
@. 0 @. 0 @. 0 @. \\
@. @VVV @VVV @VVV \\
0 @>>> \bQ_X @>>> \bQ_{X_0}
@>>> \bQ_{X_1}
@>>> \cdots \\
@. @VVV @VVV @VVV \\
0 @>>> \bQ_{D_0} @>>> \bQ_{D_0 \cap X_0}
@>>> \bQ_{D_0 \cap X_1}
@>>> \cdots \\
@. @VVV @VVV @VVV \\
0 @>>> \bQ_{D_1} @>>> \bQ_{D_1 \cap X_0}
@>>> \bQ_{D_1 \cap X_1}
@>>> \cdots \\
@. @VVV @VVV @VVV \\
@. \vdots @. \vdots @. \vdots
\end{CD}
\end{equation*}
shifted by $1$. 
All the lines of the diagram above are exact
because they are the Mayer--Vietoris exact sequences
for $X$ and for $D_k$.
Then the single complex associated to the double complex above
is acyclic.
Thus we obtain that the canonical morphism 
$\bQ_X \longrightarrow (\varepsilon_X)_{\ast}\bQ_{(D \cap X)\sso}$ 
is a quasi-isomorphism.
\end{say}

\begin{thm}[GPVMHS for snc pair]
\label{GPVMHS for snc pair}
Let $(X,D)$ be a \snc pair with $D$ reduced
and $f:X \longrightarrow Y$ a projective surjective morphism
to a smooth algebraic variety $Y$.
Let $\zo{Y}$ be the non-empty Zariski open subset of $Y$ such that
all the strata of $(X,D)$ are smooth over $\zo{Y}$.
Then $(R^i(f|_{X \setminus D})_!\bQ_{X \setminus D})|_{\zo{Y}}$
underlies an admissible \gpv of $\bQ$-mixed Hodge structure for every $i$.
\end{thm}
\begin{proof}
As we mentioned in \ref{preliminaries for snc pair},
we have the commutative diagram
\begin{equation*}
\begin{CD}
D\sso @>{\varphi}>> (D \cap X)\sso \\
@V{\varepsilon_D}VV @VV{\varepsilon_X}V \\
D @>>> X
\end{CD}
\end{equation*}
such that $\varepsilon_D$ and $\varepsilon_X$ are hyperresolutions.
Then we obtain the conclusion by the same way
as the proof of Theorem \ref{GPVMHS for compact support cohomology}
from Lemma \ref{admissibility for relative cohomology}.
\end{proof}

\begin{rem}
\label{remark for GrF0}
In the situation above,
the inverse images of the open subset $\zo{Y}$
are indicated by the superscript $^{\ast}$
such as $\zo{X}=f^{-1}(\zo{Y})$.
From the proof of Lemma \ref{lemma for relative cohomology},
we can check that
$\Gr_F^p(\xO_{\zo{Y}}
\otimes (R^i(f|_{X \setminus D})_!\bQ_{X \setminus D})|_{\zo{Y}})$
coincides with the $(i-p)$-th direct image
of the single complex associated to the double complex
\begin{equation*}
\begin{CD}
@. 0 @. 0 @. 0 @. \\
@. @VVV @VVV @VVV \\
0 @>>> \Omega^p_{\zo{X}_0/\zo{Y}} @>>> \Omega^p_{\zo{X}_1/\zo{Y}}
@>>> \Omega^p_{\zo{X}_2/\zo{Y}} @>>> \cdots \\
@. @VVV @VVV @VVV \\
0 @>>> \Omega^p_{\zo{D}_0 \cap \zo{X}_0/\zo{Y}}
@>>> \Omega^p_{\zo{D}_0 \cap \zo{X}_1/\zo{Y}}
@>>> \Omega^p_{\zo{D}_0 \cap \zo{X}_2/\zo{Y}} @>>> \cdots \\
@. @VVV @VVV @VVV \\
0 @>>> \Omega^p_{\zo{D}_1 \cap \zo{X}_0/\zo{Y}}
@>>> \Omega^p_{\zo{D}_1 \cap \zo{X}_1/\zo{Y}}
@>>> \Omega^p_{\zo{D}_1 \cap \zo{X}_2/\zo{Y}} @>>> \cdots \\
@. @VVV @VVV @VVV \\
@. \vdots @. \vdots @. \vdots
\end{CD}
\end{equation*}
by $f$.
Therefore we have the canonical isomorphism
\begin{equation*}
R^if_{\ast}\xO_X(-D)
\simeq
\Gr_F^0(\xO_{\zo{Y}}
\otimes (R^i(f|_{X \setminus D})_!\bQ_{X \setminus D})|_{\zo{Y}})
\end{equation*}
for every $i$.
\end{rem}

\section{Semipositivity theorem}\label{sub42}

In this section, 
we discuss a purely Hodge theoretic 
aspect of the Fujita--Kawamata semipositivity theorem 
(cf.~\cite{zucker1} and \cite[\S 4 Semi-positivity]{kawamata1}). 
Our formulation is different from
Kawamata's original one but is indispensable 
for our main theorem:~Theorem \ref{main} (4). 
For related topics, see \cite[Section 5]{mori}, 
\cite[Section 5]{fujino-rem}, 
\cite[3.2.~Semi-positivity theorem]{high}, 
and \cite[8.10]{ko-ko}. 
We use the theory of logarithmic integrable 
connections. For the basic properties and results on logarithmic integrable 
connections, see \cite{Deligne}, \cite{katz2}, and \cite[IV.~Regular 
connections, after Deligne]{borel} by Bernard Malgrange. 
For a different approach, see \cite[Section 4]{ffs}. 

We start with easy observations.

\begin{lem}
\label{lemma 5.1}
Let $X$ be a complex manifold,
$U$ a dense open subset of $X$
and $\xV$ a locally free $\xO_X$-module of finite rank.
Assume that two $\xO_X$-submodules $\xF$ and $\xG$
satisfy the following conditions:
\begin{newitemize}
\itemno
$\xG$ and $\xV/\xG$ are locally free $\xO_X$-modules of finite rank.
\itemno
$\xF|_{U}=\xG|_{U}$.
\end{newitemize}
Then we have the inclusion $\xF \subset \xG$ on $X$.
\end{lem}

\begin{cor}
\label{corollary for the extensions of a filtrartion}
Let $X$, $U$ and $\xV$ be as above.
Two finite decreasing filtrations $F$ and $G$ on $\xV$
satisfy the following conditions:
\begin{itemize}
\item
$\Gr_G^p\xV$ is a locally free $\xO_X$-module of finite rank
for every $p$.
\item
$F^p\xV|_{U}=G^p\xV|_{U}$ for every $p$.
\end{itemize}
Then we have $F^p\xV \subset G^p\xV$ on $X$ for every $p$.
In particular, $F^p\xV=G^p\xV$ for every $p$,
if, in addition, $\Gr_F^p\xV$ is locally free of finite rank for every $p$.
\end{cor}

\begin{say}
\label{setting for section 5}
Let $X$ be a complex manifold
and $D=\sum_{i \in I}D_i$ a \snc divisor on $X$,
where $D_i$ is a smooth irreducible divisor on $X$ for every $i \in I$.
We set
\begin{equation*}
D(J)=\bigcap_{i \in J}D_i, \quad
D_J=\sum_{i \in J}D_i
\end{equation*}
for any subset $J$.
Note that $D(\emptyset)=X$ and $D_{\emptyset}=0$ by definition. 
Moreover we set
$D(J)^*=D(J) \setminus D(J) \cap D_{I \setminus J}$
for $J \subset I$.
For the case of $J=\emptyset$, we set $X^*=D(\emptyset)^*=X \setminus D$.

Let $\xV$ be a locally free $\xO_X$-module of finite rank and
\begin{equation*}
\nabla: \xV \longrightarrow \Omega_X^1(\log D) \otimes \xV
\end{equation*}
a logarithmic integrable connection on $\xV$. 
The residue of $\nabla$ along $D_i$ is denoted by 
\begin{equation*}
\res_{D_i}(\nabla):
\xO_{D_i} \otimes \xV \longrightarrow \xO_{D_i} \otimes \xV. 
\end{equation*}
We assume the following condition
throughout this section:
\begin{newitemize}
\itemno
$\res_{D_i}(\nabla):
\xO_{D_i} \otimes \xV \longrightarrow \xO_{D_i} \otimes \xV$
is nilpotent for every $i \in I$.
\end{newitemize}
This is equivalent to the condition that
the local system
$\Ker(\nabla)|_{X^*}$ is of unipotent local monodromy.
\end{say}

\begin{say}
In the situation above, 
the morphism
\begin{equation*}
\id \otimes \res_{D_i}(\nabla):
\xO_{D(J)} \otimes \xV \longrightarrow \xO_{D(J)} \otimes \xV
\end{equation*} 
is denoted 
by $N_{i,D(J)}$ for a subset $J$ of $I$
and for $i \in J$.
We simply write $N_i$ if there is no danger of confusion.
We have
\begin{equation*}
N_{i,D(J)}N_{j,D(J)}=N_{j,D(J)}N_{i,D(J)}
\end{equation*}
for every $i,j \in J$.
For two subsets $J,K$ of $I$ with $K \subset J$,
we set $N_{K,D(J)}=\sum_{i \in K}N_{i,D(J)}$,
which is nilpotent by the assumption above.
Once a subset $J$ is fixed,
we use the symbols $N_K$ for short.
We have the monodromy weight filtration $W(K)$ on $\xO_{D(J)} \otimes \xV$
which is characterized by the condition that
$N_K^q$ induces an isomorphism
\begin{equation*}
\Gr_q^{W(K)}(\xO_{D(J)} \otimes \xV)
\overset{\simeq}{\longrightarrow}
\Gr_{-q}^{W(K)}(\xO_{D(J)} \otimes \xV)
\end{equation*}
for all $q \ge 0$. 
For $K=\emptyset$,
$W(\emptyset)$ is trivial, that is,
$W(\emptyset)_{-1}\xO_{D(J)} \otimes \xV=0$
and $W(\emptyset)_0\xO_{D(J)} \otimes \xV=\xO_{D(J)} \otimes \xV$. 

For the case of $J=K$,
we set
\begin{equation*}
\xP_k(J)=\Ker(N_J^{k+1}:
\Gr_k^{W(J)}(\xO_{D(J)} \otimes \xV)
\longrightarrow \Gr_{-k-2}^{W(J)}(\xO_{D(J)} \otimes \xV))
\end{equation*}
for every non-negative integer $k$,
which is called the primitive part of $\Gr_k^{W(J)}(\xO_{D(J)} \otimes \xV)$
with respect to $N_J$.
Then we have the primitive decomposition
\begin{equation*}
\Gr_k^{W(J)}(\xO_{D(J)} \otimes \xV)
= \bigoplus_{l \ge \max(0, -k)}N_J^l(\xP_{k+2l}(J))
\end{equation*}
for every $k$,
and $N_J^l$ induces an isomorphism
\begin{equation*}
\xP_{k+2l}(J) \longrightarrow N_J^l(\xP_{k+2l}(J))
\end{equation*}
for every $k, l$ with $\ge \max(0,-k)$.
\end{say}

\begin{lem}
\label{local freeness of grW(K)}
In the situation above,
$\Gr_k^{W(K)}(\xO_{D(J)} \otimes \xV)$
is a locally free $\xO_{D(J)}$-module of finite rank
for every $k$
and for every subsets $J, K$ of $I$ with $K \subset J$.
\end{lem}
\begin{proof}
Easy by the local description of logarithmic integrable connection 
(see e.g.~Deligne \cite{Deligne}, Katz \cite{katz2}). 
\end{proof}

\begin{cor}
In the situation above,
we fix a subset $J$ of $I$.
For any subset $K$ of $J$
we have the equality
\begin{equation*}
W(K)=W(N_{K}(x))
\end{equation*}
on $\xV(x)=\xV \otimes \bC(x)$
for every point $x \in D(J)$,
where the left hand side denotes the filtration on $\xV(x)$
induced by $W(K)$.
\end{cor}

\begin{rem}
\label{functoriality of W(K)}
Let $(\xV_1,\nabla_1)$ and $(\xV_2,\nabla_2)$ be
pairs of locally free sheaves of $\xO_X$-modules of finite rank
and integrable logarithmic connections on them.
We assume that they satisfy the condition in \ref{setting for section 5}.
If the morphism $\varphi:\xV_1 \longrightarrow \xV_2$
of $\xO_X$-modules
is compatible with the connections $\nabla_1$ and $\nabla_2$,
then the diagram
\begin{equation*}
\begin{CD}
\xO_{D(J)} \otimes \xV_1 @>{\id \otimes \varphi}>> \xO_{D(J)} \otimes \xV_2 \\
@V{N_{i,D(J)}}VV @VV{N_{i,D(J)}}V \\
\xO_{D(J)} \otimes \xV_1 @>>{\id \otimes \varphi}> \xO_{D(J)} \otimes \xV_2
\end{CD}
\end{equation*}
is commutative for every subset $J$ of $I$
and for every $i \in J$.
Therefore $\id \otimes \varphi$
preserves the filtration $W(K)$ for every $K \subset J$.
\end{rem}

\begin{say}
\label{condition mh}
Let $m$ be an integer.
For a finite decreasing filtration $F$ on $\xV$,
we consider the following condition:
\begin{itemize}
\item[($m$MH)]
The triple
\begin{equation*}
(\xV(x), W(J)[m], F)
\end{equation*}
underlies an $\bR$-mixed Hodge structure
for any subset $J$ of $I$
and for any point $x \in D(J)^*$.
\end{itemize}
Here we remark that we do not assume the local freeness
of $\Gr_F^p\xV$ at the beginning.
\end{say}

The following lemma is the counterpart of
Schmid's results in \cite{schmid}.

\begin{lem}
Let $U$ be an open subset of $X \setminus D$,
such that $X \setminus U$ is nowhere dense analytic subspace of $X$.
Moreover, we are given a finite decreasing filtration $F$ on $\xV|_U$.
If $(\xV, F, \nabla)|_U$ underlies
a \pv of $\bR$-Hodge structure of weight $m$ on $U$,
then there exists a finite decreasing filtration $\widetilde{F}$ on $\xV$
satisfying the following three conditions:
\begin{enumerate}
\item
$\widetilde{F}^p\xV|_U=F^p\xV|_U$ for every $p$.
\item
$\Gr_{\widetilde{F}}^p\xV$ is a locally free $\xO_X$-module
of finite rank for every $p$.
\item
$\widetilde{F}$ satisfies the condition $(m{\rm MH})$ in $\ref{condition mh}$.
\end{enumerate}
\end{lem}
\begin{proof}
See \cite{schmid}.
\end{proof}

\begin{lem}
\label{lemma for the conition mh}
Let $U$ be as above,
and $F$ a finite decreasing filtration on $\xV$
in the situation {\rm \ref{setting for section 5}}.
We assume that $(\xV,F,\nabla)|_U$ underlies
a \pv of $\bR$-Hodge structure of weight $m$ on $U$.
Then $\Gr_F^p\xV$ is locally free of finite rank for every $p$
if and only if $F$ satisfies the condition $(m{\rm MH})$ in $\ref{condition mh}$.
\end{lem}
\begin{proof}
By the lemma above,
there exists a finite decreasing filtration $\widetilde{F}$ on $\xV$
satisfying the three conditions above.
By Corollary \ref{corollary for the extensions of a filtrartion},
the local freeness of $\Gr_F^p\xV$ for every $p$
is equivalent to the equality $F^p\xV=\widetilde{F}^p\xV$ for every $p$.
If $F=\widetilde{F}$ on $\xV$,
$F$ satisfies the condition ($m$MH) by the lemma above.
Thus it suffices to prove the equality $F=\widetilde{F}$ on $\xV$
under the assumption that $F$ satisfies the condition ($m$MH).
By Corollary \ref{corollary for the extensions of a filtrartion} again,
we have $F^p\xV \subset \widetilde{F}^p\xV$ for every $p$.
On the other hand,
$(\xV(x),W(J)[m],F)$ and $(\xV(x),W(J)[m],\widetilde{F})$
are $\bR$-mixed Hodge structures
for every $x \in D(J)^*$,
if $F$ satisfies the condition ($m$MH).
Therefore we obtain $F(\xV(x))=\widetilde{F}(\xV(x))$ for every $x \in X$,
which implies the equality $F=\widetilde{F}$ on $\xV$.
\end{proof}

\begin{say}
\label{setting on filtrations}
In addition to the situation \ref{setting for section 5},
we assume that we are given a finite decreasing filtration $F$ on $\xV$
satisfying the following three conditions:
\begin{itemize}
\item
The Griffiths transversality holds, that is,
we have
$\nabla(F^p) \subset \Omega_X^1(\log E) \otimes F^{p-1}$
for every $p$.
\item
$(\xV, F, \nabla)|_{X^*}$ underlies
a \pv of $\bR$-Hodge structure of weight $m$.
\item
$\Gr_F^p\xV$ is locally free of finite rank for every $p$,
or equivalently, $F$ satisfies the condition ($m$MH).
\end{itemize}
For a subset $J$ of $I$,
the Griffiths transversality implies the condition
\begin{equation*}
N_{i}(F^p(\xO_{D(J)} \otimes \xV)) \subset F^{p-1}(\xO_{D(J)} \otimes \xV)
\end{equation*}
for every $p$
and for every $i \in J$.
\end{say}

\begin{lem}
\label{lemma by Cattani-Kaplan}
In the situation above,
we have
\begin{enumerate}
\item
$N_i(W(K)_k) \subset W(K)_{k-1}$ for every $i \in K$
and for every $k$,
\item
$W(J)$ is the monodromy weight filtration of $N_K$
relative to the filtration $W(J \setminus K)$
\end{enumerate}
on $\xO_{D(J)} \otimes \xV$
for every two subsets $J, K$ of $I$ with $K \subset J$.
\end{lem}
\begin{proof}
See Cattani--Kaplan \cite[(3.3) Theorem, (3,4)]{ck}
and Steenbrink--Zucker \cite[(3.12) Theorem]{sz}.
\end{proof}

\begin{cor}
\label{corollary for condition mh}
In the situation {\rm \ref{setting for section 5}}
and {\rm \ref{setting on filtrations}},
the induced filtration $F$ on $\Gr_k^{W(J)}(\xO_{D(J)} \otimes \xV)$
satisfies the property $((m+k){\rm MH})$
for any subset $J$ of $I$.
\end{cor}
\begin{proof}
Take a subset $K$ of $I \setminus J$.
For any point $x \in D(J \cup K)^*$,
the triple
\begin{equation*}
(\xV(x),W(J \cup K)[m],F)
\end{equation*}
underlies an $\bR$-mixed Hodge structure
because $F$ satisfies the condition ($m$MH) by the assumption.
Moreover, the morphism $(2\pi\sqrt{-1})^{-1}N_{J}(x)$
is a morphism of $\bR$-mixed Hodge structures of type $(-1,-1)$
by the condition (2) in the lemma above
and by the Griffiths transversality.
Therefore
\begin{equation*}
(\Gr_k^{W(J)}\xV(x),W(J \cup K)[m],F)
\end{equation*}
is an $\bR$-mixed Hodge structure.
On the other hand,
we have
\begin{equation*}
W(J \cup K)(\Gr_k^{W(J)}\xV(x))=W(K)(\Gr_k^{W(J)}\xV(x))[k] ,
\end{equation*}
by (2) in the lemma above.
Thus
\begin{equation*}
(\Gr_k^{W(J)}\xV(x),W(K)[m+k],F)
\end{equation*}
underlies an $\bR$-mixed Hodge structure.
\end{proof}

\begin{say}\label{say614}
In the situation \ref{setting for section 5}
and \ref{setting on filtrations}
we fix a subset $J$ of $I$.
We have an exact sequence
\begin{equation*}
\begin{CD}
0 @>>> \Omega_{D(J)}^1(\log D(J) \cap D_{I \setminus J})
@>>> \Omega_X^1(\log D) \otimes \xO_{D(J)} \\
@. @>>> \xO_{D(J)}^{\oplus |J|} @>>> 0 ,
\end{CD}
\end{equation*}
where $|J|$ denotes the cardinality of $J$.
On the other hand, the integrable log connection $\nabla$
induces a commutative diagram
\begin{equation*}
\begin{CD}
\xV @>{\nabla}>> \Omega_X^1(\log D) \otimes \xV \\
@VVV @VVV \\
\xO_{D(J)} \otimes \xV @>>> \xO_{D(J)}^{\oplus |J|} \otimes \xV,
\end{CD}
\end{equation*}
where the bottom horizontal arrow coincides with
$\bigoplus_{i \in J}N_{i,D(J)}$
under the identification
$\xO_{D(J)}^{\oplus |J|} \otimes \xV \simeq (\xO_{D(J)} \otimes \xV)^{|J|}$.
Because $\nabla$ preserves the filtration $W(J)$ on $\xO_{D(J)} \otimes \xV$
by the local description in \cite{Deligne}, \cite{katz2}
and because $N_{i,D(J)}(W(J)_k) \subset W(J)_{k-1}$ for every $k$
by Lemma \ref{lemma by Cattani-Kaplan} (1),
we obtain a morphism
\begin{equation*}
\Gr_k^{W(J)}(\xO_{D(J)} \otimes \xV)
\longrightarrow
\Omega_{D(J)}^1(\log D(J) \cap D_{I \setminus J})
\otimes \Gr_k^{W(J)}(\xO_{D(J)} \otimes \xV)
\end{equation*}
for every $k$. 
It is denoted by $\nabla_k(J)$, or simply $\nabla(J)$.
It is easy to see that $\nabla(J)$ is an integrable log connection
on $\Gr_k^{W(J)}(\xO_{D(J)} \otimes \xV)$
satisfying
$\nabla(J)(F^p) \subset F^{p-1}$ for every $p$
for the induced filtration $F$ on $\Gr_k^{W(J)}(\xO_{D(J)} \otimes \xV)$.
We can easily see that the residue of $\nabla(J)$
along $D(J) \cap D_i$
coincide with $N_{i,D(J \cup \{i\})}$  for $i \in I \setminus J$.
Thus $\nabla(J)$ satisfies the condition
in \ref{setting for section 5}.
\end{say}

\begin{say}
\label{setting on polarization}
Let $(\xV, F,\nabla)$ be as in \ref{setting for section 5}
and \ref{setting on filtrations}.
Then $(\xV,F,\nabla)|_{X^*}$ is a \pv of $\bR$-Hodge structure of weight $m$.
An integrable logarithmic connection on $\xV \otimes \xV$
is defined by $\nabla \otimes \id +\id \otimes \nabla$ as usual.
Assume that we are given a morphism
\begin{equation*}
S: \xV \otimes \xV \longrightarrow \xO_X
\end{equation*}
satisfying the following:
\begin{itemize}
\item
$S$ is $(-1)^m$-symmetric.
\item
$S$ is compatible with the connections,
where $\xO_X$ is equipped with the trivial connection $d$.
\item
$S(F^p\xV \otimes F^q\xV)=0$ if $p+q > m$.
\item
$S|_{X^*}$ underlies a polarization of the variation of $\bR$-Hodge structure
$(\xV,F,\nabla)|_{X^*}$.
\end{itemize}

Now we fix a subset $J$ of $I$.
Then $S$ induces a morphism
\begin{equation*}
\xO_{D(J)} \otimes \xV \otimes \xV
\simeq (\xO_{D(J)} \otimes \xV) \otimes (\xO_{D(J)} \otimes \xV)
\longrightarrow \xO_{D(J)} ,
\end{equation*}
which is denoted by $S_J$.
\end{say}

\begin{lem}
In the situation above,
we have
\begin{equation*}
S_J(W(K)_a \otimes W(K)_b)=0
\end{equation*}
for every $K \subset J$ and for every $a,b$ with $a+b < 0$.
\end{lem}
\begin{proof}
We fix a subset $K$ of $J$.
It is sufficient to prove that
\begin{equation*}
S_J(W(K)_a \otimes W(K)_{-a-1})=0
\end{equation*}
for every non-negative integer $a$.

Since $S$ is compatible with the connections,
we have
\begin{equation*}
S_J \cdot (N_i \otimes \id+\id \otimes N_i)=0
\end{equation*}
for every $i \in J$, from which the equality
\begin{equation*}
S_J \cdot (N_K \otimes \id + \id \otimes N_K)=0
\end{equation*}
is obtained.
Then we have
\begin{equation*}
\begin{split}
S_J(W(K)_a &\otimes W(K)_{-a-1}) \\
&=(S_J \cdot \id \otimes N_K^{a+1})(W(K)_a \otimes W(K)_{a+1}) \\
&=(-1)^{a+1}(S_J \cdot N_K^{a+1} \otimes \id)(W(K)_a \otimes W(K)_{a+1}) \\
&=(-1)^{a+1}S_J(W(K)_{-a-2} \otimes W(K)_{a+1}) \\
&=(-1)^{a+1+m}S_J(W(K)_{a+1} \otimes W(K)_{-a-2})
\end{split}
\end{equation*}
by using the equality $W(K)_{-k}=N^k(W(K)_k)$ for $k \ge 0$ 
(see e.g.~\cite[(2.2) Corollary]{sz}). 
Thus we obtain the conclusion by descending induction on $a$.
\end{proof}

\begin{cor}
In the situation above,
$S_J$ induces a morphism
\begin{equation*}
\Gr_k^{W(J)}(\xO_{D(J)} \otimes \xV) \otimes \Gr_{-k}^{W(J)}(\xO_{D(J)} \otimes \xV)
\longrightarrow \xO_{D(J)}
\end{equation*}
for a non-negative integer $k$.
\end{cor}

\begin{say}
In the situation above,
we define a morphism
\begin{equation*}
\overline{S}_k(J):\xP_k(J) \otimes \xP_k(J) \longrightarrow \xO_{D(J)}
\end{equation*}
by $\overline{S}_k(J)=S_J \cdot (\id \otimes N_J^k)$
for every subset $J \subset I$ and for every non-negative integer $k$.

By using the direct sum decomposition
\begin{equation*}
\Gr_k^{W(J)}(\xO_{D(J)} \otimes \xV)
=\bigoplus_{l \ge 0}N^l(\xP_{k+2l}(J))
\end{equation*}
for a non-negative integer $k$,
we obtain a morphism
\begin{equation*}
S_k(J):
\Gr_k^{W(J)}(\xO_{D(J)} \otimes \xV) \otimes \Gr_k^{W(J)}(\xO_{D(J)} \otimes \xV)
\longrightarrow \xO_{D(J)}
\end{equation*}
which is characterized by the following properties:
\begin{itemize}
\item
For non-negative integers $a,b$
we have
\begin{equation*}
S_k(J)(N^a(\xP_{k+2a}(J)) \otimes N^b(\xP_{k+2b}(J)))=0
\end{equation*}
if $a \not= b$.
\item
The diagram
\begin{equation*}
\begin{CD}
\xP_{k+2l}(J) \otimes \xP_{k+2l}(J) @>{\overline{S}_{k+2l}(J)}>> \xO_{D(J)} \\
@V{N^l \otimes N^l}VV @| \\
N^l(\xP_{k+2l}(J)) \otimes N^l(\xP_{k+2l}(J)) @>>{S_k(J)}> \xO_{D(J)}
\end{CD}
\end{equation*}
is commutative for every non-negative integer $l$.
\end{itemize}
For a positive integer $k$,
the morphism
\begin{equation*}
S_{-k}(J):
\Gr_{-k}^{W(J)}(\xO_{D(J)} \otimes \xV) \otimes \Gr_{-k}^{W(J)}(\xO_{D(J)} \otimes \xV)
\longrightarrow \xO_{D(J)}
\end{equation*}
is defined by identifying $\Gr_{-k}^{W(J)}(\xO_{D(J)} \otimes \xV)$
with $\Gr_k^{W(J)}(\xO_{D(J)} \otimes \xV)$ via the morphism $N(J)^k$.
More precisely,
$S_{-k}(J)$ is the unique morphism
such that the diagram
\begin{equation*}
\begin{CD}
\Gr_k^{W(J)}(\xO_{D(J)} \otimes \xV) \otimes \Gr_k^{W(J)}(\xO_{D(J)} \otimes \xV)
@>{S_k(J)}>> \xO_{D(J)} \\
@V{N(J)^k \otimes N(J)^k}VV @| \\
\Gr_{-k}^{W(J)}(\xO_{D(J)} \otimes \xV) \otimes \Gr_{-k}^{W(J)}(\xO_{D(J)} \otimes \xV)
@>>{S_{-k}(J)}> \xO_{D(J)}
\end{CD}
\end{equation*}
is commutative.
\end{say}

The following proposition 
plays an essential role
for the inductive argument 
in the proof of semipositivity theorem.

\begin{prop}
\label{key proposition}
In the situation {\rm \ref{setting for section 5}},
{\rm \ref{setting on filtrations}} and {\rm \ref{setting on polarization}},
the data
\begin{equation*}
(\Gr_k^{W(J)}(\xO_{D(J)} \otimes \xV), F, \nabla(J), S_k(J))
\end{equation*}
satisfies the conditions in
{\rm \ref{setting for section 5}},
{\rm \ref{setting on filtrations}}
and {\rm \ref{setting on polarization}} again.
\end{prop}
\begin{proof}
By Lemma \ref{local freeness of grW(K)} and \ref{say614},
the pair $(\Gr_k^{W(J)}(\xO_{D(J)} \otimes \xV),\nabla(J))$
satisfies the condition \ref{setting for section 5}.
As remarked in \ref{say614},
we have $\nabla(J)(F^p) \subset F^{p-1}$ for every $p$.
Moreover, the filtration $F$ on $\Gr_k^{W(J)}(\xO_{D(J)} \otimes \xV)$
satisfies the condition ($(m+k)$MH)
by Corollary \ref{corollary for condition mh}.

By definition,
the morphism
\begin{equation*}
S_J:
\Gr_k^{W(J)}(\xO_{D(J)} \otimes \xV) \otimes \Gr_{-k}^{W(J)}(\xO_{D(J)} \otimes \xV)
\end{equation*}
is compatible with the connections on both sides.
Therefore $\overline{S}_k(J)$ is compatible with the connections
because
\begin{equation*}
N_J:
\Gr_k^{W(J)}(\xO_{D(J)} \otimes \xV)
\longrightarrow \Gr_{k-2}^{W(J)}(\xO_{D(J)} \otimes \xV)
\end{equation*}
is compatible with the connection $\nabla(J)$ on the both sides.
Thus $S_k(J)$ is compatible with the connection.
Moreover we can check the equality
\begin{equation*}
S_k(J)(F^p \otimes F^q)=0
\end{equation*}
for $p+q > m+k$
by using $N_J^k(F^q) \subset F^{q-k}$.

There exists an open subset $U$ of $D(J)^*$
such that $\Gr_F^p\Gr_k^{W(J)}(\xO_{D(J)} \otimes \xV)$
is a locally free $\xO_{D(J)}$-module of finite rank for every $p$
and that $D(J) \setminus U$ 
is a nowhere dense closed analytic subspace of $D(J)$.

By the local description as in Deligne \cite{Deligne},
Katz \cite{katz2}
and by the property (1) in Lemma \ref{lemma by Cattani-Kaplan},
we can easily check that $\Ker(\nabla_k(J))|_{D(J)^*}$
admits an $\bR$-structure,
that is,
there exists a local system $\mathbb V_k(J)$
of finite dimensional $\bR$-vector spaces
with the property
$\bC \otimes \mathbb V_k(J) \simeq \Ker(\nabla_k(J))|_{D(J)^*}$.
Then the data
\begin{equation*}
(\mathbb V_k(J), (\Gr_k^{W(J)}(\xO_{D(J)} \otimes \xV), F), \nabla(J),S_k(J))|_U
\end{equation*}
is a polarized variation of $\bR$-Hodge structure of weight $m+k$,
by Schmid \cite{schmid}.
By Lemma \ref{lemma for the conition mh},
$\Gr_F^p\Gr_k^{W(J)}(\xO_{D(J)} \otimes \xV)$ turns out
to be locally free for every $k,p$
and then
\begin{equation*}
(\Gr_k^{W(J)}(\xO_{D(J)} \otimes \xV), F, \nabla(J), S_k(J))|_{D(J)^*}
\end{equation*}
underlies a polarized variation of $\bR$-Hodge structure of weight $m+k$ as desired.
By the continuity, $S_k(J)$ is $(-1)^{m+k}$-symmetric.
\end{proof}

Let us recall the definition of semipositive vector bundles in the 
sense of Fujita--Kawamata. Example \ref{ex-abe} below helps us 
understand the Fujita--Kawamata semipositivity. 

\begin{defn}[Semipositivity]\label{def620} 
A locally free sheaf (or a vector bundle) 
$\mathcal E$ of finite rank on a complete algebraic variety 
$X$ is said to be
{\em{semipositive}} if for 
every smooth curve $C$, for every morphism $\varphi:C\longrightarrow X$, and for every quotient 
invertible sheaf (or line bundle) $\mathcal Q$ of $\varphi^*\mathcal E$, we have 
$\deg _C\mathcal Q\geq 0$. 

It is easy to see that $\mathcal E$ is semipositive if and only if 
$\mathcal O_{\mathbb P_X(\mathcal E)}(1)$ is nef where 
$\mathcal O_{\mathbb P_X(\mathcal E)}(1)$ is the tautological 
line bundle on $\mathbb P_X(\mathcal E)$. 
\end{defn}

The following theorem is the main result of this section (cf.~\cite[Theorem 5]{kawamata1}).
It is a completely Hodge theoretic result.

\begin{thm}[Semipositivity theorem]\label{semi-po}
Let $X$ be a smooth complete complex variety,
$D$ a \snc divisor on $X$,
$\xV$ a locally free $\xO_X$-module of finite rank
equipped with a finite increasing filtration $W$
and a finite decreasing filtration $F$.
We assume the following$:$
\begin{enumerate}
\item
$F^a\xV=\xV$ and $F^{b+1}\xV=0$ for some $a < b$.
\item
$\Gr_F^p\Gr_m^W\xV$ is a locally free $\xO_X$-module of finite rank
for all $m,p$.
\item
For all $m$,
$\Gr_m^W\xV$ admits an integrable logarithmic connection
$\nabla_m$ with the nilpotent residue morphisms
which satisfies the conditions
$$\nabla_m(F^p\Gr_m^W\xV) \subset \Omega_X^1(\log D) \otimes F^{p-1}\Gr_m^W\xV$$
for all $p$.
\label{connection on Gr}
\item
The triple 
$(\Gr_m^W\xV, F\Gr_m^W\mathcal V, \nabla_m)|_{X \setminus D}$
underlies a polarizable variation of $\bR$-Hodge structure
of weight $m$ for every integer $m$.
\label{polarization on Gr}
\end{enumerate}
Then $(\Gr_F^a\xV)^*$ and $F^b\xV$ are semipositive.
\end{thm}
\begin{proof}
Since a vector bundle which is an extension of 
two semipositive vector bundles is also semipositive,
we may assume that the given $\xV$ is pure of weight $m$,
that is,
$W_m\xV=\xV, W_{m-1}\xV=0$ for an integer $m$, without loss of generality.
Then $\xV$ carries an integrable logarithmic connection $\nabla$
whose residue morphisms are nilpotent.
Thus the data $(\xV,F,\nabla)$
satisfies the conditions
in \ref{setting for section 5} and \ref{setting on filtrations}.
Note that $\xV$ is the canonical extension of $\xV|_{X \setminus D}$
because the residue morphisms of $\nabla$ are nilpotent.

By the assumption \eqref{polarization on Gr} above,
$\xV|_{X \setminus D}$ carries a polarization
which extends to a morphism
\begin{equation*}
S : \xV \otimes \xV \longrightarrow \xO_{X}
\end{equation*}
by functoriality of the canonical extensions.
We can easily see that the data $(\xV,F,\nabla)$ and $S$
satisfies the conditions
in \ref{setting for section 5},
\ref{setting on filtrations}
and \ref{setting on polarization}.

For the case of $\dim X=1$,
we obtain the conclusion by Zucker \cite{zucker1} 
(see also Kawamata \cite{kawamata1} and the proof of \cite[Theorem 5.20]{subadd}).

Next, we study the case of $\dim X > 1$.
Let $\varphi:C \longrightarrow X$ be a morphism from a smooth
projective curve.
The irreducible decomposition of $D$ is denoted by
$D=\sum_{i \in I}D_i$ as in \ref{setting for section 5}.
We set $J=\{i \in I ; \varphi(C) \subset D_i\} \subset I$.
Then $\varphi(C) \subset D(J)$, $\varphi(C) \cap D(J)^* \not= \emptyset$
and $\varphi^*D_{I \setminus J}$ is an effective divisor on $C$.
By Proposition \ref{key proposition},
the locally free sheaf $\xO_{D(J)} \otimes \xV$
with the finite increasing filtration $W(J)$
and the finite decreasing filtration $F$
satisfies the assumptions (1)-(4)
for $D(J)$ with the \snc divisor $D(J) \cap D_{I \setminus J}$.
Therefore
$\varphi^*\xV=\varphi^*(\xO_{D(J)} \otimes \xV)$
with the induced filtrations $W$ and $F$
satisfies the assumptions (1)-(4)
for $C$ with the effective divisor $\varphi^*D_{I \setminus J}$.
Then we conclude the desired semipositivity by the case of $\dim X=1$.
\end{proof}

\begin{rem}
In Theorem \ref{semi-po}, if $X$ is not complete, 
then we have the following statement. 
Let $V$ be a complete subvariety 
of $X$. 
Then $(\Gr_F^a\mathcal V)^*|_V$ and $(F^b\mathcal V)|_V$ 
are semipositive 
locally free sheaves on $V$. 
It is obvious by the proof of Theorem \ref{semi-po}. 
\end{rem}

\begin{cor}\label{cor622}
Let $X$ and $D$ be as in {\em{Theorem \ref{semi-po}}}.
Assume that we are given an admissible \gpv of $\bR$-mixed Hodge structure
$V=((\bV, W),F)$ on $X \setminus D$ of unipotent monodromy.
We assume the conditions $F^a\xV=\xV$ and $F^{b+1}\xV=0$.
The canonical extensions of $\xV=\xO_{X \setminus D} \otimes \bV$
and of $W_k\xV=\xO_{X \setminus D} \otimes W_k$
are denoted by $\widetilde{\xV}$ and by $W_k\widetilde{\xV}$
for all $k$.
As stated in {\em{Proposition \ref{extension of Hodge filtration}}},
the Hodge filtration $F$ extends to $\widetilde{\xV}$ such that
$\Gr_F^p\Gr_k^W\widetilde{\xV}$ is locally free of finite rank for all $k,p$.
Then $(\Gr_F^a\widetilde{\xV})^*$ and $F^b\widetilde{\xV}$
are semipositive.
\end{cor}

We learned the following 
remark from Hacon. 

\begin{rem}
The proof of the semipositivity theorem 
in \cite[Theorem 8.10.12]{ko-ko} contains 
some ambiguities. 
In the same notation as in \cite[Theorem 8.10.12]{ko-ko}, 
if $D$ is a simple normal crossing 
divisor 
but is not a {\em{smooth}} 
divisor, then it is not clear how to express 
$R^mf_*\omega_{X/Y}(D)$ as an extension of $R^mf_*\omega_{D_J/Y}$'s. 
The case when $D=F$ is a {\em{smooth}} 
divisor is treated in the proof of \cite[Theorem 8.10.12]{ko-ko}. 
The same argument does not seem to be sufficient for the general case. 

Fortunately, \cite[Theorem 3.9]{high} 
is sufficient for all applications in \cite{ko-ko} (see also 
\cite{fujino-gongyo2}). 
For some related topics, see \cite{ffs}. 
\end{rem}

\section{Vanishing and torsion-free theorems}\label{sec-quasi-proj}
In this section, we discuss some 
generalizations of torsion-free and vanishing theorems 
for {\em{quasi-projective}} simple normal crossing pairs. 

First, let us recall the following very useful lemma. 
For a proof, see, for example, \cite[Lemma 3.3]{fujino-vanishing}. 

\begin{lem}[Relative vanishing lemma]
\label{rf} 
Let $f:Y\longrightarrow X$ be a proper morphism 
from a simple normal crossing 
pair $(Y, \Delta)$ to an algebraic variety $X$ 
such that $\Delta$ is a boundary 
$\mathbb R$-divisor on $Y$. 
We assume that $f$ is an isomorphism 
at the generic point of any stratum of the pair $(Y, \Delta)$. 
Let $L$ be a Cartier divisor on $Y$ such 
that $L\sim _{\mathbb R}K_Y+\Delta$. 
Then $R^qf_*\mathcal O_Y(L)=0$ for every $q>0$. 
\end{lem} 

As an application of Lemma \ref{rf}, we obtain 
Lemma \ref{lemA}. We will use it several times in Section \ref{sec4}. 

\begin{lem}[{cf.~\cite[Lemma 2.7]{fujino-lc}}]\label{lemA}
Let $(V_1, D_1)$ and $(V_2, D_2)$ be simple normal crossing 
pairs such that 
$D_1$ and $D_2$ are reduced. 
Let $f:V_1\longrightarrow V_2$ be a proper morphism. 
Assume that there is a Zariski open subset $U_1$ {\em{(}}resp.~$U_2${\em{)}} of $V_1$ 
{\em{(}}resp.~$V_2${\em{)}} such that 
$U_1$ {\em{(}}resp.~$U_2${\em{)}} contains the generic point of any stratum of $(V_1, D_1)$ 
{\em{(}}resp.~$(V_2, D_2)${\em{)}} and that 
$f$ induces an isomorphism between $U_1$ and $U_2$. 
Then $R^if_*\omega_{V_1}(D_1)=0$ for every $i>0$ and 
$f_*\omega_{V_1}(D_1)\simeq \omega_{V_2}(D_2)$. 
By Grothendieck duality, 
we obtain that 
$R^if_*\mathcal O_{V_1}(-D_1)=0$ for every $i>0$ and $f_*\mathcal O_{V_1}(-D_1)\simeq 
\mathcal O_{V_2}(-D_2)$. 
\end{lem}
\begin{proof}
We can write 
$$
K_{V_1}+D_1=f^*(K_{V_2}+D_2)+E
$$ 
such that $E$ is $f$-exceptional. 
We consider the following commutative diagram 
$$
\begin{CD}
V_1^{\nu}@>{f^\nu}>> V_2^{\nu}\\
@V{\nu_1}VV @VV{\nu_2}V \\
V_1@>>{f}> V_2
\end{CD}
$$ 
where 
$\nu_1: V_1^\nu\longrightarrow V_1$ and $\nu_2: V_2^\nu \longrightarrow V_2$ are the normalizations. 
We can write $K_{V_1^{\nu}}+\Theta_1=\nu_1^*(K_{V_1}+D_1)$ and 
$K_{V_2^{\nu}}+\Theta_2=\nu_2^*(K_{V_2}+D_2)$. 
By pulling back $K_{V_1}+D_1=f^*(K_{V_2}+D_2)+E$ to $V_1^\nu$ by $\nu_1$, 
we have 
$$
K_{V_1^\nu}+\Theta_1=(f^\nu)^*(K_{V_2^\nu}+\Theta_2)+\nu_1^*E. 
$$ 
Note that $V_2^\nu$ is smooth and $\Theta_2$ is a reduced simple normal crossing divisor on $V_2^\nu$. 
By assumption, $f^\nu$ is an isomorphism 
over the generic point of any lc center of the 
pair $(V_2^\nu, \Theta_2)$ (cf.~\ref{110basic}). 
Therefore, $\nu_1^*E$ is effective since $K_{V_2^\nu}+\Theta_2$ is Cartier. 
Thus, we obtain that $E$ is effective. Since $V_2$ satisfies Serre's $S_2$ condition, 
we can check that $\mathcal O_{V_2}\simeq f_*\mathcal O_{V_1}$ and $f_*\mathcal O_{V_1}
(K_{V_1}+D_1)\simeq \mathcal O_{V_2}(K_{V_2}+D_2)$. 
On the other hand, 
we obtain $R^if_*\mathcal O_{V_1}(K_{V_1}+D_1)=0$ for 
every $i>0$ by Lemma \ref{rf}. 
Therefore, $Rf_*\mathcal O_{V_1}(K_{V_1}+D_1)\simeq \mathcal O_{V_2}(K_{V_2}+D_2)$ in the derived 
category of coherent sheaves on $V_2$. 
Since $V_1$ and $V_2$ are Gorenstein, 
we have 
\begin{align*}
Rf_*\mathcal O_{V_1}(-D_1)&
\simeq R\mathcal H om (Rf_*\omega^\bullet_{V_1}(D_1), \omega^\bullet_{V_2})
\\&\simeq R\mathcal H om (Rf_*\omega_{V_1}(D_1), \omega_{V_2})\\ 
&\simeq R\mathcal Hom (\omega_{V_2}(D_2), \omega_{V_2})\simeq \mathcal O_{V_2}(-D_2) 
\end{align*} 
in the derived category of coherent sheaves on $V_2$ by Grothendieck duality. 
Therefore, 
we have $R^if_*\mathcal O_{V_1}(-D_1)=0$ for every $i>0$ and 
$f_*\mathcal O_{V_1}(-D_1)\simeq \mathcal O_{V_2}(-D_2)$. 
\end{proof}

Next, we prove the following theorem, which 
was proved for {\em{embedded simple normal crossing pairs}} 
in \cite[Theorem 2.39]{book} and \cite[Theorem 2.47]{book}.
We note that we do not assume the existence of 
ambient spaces in Theorem \ref{quasi-proj}. However, we need the assumption 
that $X$ is {\em{quasi-projective}}. 

\begin{thm}[{cf.~\cite[Theorem 2.39 and Theorem 2.47]{book}}]\label{quasi-proj} 
Let $(X, B)$ be a {\em{quasi-projective}} 
simple normal crossing 
pair such that $B$ is a boundary $\mathbb R$-divisor on $X$. 
Let $f:X\longrightarrow Y$ be a proper morphism between algebraic 
varieties and 
let $L$ be a Cartier divisor on $X$. 
Let $q$ be an arbitrary integer. 
Then we have the following 
properties. 
\begin{itemize}
\item[(i)] Assume that $L-(K_X+B)$ is $f$-semi-ample, 
that is, $L-(K_X+B)=\sum _i a_iD_i$ where 
$D_i$ is an $f$-semi-ample Cartier divisor on $X$ and $a_i$ is a 
positive real number for every $i$. 
Then every associated prime of $R^qf_*\mathcal O_X(L)$ is 
the generic point of the $f$-image 
of some stratum of $(X, B)$. 
\item[(ii)] Let $\pi:Y\longrightarrow Z$ be a proper 
morphism. We assume that 
$L-(K_X+B)\sim _{\mathbb R}f^*A$ for some 
$\mathbb R$-Cartier $\mathbb R$-divisor $A$ on $Y$ such that 
$A$ is nef and log big over $Z$ with respect to 
$f:(X, B)\longrightarrow Y$ {\em{(}}see \cite[Definition 2.46]{book}{\em{)}}. 
Then $R^qf_*\mathcal O_X(L)$ is 
$\pi_*$-acyclic, that is, 
$R^p\pi_*R^qf_*\mathcal O_X(L)=0$ for 
every $p>0$. 
\end{itemize}
\end{thm}
\begin{proof}
Since $X$ is quasi-projective, 
we can embed $X$ into a smooth projective 
variety $V$. 
By Lemma \ref{lem02} below, 
we can replace $(X, B)$ and $L$ with 
$(X_k, B_k)$ and $\sigma^*L$ 
and assume that there exists 
an $\mathbb R$-divisor 
$D$ on $V$ such that 
$B=D|_X$. 
Then, by using Bertini's theorem, 
we can take a general complete intersection $W\subset 
V$ such that 
$\dim W=\dim X+1$, 
$X\subset W$, 
and $W$ is smooth at the 
generic point of every stratum of $(X, B)$ 
(cf.~the proof of \cite[Proposition 10.59]{kollar-book}). 
We take a suitable resolution $\psi:M\longrightarrow W$ with the following 
properties.  
\begin{itemize}
\item[(A)] The strict transform $X'$ of 
$X$ is a simple normal crossing 
divisor on $M$. 
\item[(B)] We can write 
$$
K_{X'}+B'=\varphi^*(K_X+B)+E
$$
such that $\varphi=\psi|_{X'}$, $(X', B'-E)$ is a globally embedded simple normal 
crossing pair (cf.~Definition \ref{gsnc0}), $B'$ is a boundary $\mathbb R$-divisor 
on $X'$, $\lceil E\rceil$ is effective and $\varphi$-exceptional, and 
the $\varphi$-image of every stratum of $(X', B')$ is a stratum of $(X, B)$. 
\item[(C)] $\varphi$ is an isomorphism over the 
generic point of every stratum of $(X, B)$. 
\item[(D)] $\varphi$ is an isomorphism at the generic point of 
every stratum of $(X', B')$. 
\end{itemize} 
Then  $$
K_{X'}+B'+\{-E\}=\varphi^*(K_X+B)+\lceil E\rceil, 
$$
$$\varphi_*\mathcal  O_{X'}(\varphi^*L+\lceil E\rceil)\simeq \mathcal O_X(L), 
$$
and 
$$
R^q\varphi_*\mathcal O_{X'}(\varphi^*L+\lceil E\rceil)=0
$$  for every $q>0$ by Lemma \ref{rf}. 
We note that 
$$
\varphi^*L+\lceil E\rceil -(K_{X'}+B'+\{-E\})=\varphi^*(L-(K_X+B))
$$ 
and that $\varphi$ is an isomorphism at the generic point of 
every stratum of $(X', B'+\{-E\})$.  

Therefore, by replacing $(X, B)$ and $L$ with 
$(X', B'+\{-E\})$ and $\varphi^*L+\lceil E\rceil$, 
we may assume that 
$(X, B)$ is a quasi-projective globally 
embedded simple normal crossing 
pair (cf.~Definition \ref{gsnc0}). 
In this case, the claims have already been established by \cite[Theorem 2.39]{book} 
and \cite[Theorem 2.47]{book}. 
\end{proof}

For some generalizations of Theorem \ref{quasi-proj} for 
{\em{semi log canonical 
pairs}}, see \cite{fujino-slc}. 
\begin{rem}
Theorem \ref{quasi-proj} (i) is contained in \cite[Theorem 1.1 (i)]{fujino-vanishing}.  
In \cite[Theorem 1.1]{fujino-vanishing}, 
$X$ is not assumed to be {\em{quasi-projective}}. 
On the other hand, 
we do not know how to 
remove the quasi-projectivity of $X$ from Theorem \ref{quasi-proj} (ii). 
\end{rem}

By direct calculations, we can obtain the following elementary 
lemma. It was used in the proof of Theorem \ref{quasi-proj}. 

\begin{lem}[{cf.~\cite[Lemma 3.60]{book}}]\label{lem02}
Let $(X, B)$ be a simple normal crossing 
pair such that 
$B$ is a boundary $\mathbb R$-divisor. 
Let $V$ be a smooth 
variety such that 
$X\subset V$. 
Then we can construct a sequence of blow-ups
$$
V_k\longrightarrow V_{k-1}\longrightarrow \cdots \longrightarrow V_0=V
$$
with the following properties. 
\begin{itemize}
\item[(1)] $\sigma _{i+1}:V_{i+1}\longrightarrow V_i$ is the 
blow-up along a smooth irreducible component of $\Supp B_i$ 
for every $i\geq 0$. 
\item[(2)] We put $X_0=X$, $B_{0}=B$, and 
$X_{i+1}$ is the strict transform of $X_i$ for every $i\geq 0$. 
\item[(3)]  We put $K_{X_{i+1}}+B_{i+1}=\sigma ^*_{i+1}(K_{X_i}+B_i)$ for 
every $i\geq 0$. 
\item[(4)] There exists an $\mathbb R$-divisor $D$ on $V_k$ such that 
$D|_{X_k}=B_k$ and $B_k$ is a boundary $\mathbb R$-divisor on $X_k$. 
\item[(5)] $\sigma_*\mathcal O_{X_k}\simeq \mathcal O_X$ and 
$R^q\sigma_*\mathcal O_{X_k}=0$ for 
every $q>0$, 
where $\sigma: V_k\longrightarrow V_{k-1}\longrightarrow \cdots \longrightarrow V_0=V$. 
\end{itemize}
\end{lem}
\begin{proof}
All we have to do is to check the property (5). 
We note that 
$
\sigma_{i+1*}\mathcal O_{V_{i+1}}(K_{V_{i+1}})\simeq \mathcal O_{V_{i+1}}(K_{V_{i+1}})$
and 
$
R^q\sigma_{i+1*}\mathcal O_{V_{i+1}}(K_{V_{i+1}})=0$ for 
every $q$ and for each step $\sigma_{i+1}: V_{i+1}\longrightarrow V_i$ 
by Lemma \ref{lemA}. 
Therefore we obtain 
$
R^q\sigma_*\mathcal O_{X_k}(K_{X_k})=0$ for every $q>0$ and 
$\sigma_*\mathcal O_{X_k}(K_{X_k})\simeq \mathcal O_X(K_X)$. 
Thus by Grothendieck duality we obtain $R^q\sigma_*\mathcal O_{X_k}=0$ 
for every $q>0$ and $\sigma_*\mathcal O_{X_k}\simeq \mathcal O_X$ as in the 
proof of Lemma \ref{lemA}. 
\end{proof}

As a special case of Theorem \ref{quasi-proj} (i), 
we have: 

\begin{cor}[Torsion-freeness]\label{cor-new-torsion} 
Let $(X, D)$ be a quasi-projective simple normal crossing 
pair such that $D$ is reduced and let $f:X\to Y$ be a projective 
surjective morphism onto a smooth algebraic variety $Y$. Assume that 
every stratum of $(X, D)$ is dominant onto $Y$. 
Then $R^if_*\omega_{X/Y}(D)$ is torsion-free for every $i$. 
\end{cor}
\begin{proof} 
It is sufficient to prove that $R^if_*\mathcal O_X(K_X+D)$ is torsion-free 
for every $i$ since $\mathcal O_Y(K_Y)$ is locally free. 
By Theorem \ref{quasi-proj} (i), every associated prime of $R^if_*\mathcal O_X(K_X
+D)$ is the generic point of $Y$ for every $i$. 
This means that $R^if_*\mathcal O_X(K_X+D)$ is torsion-free for every $i$. 
\end{proof}

We will use it in Section \ref{sec4}. 

\section{Higher direct images of log canonical divisors}
\label{sec4}
This section is the main part of this paper. 
The following theorem is our main theorem 
(cf.~\cite[Theorem 5]{kawamata1}, 
\cite[Theorem 2.6]{ko2}, \cite[Theorem 1]{n}, 
\cite[Theorems 3.4 and 3.9]{high}, 
and \cite[Theorem 1.1]{kawamata}), 
which is a {\em{natural}} generalization of 
the Fujita--Kawamata semipositivity theorem for 
simple normal crossing pairs. 

\begin{thm}\label{main}
Let $(X, D)$ be a simple normal crossing pair such that 
$D$ is reduced and let $f:X\longrightarrow Y$ be 
a projective surjective morphism onto a smooth 
algebraic variety $Y$. 
Assume that every stratum of $(X, D)$ is dominant onto $Y$. 
Let $\Sigma$ be a simple normal crossing divisor on $Y$ such that 
every stratum of $(X, D)$ is smooth over $Y^*=Y\setminus \Sigma$. 
We put $X^*=f^{-1}(Y^*)$, $D^*=D|_{X^*}$, and $d=\dim X-\dim Y$. 
Let $\iota: X^*\setminus D^*\longrightarrow X^*$ be the natural open immersion. 
Then we obtain 
\begin{itemize}
\item[(1)] $R^k(f|_{X^*})_{*}\iota_!\mathbb Q_{X^*\setminus D^*}
\simeq R^k(f|_{X^*\setminus D^*})_!\mathbb 
Q_{X^*\setminus D^*}$ underlies a graded polarizable 
variation of $\mathbb Q$-mixed Hodge structure on $Y^*$ for every $k$. 
Moreover, it is {\em{admissible}}. 
\end{itemize}
We put $\mathcal V^{k}_{Y^*}=
R^k(f|_{X^*})_{*}\iota_!\mathbb Q_{X^*\setminus D^*}\otimes 
\mathcal O_{Y^*}$ for every $k$. 
Let $$\cdots \subset  F^{p+1}(\mathcal V^k_{Y^*})
\subset F^{p}(\mathcal V^k_{Y^*})
\subset F^{p-1}(\mathcal V^k_{Y^*})\subset \cdots$$ be 
the Hodge filtration. 
We assume 
that all the local monodromies on $R^{d-i}(f|_{X^*})_{*}\iota_!
\mathbb Q_{X^*\setminus D^*}$ around $\Sigma$ are unipotent. 
Then we obtain 
\begin{itemize}
\item[(2)] $R^{d-i}f_*\mathcal O_X(-D)$ is isomorphic to 
the canonical extension of $$\Gr^0_F(\mathcal 
V^{d-i}_{Y^*})=F^0(\mathcal V^{d-i}_{Y^*})/ F^1(\mathcal V^{d-i}_{Y^*}).$$ 
It is denoted by $\Gr^0_F(\mathcal V^{d-i}_Y)$. 
In particular, $R^{d-i}f_*\mathcal O_X(-D)$ is locally free. 
\end{itemize}
By Grothendieck duality, we obtain 
\begin{itemize}
\item[(3)] $R^{i}f_*\omega_{X/Y}(D)$ is isomorphic to the 
canonical extension of 
$$(\Gr ^0_F(\mathcal V^{d-i}_{Y^*}))^*=\mathcal {H}om_{\mathcal O_{Y^*}}(\Gr^0_F(\mathcal V^{d-i}_{Y^*}), 
\mathcal O_{Y^*}).$$ 
Thus, we have $R^if_*\omega_{X/Y}(D)\simeq (\Gr^0_F(\mathcal V^{d-i}_Y))^*$. 
In particular, $R^{i}f_*\omega_{X/Y}(D)$ is locally free. 
\item[(4)] We further assume that 
$Y$ is complete. 
Then $R^if_*\omega_{X/Y}(D)$ is semipositive. 
\end{itemize}
\end{thm}

Even the following 
very special case of Theorem \ref{main} has never been checked before. 
It does not follow from \cite[Theorem 1.1]{kawamata}. 

\begin{cor}
Let $f:X\longrightarrow Y$ be a projective morphism 
from a simple normal crossing 
variety $X$ to a smooth complete algebraic variety $Y$. 
Assume that every stratum of $X$ is smooth over $Y$. Then 
$R^if_*\omega_{X/Y}$ is a semipositive locally free sheaf for every $i$. 
\end{cor}

It is natural to prove Theorem \ref{main2}, which is a slight generalization of 
(2) and (3) in Theorem \ref{main}, simultaneously with Theorem \ref{main}. 

\begin{thm}[{cf.~\cite[Theorem 2.6]{ko2}}]\label{main2}
We use the same notation and assumptions as in {\em{Theorem \ref{main}}}. 
We do not assume that the local monodromies on 
$R^{d-i}{(f|_{X^*})}_*\iota_!\mathbb Q_{X^*\setminus D^*}$ around $\Sigma$ are unipotent. 
Then we obtain the following properties. 
\begin{itemize}
\item[(a)] $R^{d-i}f_*\mathcal O_X(-D)$ is isomorphic to 
the lower canonical extension of $$\Gr^0_F(\mathcal 
V^{d-i}_{Y^*})=F^0(\mathcal V^{d-i}_{Y^*})/F^1(\mathcal V^{d-i}_{Y^*}).$$ 
In particular, $R^{d-i}f_*\mathcal O_X(-D)$ is locally free. 
\end{itemize}
By Grothendieck duality, we obtain 
\begin{itemize}
\item[(b)] $R^{i}f_*\omega_{X/Y}(D)$ is isomorphic to the 
upper canonical extension of 
$$(\Gr ^0_F(\mathcal V^{d-i}_{Y^*}))^*=\mathcal {H}om_{\mathcal O_{Y^*}}(\Gr^0_F(\mathcal V^{d-i}_{Y^*}), 
\mathcal O_{Y^*}).$$ 
In particular, $R^{i}f_*\omega_{X/Y}(D)$ is locally free. 
\end{itemize}
\end{thm}

Before we start the proof of Theorem \ref{main} and Theorem \ref{main2},
we give a remark on the canonical extensions. 

\begin{rem}[Upper and lower canonical extensions of 
Hodge bundles]\label{upper-lower} 
Let ${}^l\mathcal V_{Y^*}^k$ (resp.~${}^u\mathcal V_{Y^*}^k$) be the 
{\em{Deligne extension}} of $\mathcal V_{Y^*}^k$ on $Y$ such that 
the eigen values of the residue 
of the connection are contained in $[0, 1)$ (resp.~$(-1, 0]$). 
We call it the {\em{lower canonical extension}} (resp.~{\em{upper 
canonical extension}}) of $\mathcal V_{Y^*}^k$ following 
\cite[Definition 2.3]{ko2}. 
If the local monodromies on $R^k(f|_{X^*})_*\iota_!\mathbb Q_{X^*\setminus D^*}
\simeq R^k(f|_{X^*\setminus D^*})_!\mathbb Q_{X^*\setminus D^*}$ 
around $\Sigma$ are unipotent, then 
$$
{}^l\mathcal V_{Y^*}^k={}^u\mathcal V_{Y^*}^k
$$ 
holds. In this case, we set  
$$
\mathcal V_Y^k={}^l\mathcal V_{Y^*}^k={}^u\mathcal V_{Y^*}^k
$$ 
and call it the {\em{canonical extension}} of $\mathcal V_{Y^*}^k$. 
Let $j:Y^*\longrightarrow Y$ be the natural open immersion. 
We set 
$$
{}^l\!F^p(\mathcal V_{Y^*}^k)=j_*F^p(\mathcal V_{Y^*}^k)\cap {}^l\mathcal V_{Y^*}^k
$$ 
and call it the {\em{lower canonical extension}} of $F^p(\mathcal V_{Y^*}^k)$ 
on $Y$. 
We can define the {\em{upper canonical extension}} 
${}^u\!F^p(\mathcal V_{Y^*}^k)$ of 
$F^p(\mathcal V_{Y^*}^k)$ on $Y$ similarly.
As above, when the local monodromies 
on $R^k(f|_{X^*})_*\iota_!\mathbb Q_{X^*\setminus D^*}$ 
around $\Sigma$ are unipotent, we write $F^p(\mathcal V_Y^k)$ to 
denote ${}^l\!F^p(\mathcal V_{Y^*}^k)=
{}^u\!F^p(\mathcal V_{Y^*}^k)$ and call 
it the {\em{canonical extension}} of $F^p(\mathcal V_{Y^*}^k)$. 
Theorem \ref{main2} (a) means that 
the following short exact sequence 
\begin{align}\label{541}
0\longrightarrow F^1(\mathcal V_{Y^*}^{d-i})\longrightarrow 
F^0(\mathcal V_{Y^*}^{d-i})
\longrightarrow \Gr_F^0(\mathcal V_{Y^*}^{d-i})\longrightarrow 0
\end{align}
is extended to the short exact sequence 
\begin{align}\label{542}
0\longrightarrow {}^l\!F^1(\mathcal V_{Y^*}^{d-i})\longrightarrow 
{}^l\!F^0(\mathcal V_{Y^*}^{d-i})
\longrightarrow R^{d-i}f_*\mathcal O_X(-D)\longrightarrow 0. 
\end{align}
Let us consider the dual variation of mixed Hodge 
structure (cf.~Remark \ref{dual of VMHS}). 
The dual local system of 
$R^k(f|_{X^*})_*\iota_!\mathbb Q_{X^*\setminus D^*}$
underlies $(\xV_{Y^*}^k)^*$. 
The locally free sheaf $(\xV_{Y^*}^k)^*$ 
carries the Hodge filtration $F$
defined in Remark \ref{dual of VMHS}.
Theorem \ref{main2} (b) means that
the short exact sequence
\begin{align}
\label{543}
0 \longrightarrow F^1((\mathcal V_{Y^*}^{d-i})^*)
\longrightarrow F^0((\mathcal V_{Y^*}^{d-i})^*)
\longrightarrow \Gr_F^0((\mathcal V_{Y^*}^{d-i})^*) \longrightarrow 0
\end{align}
is extended to the short exact sequence 
\begin{align}
\label{544}
0 \rightarrow {}^u\!F^1((\mathcal V_{Y^*}^{d-i})^*)
\rightarrow {}^u\!F^0((\mathcal V_{Y^*}^{d-i})^*)
\rightarrow R^if_*\omega_{X/Y}(D) \rightarrow 0.
\end{align}
We note that
\begin{equation*}
\Gr_F^{-p}((\mathcal V_{Y^*}^{d-i})^*)
\simeq
(\Gr_F^p(\mathcal V_{Y^*}^{d-i}))^*
\end{equation*}
for every $p$ as in Remark \ref{dual of VMHS}.
We also note that all the terms in \eqref{542} and \eqref{544} are 
locally free sheaves by \cite[Proposition 1.11.3]{kashiwara} since 
$R^k(f|_{X^*})_*\iota_!\mathbb Q_{X^*\setminus D^*}$ underlies 
an admissible graded polarized variation of $\mathbb Q$-mixed Hodge structure 
on $Y^*$ for every $k$ by Theorem \ref{main} (1). 
See also Proposition \ref{extension of Hodge filtration}. 
Let us see the relationship between \eqref{542} and \eqref{544} 
in details for the reader's convenience. 
By definition, it is easy to see that 
$$
({}^l\mathcal V_{Y^*}^k)^*={}^u((\mathcal V_{Y^*}^k)^*)
$$ 
for every $k$. 
We can check that 
$$
0\longrightarrow {}^u\!F^p((\mathcal V_{Y^*}^{k})^*)\longrightarrow 
({}^l\mathcal V_{Y^*}^k)^*\longrightarrow 
({}^l\!F^{1-p}(\mathcal V_{Y^*}^{k}))^*\longrightarrow 0
$$ 
is exact for every $p$ and $k$ (cf.~Lemma \ref{lemma 5.1}). 
Then we have  the following big commutative diagram. 
$$
\xymatrix{
&&&0\ar[d]&\\
&0\ar[d]&&R^if_*\omega_{X/Y}(D)\ar[d]&\\
0\ar[r]&{}^u\!F^1((\mathcal V_{Y^*}^{d-i})^*)\ar[d]\ar[r]&({}^l\mathcal V_{Y^*}^{d-i})^*
\ar@{=}[d]\ar[r]&({}^l\!F^0(\mathcal V_{Y^*}^{d-i}))^*\ar[r]\ar[d]&0\\
0\ar[r]&{}^u\!F^0((\mathcal V_{Y^*}^{d-i})^*)\ar[d]\ar[r]
&({}^l\mathcal V_{Y^*}^{d-i})^*\ar[r]&({}^l\!F^1(\mathcal V_{Y^*}^{d-i}))^*\ar[d]\ar[r]&0\\
&R^if_*\omega_{X/Y}(D)\ar[d]&&0&\\ 
&0&&&
}
$$
The first vertical line is nothing but \eqref{544} and 
the third 
vertical line is the dual 
of \eqref{542}. 

Theorem \ref{main} (2) (resp.~(3)) is a special case of Theorem 
\ref{main2} (a) (resp.~(b)). 
\end{rem}

Let us start the proof of Theorem \ref{main} and Theorem \ref{main2}. 

\begin{proof}[Proof of {\em{Theorem \ref{main}}} and {\em{Theorem \ref{main2}}}] 
The statement (1) in Theorem \ref{main} follows from 
Theorem \ref{GPVMHS for snc pair}.
We note that 
(4) in Theorem \ref{main} follows 
from Theorem \ref{semi-po} and Corollary \ref{cor622} 
by (3) in Theorem \ref{main}. 

Without loss of generality, by \cite[Theorem 1.4]{bierstone-p} and Lemma \ref{lemA}, we may 
assume that $\Supp (f^*\Sigma \cup D)$ is a simple normal crossing divisor on $X$.

In Step \ref{step1} and Step \ref{step2}, 
we prove (2) and (3) in Theorem \ref{main} for every $i$ under the assumption that all 
the local monodromies 
on $R^kf_*\mathbb C_{S^*}$, where 
$S$ is a stratum of $(X, D)$, $S^*=S|_{X^*}$, and 
$k$ is any integer, 
around $\Sigma$ are unipotent. 
In Step \ref{step33} and Step \ref{step44}, we prove Theorem \ref{main2}, 
which contains (2) and (3) in Theorem \ref{main}. 

From now on, we assume that all the local monodromies 
on $R^kf_*\mathbb C_{S^*}$, where 
$S$ is a stratum of $(X, D)$, $S^*=S|_{X^*}$, and 
$k$ is any integer, around $\Sigma$ 
are unipotent. 

\setcounter{step}{0}
\begin{step}[The case when $\dim Y=1$]\label{step1}
By shrinking $Y$, we may assume that 
$Y$ is the unit disc $\Delta$ in $\mathbb C$ and $\Sigma =\{0\}$ in $\Delta$. 
We put $E=f^{-1}(0)$. 
By considering
$\Omega_{(D\cap X)_{\bullet}/Y}(\log E_\bullet)$
as in the proof of Lemma \ref{GPVMHS for semi-simplicial variety},
we obtain 
that $R^{d-i}f_*\mathcal O_X(-D)$ is isomorphic to the canonical extension of 
$\Gr ^0_F(\mathcal V_{Y^*}^{d-i})$ for every $i$ (see also 
Remark \ref{remark for GrF0}). 
Therefore, we obtain
$R^if_*\omega_{X/Y}(D)\simeq 
(\Gr_F^0(\mathcal V_Y^{d-i}))^*$ for every $i$ by 
Grothendieck duality.  
\end{step}

\begin{step}
[The case when $l:=\dim Y\geq 2$]\label{step2}
We shall prove the statement (3) by induction on $l$ for every $i$. 

By Step \ref{step1}, there is an open subset $Y_1$ of $Y$ 
such that $\codim (Y\setminus Y_1)\geq 2$ and 
that 
$$
R^if_*\omega_{X/Y}(D)|_{Y_1}\simeq (\Gr^0_F(\mathcal V^{d-i}_Y))^*|_{Y_1}. 
$$ 
Since $(\Gr^0_F(\mathcal V^{d-i}_Y))^*$ is locally free 
(see Remark \ref{upper-lower}), we obtain a 
homomorphism 
$$
\varphi_Y^i: 
R^if_*\omega_{X/Y}(D)\longrightarrow (\Gr^0_F(\mathcal V^{d-i}_Y))^*. 
$$ 
We will prove that $\varphi^i_Y$ is an isomorphism. Without loss of generality, 
we may assume that $X$ and $Y$ are quasi-projective by shrinking $Y$.  
By Corollary \ref{cor-new-torsion}, $R^if_*\omega_{X/Y}(D)$ is torsion-free. 
Therefore, $\Ker \varphi_Y^i=0$. 
We put $G_Y^i:= \Coker \varphi_Y^i$. 
Taking a general hyperplane cut, we see that $\Supp G_Y^i$ 
is a finite set by the induction hypothesis. 
Assume that $G_Y^i\ne 0$. 
We may also assume that $\Supp G^i_Y=\{P\}$ by shrinking $Y$. 
Let $\mu:W\longrightarrow Y$ be the blowing up at $P$ and 
set $E=\mu^{-1}(P)$. 
Then $E\simeq \mathbb P^{l-1}$. 
By \cite[Theorem 1.5]{bierstone-milman} and \cite[Theorem 1.4]{bierstone-p}, 
we can take a projective birational 
morphism $\pi:X'\longrightarrow X$ from a simple normal crossing  
variety $X'$ with the following properties: 
\begin{itemize}
\item[(i)] the composition 
$X'\longrightarrow X\longrightarrow Y\dashrightarrow W$ 
is a morphism. 
\item[(ii)] $\pi$ is an isomorphism over $X^*$. 
\item[(iii)] $\Exc (\pi)\cup D'$ is a 
simple normal crossing divisor on $X'$, 
where $D'$ is the strict transform of $D$. 
\end{itemize}
We obtain that $R^qf_*\omega _{X/Y}(D)\simeq 
R^q(f\circ \pi)_*\omega_{X'/Y}(D')$ for every $q$ because 
$R\pi_*\omega_{X'}(D')\simeq \omega_X(D)$ in the derived category of coherent sheaves on $X$ by Lemma 
\ref{lemA}. 
We note that every stratum of $(X', D')$ is dominant onto 
$Y$. We also note the following commutative diagram. 
$$
\begin{CD}
X'@>{\pi}>>X\\ 
@V{g}VV @VV{f}V\\
W@>>{\mu}> Y
\end{CD}
$$
By replacing $(X,D)$ with 
$(X',D')$, we may assume that there is a morphism 
$g:X\longrightarrow W$ such that $f=\mu\circ g$. 
Since $g:X\longrightarrow W$ is in the same situation as $f$, 
we obtain the exact sequence: 
$$
0\longrightarrow R^ig_*\omega_{X/W}(D)\longrightarrow 
(\Gr^0_F(\mathcal V^{d-i}_W))^*\longrightarrow 
G_W^i \longrightarrow 0.
$$ 
Tensoring $\mathcal O_W(\nu E)$ for $0\leq \nu \leq l-1$ and 
applying $R^j\mu_*$ for $j\geq 0$ to each $\nu$, 
we have a exact sequence 
\begin{eqnarray*}
0 & \longrightarrow & \mu_*(R^ig_*\omega_{X/W}(D)\otimes \mathcal O_W(\nu E)) 
\longrightarrow \mu_*((\Gr^0_F(\mathcal V^{d-i}_W))^*\otimes \mathcal O_W(\nu E)) \\ 
&\longrightarrow &\mu_*(G_W^i\otimes \mathcal O_W(\nu E)) 
\longrightarrow  R^1\mu_*(R^ig_*\omega _{X/W}(D)\otimes \mathcal O_{W}(\nu E)) \\ 
&\longrightarrow & R^1\mu_*((\Gr^0_F(\mathcal V^{d-i}_W))^*\otimes 
\mathcal O_W(\nu E)) 
\longrightarrow 0 
\end{eqnarray*} 
and $R^q\mu_*(R^ig_*\omega_{X/W}(D)\otimes \mathcal O_W(\nu E))
\simeq R^q\mu_*((\Gr^0_F(\mathcal V^{d-i}_W))^*\otimes \mathcal O_W(\nu E))
$ for $q\geq 2$. 

By \cite[Proposition 1]{kawamata2} 
and Remark \ref{upper-lower}, we obtain 
$(\Gr^0_F(\mathcal V^{d-i}_W))^*\simeq 
\mu^*(\Gr^0_F(\mathcal V^{d-i}_Y))^*$. 
We have 
$$
\mu_*((\Gr^0_F(\mathcal V^{d-i}_W))^*\otimes 
\mathcal O_W(\nu E))\simeq (\Gr^0_F(\mathcal V^{d-i}_Y))^*
$$ 
and 
$$
R^q\mu_*((\Gr^0_F(\mathcal V^{d-i}_W))^*\otimes 
\mathcal O_W(\nu E))=0
$$ for $q\geq 1$. 
Therefore, $R^q\mu_*(R^ig_*\omega_{X/W}(D)\otimes \mathcal O_W(\nu E))=0$ 
for $q\geq 2$ and 
\begin{eqnarray*}
0 & \longrightarrow & \mu_*(R^ig_*\omega_{X/W}(D)\otimes \mathcal O_W(\nu E)) 
\longrightarrow \mu_*((\Gr^0_F(\mathcal V^{d-i}_W))^*\otimes \mathcal O_W(\nu E)) \\ 
&\longrightarrow &\mu_*(G_W^i\otimes \mathcal O_W(\nu E)) 
\longrightarrow  R^1\mu_*(R^ig_*\omega _{X/W}(D)\otimes \mathcal O_{W}(\nu E))\\ &\longrightarrow 
& 0 
\end{eqnarray*} 
is exact. 
Since $\omega_W=\mu^*\omega_Y\otimes \mathcal O_W((l-1)E)$, 
we have a spectral sequence 
$$
E^{p,q}_2=R^p\mu_*(R^qg_*\omega_{X/W}(D)\otimes \mathcal O_W((l-1)E)) 
\Longrightarrow R^{p+q}f_*\omega_{X/Y}(D). 
$$ 
However, $E^{p,q}_2=0$ for $p\geq 2$ by the above argument. Thus 
\begin{eqnarray*}
0&\longrightarrow& R^1\mu_*R^{i-1}g_*\omega_{X/Y}(D)\longrightarrow R^if_*\omega_{X/Y}(D)\\ 
&\longrightarrow& 
\mu_*(R^ig_*\omega _{X/W}(D)\otimes \mathcal O_W((l-1)E))\longrightarrow 0. 
\end{eqnarray*} 
By Corollary \ref{cor-new-torsion}, $R^if_*\omega_{X/Y}(D)$ is torsion-free.  
So, we obtain $$R^1\mu_*R^{i-1}g_*\omega_{X/Y}(D)=0. $$
Therefore, for $q\geq 1$, we 
obtain 
\begin{itemize}
\item[(A)] $R^if_*\omega_{X/Y}(D)\simeq \mu_*(R^ig_*\omega_{X/W}(D)
\otimes \mathcal O_W((l-1)E))$ and  
\item[(B)] $R^q\mu_*(R^ig_*\omega_{X/W}(D)\otimes \mathcal O_W((l-1)E))=0$ 
\end{itemize} 
for every $i$. 

Next, we shall consider the following commutative 
diagram. 
$$
\begin{matrix}
0 & &0\\
\downarrow & &\downarrow\\ 
R^ig_*\omega _{X/W}(D)\otimes \mathcal O_W((l-2)E)&
\to &
R^ig_*\omega _{X/W}(D)\otimes \mathcal O_W((l-1)E) \\ 
\downarrow & & \downarrow\\ 
(\Gr^0_F(\mathcal V^{d-i}_W))^*\otimes \mathcal O_W((l-2)E)
& \to &
(\Gr^0_F(\mathcal V^{d-i}_W))^*\otimes \mathcal O_W((l-1)E) \\ 
\downarrow & & \downarrow\\
G_W^{i}\otimes \mathcal O_W((l-2)E)
&\to & 
G_W^{i}\otimes \mathcal O_W((l-1)E)\\ 
\downarrow & & \downarrow\\ 
0 &  &0
\end{matrix}
$$ 
By applying $\mu_*$, we have the next commutative diagram. 
$$
\begin{matrix}
0 & &0\\
\downarrow & &\downarrow\\ 
\mu_*(R^ig_*\omega _{X/W}(D)\otimes \mathcal O_W((l-2)E))
&\to &
\mu_*(R^ig_*\omega _{X/W}(D)\otimes \mathcal O_W((l-1)E)) \\
\downarrow & & \downarrow\\ 
(\Gr^0_F(\mathcal V^{d-i}_Y))^*
& \simeq &
(\Gr^0_F(\mathcal V^{d-i}_Y))^*\\
\downarrow & & \downarrow\\
\mu_*(G_W^{i}\otimes \mathcal O_W((l-2)E))
&\to & 
\mu_*(G_W^{i}\otimes \mathcal O_W((l-1)E))\\
& & \downarrow\\ 
&  &0
\end{matrix}
$$ 
By (A) and (B), 
$G_Y^i\simeq \mu_*(G_W^i\otimes \mathcal O_W((l-1)E))$ 
and 
$$
\mu_*(G_W^i\otimes \mathcal O_W((l-2)E))\longrightarrow 
\mu_*(G_W^i\otimes \mathcal O_W((l-1)E)) 
$$ 
is surjective. 
Since $\dim \Supp G_W^i=0$ and $E\cap \Supp G_W^i\ne \emptyset$, 
it follows that $G^i_W=0$ by Nakayama's lemma. 
Therefore, $G_Y^i=0$. 
This implies $R^if_*\omega_{X/Y}(D)\simeq 
(\Gr^0_F(\mathcal V^{d-i}_Y))^*$. 
By Grothendieck duality, 
$R^{d-i}f_*\mathcal O_X(-D)\simeq \Gr^0_F(\mathcal V^{d-i}_Y)$. 
\end{step}
From now on, we treat the general case, that is, we do not assume that 
local monodromies are unipotent. 
\begin{step}\label{step33} 
In this step, we 
prove the local freeness of 
$R^if_*\omega_{X/Y}(D)$ for every $i$. 
We use the unipotent reduction with 
respect to all the local systems 
after shrinking $Y$ suitably. 
This means that, shrinking $Y$, we have 
the following commutative diagram: 
$$
\begin{CD}
X@<{\alpha}<<X' @<{\beta}<<{\widetilde X}\\ 
@V{f}VV @V{f'}VV @VV{\widetilde f}V \\
Y@<<{\tau}<Y'@=Y',  
\end{CD}
$$ 
which satisfies the following properties. 
\begin{enumerate}
\item[(i)] $\tau:Y'\longrightarrow Y$ is a finite Kummer covering from a nonsingular 
variety $Y'$ and $\tau$ ramifies 
only along $\Sigma$. 
\item[(ii)] $f':X'\longrightarrow Y'$ is the base change of $f:X\longrightarrow Y$ by $\tau$ over $Y\setminus \Sigma$. 
\item[(iii)] $(X', \alpha^*D)$ is a 
semi divisorial log terminal pair in the sense of 
Koll\'ar (see Definition \ref{def-sdlt}). 
Let $X_j$ be any irreducible component of $X$. 
Then $X'_j =\alpha^{-1}(X_j)$ is the 
normalization of the base change of $X_j\longrightarrow Y$ by $\tau:Y'\longrightarrow Y$ and 
$X'=\bigcup _j X'_j$. We note that $X'_j$ is a $V$-manifold for 
every $j$. 
More precisely, 
$X'_j$ is toroidal for every $j$.  
\item[(iv)] $\beta$ is a projective birational 
morphism from a simple normal crossing variety $\widetilde X$ 
and $\widetilde D\cup \Exc(\beta)$ is a simple normal crossing divisor 
on $\widetilde X$, where $\widetilde D$ is the strict transform of 
$\alpha^*D$ (cf.~\cite[Theorem 1.5]{bierstone-milman} and \cite[Theorem 1.4]{bierstone-p}). 
We may further assume that 
$\beta$ is an isomorphism over the 
largest Zariski open set $U$ of $X'$ 
such that 
$(X', \alpha^*D)|_{U}$ is a simple normal crossing pair. 
\item[(v)] $\widetilde f: \widetilde X\longrightarrow Y'$, $\widetilde D$, and 
$\tau^{-1}\Sigma$ satisfy the conditions 
and assumptions in Theorem \ref{main} and all the local monodromies on all the local 
systems around $\tau^{-1}\Sigma$ 
are unipotent. 
\end{enumerate}
Therefore, $R^i\widetilde f_*\omega_{\widetilde X}(\widetilde D)$ is locally free by Step \ref{step1} and 
Step \ref{step2}. 
On the other hand, 
we can prove 
$$R^p\widetilde f_*\omega_{\widetilde X}(\widetilde D)\simeq R^pf'_*\omega_
{X'}(\alpha ^*D)$$ for every $p\geq 0$.  
We note that 
$$
K_{\widetilde X}+\widetilde D=\beta^*(K_{X'}+\alpha^*D)+F
$$ 
where $F$ is $\beta$-exceptional, 
$F$ is permissible on $\widetilde X$, $\Supp F$ is a simple normal crossing divisor on $\widetilde X$, 
and 
$\lceil F\rceil$ is effective. 
Thus we obtain that $\beta_*\omega_{\widetilde X}(\widetilde D)\simeq \omega_{X'}(\alpha^*D)$ and 
that $R^q\beta_*\omega_{\widetilde X}(\widetilde D)=0$ for every $q>0$ by 
Lemma \ref{rf}. 
Thus, $R^if'_*\omega_
{X'}(\alpha ^*D)$ is locally free for every $i$. 
Since $R^if_*\omega_X(D)$ is a direct summand of $$\tau_*
R^if'_*\omega_{X'}(\alpha ^*D)\simeq R^if_*(\alpha_*\omega_{X'}(\alpha^*D)),$$ we obtain 
that $R^if_*\omega_X(D)$ is locally free, equivalently, 
$R^if_*\omega_{X/Y}(D)$ is locally free for every $i$. 
We note that, by Grothendieck duality, $R^{d-i}f_*\mathcal O_X(-D)$ is also locally free for every $i$. 
\end{step}
\begin{step}\label{step44}
In this last step, 
we prove that $R^{d-i}f_*\mathcal O_X(-D)$ is the 
lower canonical extension for every $i$. 
By Grothendieck duality and Step \ref{step33}, 
$R^{d-i}\widetilde {f}_*\mathcal O_{\widetilde X}(-\widetilde D)$ is locally free. 
By Step \ref{step33}, we obtain 
$R\beta_*\omega_{\widetilde X}(\widetilde D)\simeq \omega_{X'}(\alpha^*D)$ in the derived 
category of coherent sheaves on $X'$. Therefore, 
we obtain 
\begin{align*}
R\beta_*\mathcal O_{\widetilde X}(-\widetilde D)&\simeq R\mathcal H om 
(R\beta_*\omega^\bullet_{\widetilde X}(\widetilde D), 
\omega^\bullet_{X'})\\
&\simeq R\mathcal H om (\omega^\bullet_{X'}
(\alpha^*D), \omega^\bullet_{X'})\simeq \mathcal O_{X'}(-\alpha^*D)
\end{align*} 
in the derived category of coherent sheaves on $X'$. 
Note that $X'$ is Cohen--Macaulay (cf.~\cite[Theorem 4.2]{fujino-vanishing}) 
and that $\omega^\bullet_{X'}\simeq \omega_{X'}[\dim X']$. 
Thus, we have  
$$R^p\widetilde f_*\mathcal O_{\widetilde X}(-\widetilde D)\simeq R^pf'_*\mathcal O_{X'}(-\alpha^*D)$$ 
for every $p$. Let $G$ be the Galois group 
of $\tau:Y'\longrightarrow Y$. 
Then we have 
\begin{align*}
(\tau_*R^pf'_*\mathcal O_{X'}(-\alpha^*D))^G\simeq 
R^pf_*(\alpha_*\mathcal O_{X'}(-\alpha^*D))^G\simeq 
R^pf_*\mathcal O_X(-D). 
\end{align*}
Thus, we obtain that $R^{d-i}f_*\mathcal O_X(-D)$ is the 
lower canonical extension for every $i$ (cf.~\cite[Notation 2.5 (iii)]{ko2}). 
By Grothendieck duality, 
$R^if_*\omega_{X/Y}(D)$ is the upper canonical extension for every $i$. 
\end{step}
We finish the proof of Theorem \ref{main} and Theorem \ref{main2}. 
\end{proof}

The following theorem is a generalization of \cite[Proposition 7.6]{kollar1}. 

\begin{thm}
Let $f:X\longrightarrow Y$ be a projective surjective morphism 
from a simple normal crossing variety to a smooth 
algebraic variety $Y$ with connected fibers. 
Assume that every stratum of $X$ is dominant onto $Y$. 
Then $R^df_*\omega_X\simeq \omega_Y$ where 
$d=\dim X-\dim Y$. 
\end{thm}
\begin{proof} 
By \cite[Theorem 1.5]{bierstone-milman} and \cite[Theorem 1.4]{bierstone-p}, 
we can construct a commutative diagram 
$$
\begin{CD}
V@>{\pi}>> X\\ 
@V{g}VV @VV{f}V\\
W@>>{p}>Y
\end{CD}
$$ 
with the following properties. 
\begin{itemize}
\item[(i)] $p:W\longrightarrow Y$ is a projective birational morphism from a smooth quasi-projective 
variety $W$. 
\item[(ii)] $V$ is a simple normal crossing variety. 
\item[(iii)] $\pi$ is projective birational and $\pi$ induces an isomorphism $\pi^0=
\pi|_{V^0}:V^0\longrightarrow X^0$ where 
$X^0$ (resp.~$V^0$) is a Zariski open set of $X$ (resp.~$V$) which 
contains the generic point of any stratum of $X$ (resp.~$V$). 
\item[(iv)] $g$ is projective. 
\item[(v)] there is a simple normal crossing divisor $\Sigma$ on $W$ such that 
every stratum of $V$ is smooth over $W\setminus \Sigma$. 
\end{itemize}
We note that $R^jg_*\omega_V$ is locally free for every $j$ by Theorem \ref{main2}. 
By Grothendieck duality, 
we have
$$
Rg_*\mathcal O_V\simeq R\mathcal {H}om _{\mathcal O_W}(Rg_*\omega^{\bullet}_V, \omega^{\bullet}_W). 
$$ 
Therefore, we have 
$$
\mathcal O_W\simeq \mathcal {H}om_{\mathcal O_W}(R^dg_*\omega_V, \omega_W). 
$$ 
Note that, 
by Zariski's main theorem, $f_*\mathcal O_X\simeq \mathcal O_Y$ 
since every 
stratum of $X$ is dominant onto $Y$. 
Therefore, $g_*\mathcal O_V\simeq \mathcal O_W$. 
Thus, we obtain $R^dg_*\omega_V\simeq \omega_W$. 
By applying $p_*$, we have $p_*R^dg_*\omega_V\simeq p_*\omega_W\simeq \omega_Y$. 
We note that 
$p_*R^dg_*\omega_V\simeq R^d(p\circ g)_*\omega_V$ since 
$R^ip_*R^dg_*\omega_V=0$ for every $i>0$ (cf.~Theorem \ref{quasi-proj} (ii)). 
On the other hand, 
$$
R^d(p\circ g)_*\omega_V\simeq  R^d(f\circ \pi)_*\omega_V\simeq R^df_*\omega_X 
$$ 
since $R^i\pi_*\omega_V=0$ for every $i>0$ by Lemma \ref{rf} 
and $\pi_*\omega_V\simeq \omega _X$ (cf.~Lemma \ref{lemA}). 
Therefore, we obtain $R^df_*\omega_X\simeq \omega_Y$. 
\end{proof}

In geometric applications, we sometimes have a projective surjective morphism 
$f:X\longrightarrow Y$ from a simple normal crossing 
variety to a smooth variety $Y$ with {\em{connected fibers}} 
such that every stratum of 
$X$ is mapped onto $Y$. 
The example below shows that in general 
there is no stratum $S$ of $X$ such that 
general fibers of $S\longrightarrow Y$ are connected. 
Therefore, Kawamata's result (\cite[Theorem 1.1]{kawamata}) 
is very restrictive. He assumes that 
$S\longrightarrow Y$ has connected fibers for every stratum $S$ of $X$. 

\begin{ex}\label{disco}
We consider $W=\mathbb P^1\times \mathbb P^1\times \mathbb P^1$. 
Let $p_i:\mathbb P^1\times \mathbb P^1\times \mathbb P^1\longrightarrow \mathbb P^1$ be the 
$i$-th projection for $i=1, 2, 3$. We take general members 
$X_1\in |p^*_1\mathcal O_{\mathbb P^1}(1)\otimes p^*_2\mathcal O_{\mathbb P^1}(2)|$ 
and $X_2\in |p^*_1\mathcal O_{\mathbb P^1}(1)\otimes p^*_3\mathcal O_{\mathbb P^1}(2)|$. 
We define $X=X_1\cup X_2$, $Y=\mathbb P^1$, and 
$f=p_1|_{X}:X\longrightarrow Y$. 
Then $f$ is a projective morphism from a simple normal crossing 
variety $X$ to a smooth projective curve $Y$. 
We can directly check that 
$$
H^1(W, \mathcal O_W(-X_1))=H^1(W, \mathcal O_W(-X_2))=0
$$ 
and 
$$
H^1(W, \mathcal O_W(-X_1-X_2))=H^2(W, \mathcal O_W(-X_1-X_2))=0. 
$$ 
Therefore, by using 
$$
0\longrightarrow \mathcal O_W(-X_1-X_2)\longrightarrow \mathcal O_W(-X_2)\longrightarrow \mathcal O_{X_1}(-X_2)\longrightarrow 0, 
$$ 
we obtain $H^1(X_1, \mathcal O_{X_1}(-X_2))=0$. 
By using 
$$
0\longrightarrow \mathcal O_{X_1}(-X_2)\longrightarrow \mathcal O_{X_1}\longrightarrow \mathcal O_{X_1\cap X_2}\longrightarrow 0, 
$$ 
we obtain $H^0(X_1\cap X_2, \mathcal O_{X_1\cap X_2})=\mathbb C$ since 
$H^0(X_1, \mathcal O_{X_1})=\mathbb C$. 
This means that $C=X_1\cap X_2$ is a smooth 
connected curve. 
Therefore, every stratum of $X$ is mapped onto $Y$ by $f$. 
We note that general fibers of $f:X_1\longrightarrow Y$, 
$f:X_2\longrightarrow Y$, and $f:C\longrightarrow Y$ are disconnected. 
\end{ex}

As a special case of 
Theorem \ref{main}, we obtain the following 
theorem.  

\begin{thm}[{cf.~\cite[Theorem 5]{kawamata1}, 
\cite[Theorem 2.6]{ko2}, 
and \cite[Theorem 1]{n}}]\label{thm42}
Let $f:X\longrightarrow Y$ be a projective morphism 
between smooth complete 
algebraic varieties which satisfies the following conditions{\em{:}}
\begin{itemize}
\item[(i)] There is a Zariski open subset $Y^*$ of $Y$ such that 
$\Sigma =Y\setminus Y^*$ is a simple normal crossing 
divisor on $Y$. 
\item[(ii)] We put $X^*=f^{-1}(Y^*)$. 
Then $f|_{X^*}$ is smooth. 
\item[(iii)] The local monodromies of $R^{d+i}(f|_{X^*})_{*}\mathbb C_{X^*}$ around 
$\Sigma$ are unipotent, where 
$d=\dim X-\dim Y$. 
\end{itemize}
Then $R^if_*\omega_{X/Y}$ is a semipositive locally free sheaf on $Y$. 
\end{thm}

\begin{proof}
By Poincar\'e--Verdier duality (see, for example, \cite[Theorem 13.9]{ps}), 
$R^{d-i}(f|_{X^*})_{*}\mathbb C_{X^*}$ is the 
dual local system of $R^{d+i}(f|_{X^*})_{*}\mathbb C_{X^*}$. 
Therefore, the local monodromies of 
$R^{d-i}(f|_{X^*})_{*}\mathbb C_{X^*}$ around $\Sigma$ 
are unipotent. 
Thus, by Theorem \ref{main}, 
we obtain that 
$R^if_*\omega_{X/Y}\simeq 
(R^{d-i}f_*\mathcal O_X)^*$ 
is a semipositive locally free sheaf on $Y$. 
\end{proof}

Similarly, the semipositivity theorem in \cite{high} (cf.~\cite[Theorem 3.9]{high}) 
can be recovered by Theorem \ref{main} . 
We note that \cite[Theorem 1.1]{kawamata} 
does not cover \cite[Theorem 3.9]{high}. 
This is because Kawamata's theorem needs that $S\longrightarrow Y$ 
has connected fibers for every stratum $S$ of $(X, D)$ 
(cf.~Example \ref{disco}). 

\begin{rem}
Let $f:X\longrightarrow Y$ be a projective morphism between smooth projective varieties. 
Assume that 
there exists a simple normal crossing divisor $\Sigma$ on $Y$ such that 
$f$ is smooth 
over $Y\setminus \Sigma$. 
Then $R^if_*\omega_{X/Y}$ is locally free for every $i$ (cf.~Theorem 
\ref{main2} and \cite[Theorem 2.6]{ko2}). 
We note that 
$R^if_*\omega_{X/Y}$ is not always 
semipositive 
if we assume nothing 
on monodromies around $\Sigma$. 
\end{rem}

We close this section with an easy example. 

\begin{ex}[Double cover]\label{ex510}
We consider $\pi:Y=\mathbb P_{\mathbb P^1}(\mathcal O_{\mathbb P^1}\oplus 
\mathcal O_{\mathbb P^1}(2))\longrightarrow \mathbb P^1$. 
Let $E$ and $G$ be the sections of 
$\pi$ such 
that $E^2=-2$ and $G^2=2$. 
We note that 
$E+2F\sim G$ where $F$ is a fiber of $\pi$. 
We put 
$\mathcal L=\mathcal O_Y(E+F)$. 
Then $E+G\in |\mathcal L^{\otimes 2}|$. 
Let $f:X\longrightarrow Y$ be the double 
cover constructed 
by $E+G\in |\mathcal L^{\otimes 2}|$. 
Then $f:X\longrightarrow Y$ is \'etale outside $\Sigma =E+G$ and 
$$f_*\omega_{X/Y}\simeq \mathcal O_Y\oplus \mathcal L. 
$$
In this case, $f_*\omega_{X/Y}$ is not semipositive since $\mathcal L\cdot E=-1$. 
We note that 
the local monodromies 
on $(f|_{X^*})_{*}\mathbb C_{X^*}$ around 
$\Sigma$ are not unipotent, 
where $Y^*=Y\setminus \Sigma$ and 
$X^*=f^{-1}(Y^*)$. 
\end{ex}

In Example \ref{ex510}, $f:X\longrightarrow Y$ is finite and the general fibers of $f$ are disconnected. 
In Section \ref{sec-final}, we discuss an example $f:X\longrightarrow Y$ whose general fibers are elliptic curves 
such that $f_*\omega_{X/Y}$ is not semipositive (cf.~Corollary \ref{cor7} and 
Example \ref{214-a}). 

\section{Examples}\label{sec-final}

In this final section, we give supplementary examples for the Fujita--Kawamata 
semipositivity theorem (cf.~\cite[Theorem 5]{kawamata1}), 
Viehweg's weak positivity theorem, and the Fujino--Mori canonical bundle formula (cf.~\cite{fm}). 
For details of the original Fujita--Kawamata semipositivity 
theorem, see, for example, \cite[\S 5]{mori} and \cite[Section 5]{fujino-rem}. 

\begin{say}[Semipositivity in the sense of Fujita--Kawamata]
The following example is due to Takeshi Abe. It is a small 
remark on Definition \ref{def620}. 

\begin{ex}\label{ex-abe}
Let $C$ be an elliptic curve and let $E$ be a stable 
vector bundle on $C$ such that 
the degree of $E$ is $-1$ and 
the rank of $E$ is two. 
Let $f_m:C\longrightarrow C$ be the multiplication by $m$ where 
$m$ is a positive integer. 
In this case, every quotient line bundle 
$L$ of $E$ has non-negative degree. 
However, $\mathcal O_{\mathbb P(E)}(1)$ is not nef. 
It is because we can find a quotient line bundle 
$M$ of $f_m^*E$ whose 
degree is negative for some positive integer $m$.  
\end{ex}
\end{say}

\begin{say}[Canonical bundle formula]
We give sample computations of our canonical bundle formula obtained 
in \cite{fm}. 
We will freely use the notation in \cite{fm}. 
For details of our canonical bundle formula, 
see \cite{fm}, \cite[\S 3]{fujino-nagoya1}, 
and \cite[\S 3, \S 4, \S 5, and \S 6]{fujino-nagoya2}. 
\end{say} 

\begin{say}[Kummer manifolds]
Let $E$ be an elliptic curve and let $E^n$ be the 
$n$-times direct product of $E$. 
Let $G$ be the cyclic 
group of order two of analytic 
automorphisms of $E^n$ generated 
by an automorphism 
$$\tau:E^n\longrightarrow E^n: (z_1, \cdots, z_n)\mapsto 
(-z_1, \cdots, -z_n). $$
The automorphism $\tau$ has $2^{2n}$ fixed points. 
Each singular point is terminal for $n\geq 3$ and 
is canonical for $n\geq 2$. 
\end{say}
\begin{say}[Kummer surfaces]\label{22ku}
First, we consider 
$q:E^2/G\longrightarrow E/G\simeq \mathbb P^1$, 
which is induced by the first projection, and 
$g=q\circ \mu: Y\longrightarrow \mathbb P^1$, where 
$\mu:Y\longrightarrow E^2/G$ is the minimal resolution of sixteen 
$A_1$-singularities. 
It is easy to see that $Y$ is a $K3$ surface. 
In this case, it is obvious that 
$$g_*\mathcal O_Y(mK_{Y/\mathbb P^1})\simeq \mathcal 
O_{\mathbb P^1}(2m)$$ for every $m\geq 1$. 
Thus, we can put $L_{Y/\mathbb P^1}=D$ for 
any degree two Weil divisor $D$ on $\mathbb P^1$. 
For the definition of $L_{Y/\mathbb P^1}$, see \cite[Definition 2.3]{fm}.  
We obtain 
$K_Y=g^*(K_{\mathbb P^1}+L_{Y/\mathbb P^1})$. 
Let $Q_i$ be the branch point of $E\longrightarrow E/G\simeq \mathbb P^1$ 
for $1\leq i\leq 4$. 
Then we have 
$$L^{ss}_{Y/\mathbb P^1}=D-\sum _{i=1}^{4}\left(1-\frac{1}{2}\right)
Q_i=D-\sum _{i=1}^{4}\frac{1}{2}Q_i$$ 
by the definition of the semi-stable part $L^{ss}_{Y/\mathbb P^1}$ 
(see \cite[Proposition 2.8, Definition 4.3, and Proposition 4.7]{fm}).   
Therefore, we obtain 
$$K_Y=g^*\left(K_{\mathbb P^1}+L^{ss}_{Y/\mathbb P^1}
+\sum _{i=1}^{4}\frac{1}{2}Q_i\right). $$  
Thus, $$L^{ss}_{Y/\mathbb P^1}=D-\sum _{i=1}^{4}\frac{1}{2} 
Q_i\not \sim 0$$ but 
$$2L^{ss}_{Y/\mathbb P^1}=2D-\sum _{i=1}^{4}Q_i\sim 0.$$  
Note that $L^{ss}_{Y/\mathbb P^1}$ is not a Weil 
divisor but a $\mathbb Q$-Weil divisor on $\mathbb P^1$.  
\end{say}
\begin{say}[Elliptic fibrations]\label{23}Next, 
we consider $E^3/G$ and $E^2/G$. 
We consider the morphism $p:E^3/G\longrightarrow E^2/G$ induced by 
the projection $E^3\longrightarrow E^2: (z_1, z_2, z_3)\mapsto
(z_1, z_2)$. Let $\nu:X'\longrightarrow 
E^3/G$ 
be the weighted blow-up of $E^3/G$ at 
sixty-four $\frac{1}{2}(1, 1, 1)$-singularities. 
Thus 
$$
K_{X'}=\nu^*K_{E^3/G}+\sum _{j=1}^{64}\frac{1}{2}
E_j, 
$$ 
where $E_j\simeq \mathbb P^2$ is the 
exceptional divisor for 
every $j$. 
Let $P_i$ 
be an $A_1$-singularity of $E^2/G$ for $1\leq i\leq 16$. 
Let $\psi:X\longrightarrow X'$ be the blow-up of $X'$ along 
the strict transform of $p^{-1}(P_i)$, which is 
isomorphic to $\mathbb P^1$,  
for every $i$. Then we obtain the following 
commutative diagram. 
$$
\begin{CD}
E^3/G@<{\phi:=\nu\circ\psi}<<X\\ 
@V{p}VV @VV{f}V\\
E^2/G@<<{\mu}<Y
\end{CD}
$$ 
Note that $$K_X=\phi^*K_{E^3/G}+\sum _{j=1}^{64}\frac{1}{2}
E_j+\sum _{k=1}^{16}F_k, $$ where 
$E_j$ is 
the strict transform of $E_j$ on $X$ and 
$F_k$ is the $\psi$-exceptional 
prime divisor for every $k$. 
We can check that $X$ is a smooth projective threefold. 
We put $C_i=\mu ^{-1}(P_i)$ for every $i$. 
It can be checked that $C_i$ is a $(-2)$-curve for every 
$i$. 
It is easily checked that $f$ is smooth outside $\sum _{i=1}^{16} 
C_i$ and that the degeneration of $f$ is of 
type $I^*_0$ along 
$C_i$ for every $i$. 
We renumber $\{E_j\}_{j=1}^{64}$ as $\{E_i^j\}$, 
where $f(E_i^j)=C_i$ for every $1\leq i\leq 16$ and $1\leq 
j\leq 4$. 
We note that $f$ is flat since $f$ is equidimensional. 

Let us recall the following theorem (cf.~\cite[Theorem 20]{kawamata2} 
and \cite[Corollary 3.2.1 and Theorem 3.2.3]{nakayama}). 

\begin{thm}[..., Kawamata, Nakayama, ...]\label{nakayama} 
We have the following isomorphism. 
$$(f_*\omega_{X/Y})^{\otimes 12}\simeq 
\mathcal O_Y\left(\sum _{i=1}^{16}6C_i\right), $$ 
where $\omega_{X/Y}\simeq \mathcal O_X(K_{X/Y})=\mathcal 
O_X(K_X-f^*K_Y)$. 
\end{thm} 
The proof of Theorem \ref{nakayama} depends on the 
investigation of the upper canonical extension of 
the Hodge filtration and the period map. 
It is obvious that 
$$
2K_X=f^*\left(2K_Y+\sum _{i=1}^{16}C_i\right)$$ 
and 
$$
2mK_X=f^*\left(2mK_Y+m\sum _{i=1}^{16}C_i\right)$$ 
for all $m\geq 1$ since $f^*C_i=2F_i+\sum _{j=1}^{4} 
E_i^j$ for every $i$. 
Therefore, we have $2L_{X/Y}\sim \sum _{i=1}^{16}C_i$. 
On the other hand, $f_*\omega_{X/Y}\simeq 
\mathcal O_Y(\lfloor L_{X/Y}\rfloor)$. 
Note that $Y$ is a smooth surface and $f$ is flat. 
Since $$\mathcal O_Y(12\lfloor L_{X/Y}\rfloor )
\simeq (f_*\omega_{X/Y})^{\otimes 12}\simeq 
\mathcal O_Y\left(\sum _{i=1}^{16}6C_i\right),$$ we have 
$$12L_{X/Y}\sim 6\sum _{i=1}^{16}C_i\sim 
12\lfloor L_{X/Y}\rfloor. $$  
Thus, $L_{X/Y}$ is a Weil divisor on $Y$. It is because 
the fractional part $\{L_{X/Y}\}$ is effective 
and linearly equivalent to zero. 
So, $L_{X/Y}$ is numerically equivalent to 
$\frac{1}{2}\sum _{i=1}^{16}C_i$. 
We have $g^*Q_i=
2G_i+\sum_{j=1}^{4} C_i^j$ for every $i$. 
Here, we renumbered 
$\{C_j\}_{j=1}^{16}$ as 
$\{C_i^j\}_{i,j=1}^{4}$ such 
that $g(C_i^j)=Q_i$ for 
every $i$ and $j$. 
More precisely, we put $2G_i=g^*Q_i-\sum _{j=1}^4 C_i^j$ for 
every $i$. 
We note that we used notations in \ref{22ku}. 
We consider $A:=g^*D-\sum _{i=1}^{4}G_i$. 
Then $A$ is a Weil divisor and $2A\sim 
\sum _{i=1}^{16}C_i$. 
Thus, $A$ is numerically equivalent to 
$\frac{1}{2}\sum _{i=1}^{16}C_i$. 
Since $H^1(Y, \mathcal O_Y)=0$, we 
can put $L_{X/Y}=A$. 
So, we have 
$$L^{ss}_{X/Y}=g^*D-\sum _{i=1}^{4}G_i-\sum _{j=1}^{16} 
\frac{1}{2}C_j.$$
We obtain the following canonical bundle formula.  
\begin{thm} The next formula holds. 
$$K_X=f^*\left(K_Y+L^{ss}_{X/Y}+\sum _{j=1}^{16}\frac{1}{2}C_j\right),$$ 
where $L^{ss}_{X/Y}=g^*D-\sum _{i=1}^{4}G_i-\sum _{j=1}^{16} 
\frac{1}{2}C_j$. 
\end{thm} 
We note that $2L^{ss}_{X/Y}\sim 0$ 
but $L^{ss}_{X/Y}\not \sim 0$. 
The semi-stable part $L^{ss}_{X/Y}$ is 
not a Weil divisor but a $\mathbb Q$-divisor on $Y$.  

The next lemma is obvious since the index 
of $K_{E^3/G}$ is two. We give a direct proof here. 
\begin{lem}
$H^0(Y, L_{X/Y})=0$. 
\end{lem} 
\begin{proof}
Suppose that there exists an effective Weil divisor $B$ on $Y$ such that 
$L_{X/Y}\sim B$. 
Since $B\cdot C_i=-1$, we have 
$B\geq \frac{1}{2}C_i$ for all $i$. 
Thus $B\geq \sum _{i=1}^{16}\frac{1}{2}C_i$. 
This implies that $B-\sum _{i=1}^{16}\frac{1}{2}C_i$ is 
an effective 
$\mathbb Q$-divisor and is numerically 
equivalent to zero. 
Thus $B=\sum _{i=1}^{16}\frac{1}{2}C_i$. 
It is a contradiction. 
\end{proof} 
We can easily check the following corollary. 

\begin{cor}\label{cor7} 
We have 
\begin{equation*}
f_*\omega^{\otimes m}_{X/Y}\simeq 
\begin{cases} 
\mathcal 
O_Y(\sum _{i=1}^{16}nC_i) & {\text{if $m=2n,$}}\\ 
\mathcal O_Y(L_{X/Y}+\sum _{i=1}^{16}nC_i)& {\text{if $m=2n+1$}}.
\end{cases}
\end{equation*} 
In particular, $f_*\omega^{\otimes m}_{X/Y}$ 
is not nef for any $m\geq 1$. 
We can also check that 
\begin{equation*}
H^0(Y, f_*\omega^{\otimes m}_{X/Y})\simeq 
\begin{cases}
\mathbb C &{\text{if $m$ is even,}}\\ 
0 &{\text{if $m$ is odd.}}
\end{cases}
\end{equation*}
\end{cor}
Corollary \ref{cor7} shows that \cite[Theorem 1.9 (1)]{tsuji} 
is sharp. 
\end{say}

\begin{say}[Weak positivity]
Let us recall the definition of Viehweg's weak positivity 
(cf.~\cite[Definition 1.2]{v2} and 
\cite[Definition 2.11]{vie-book}). 
The reader can find some interesting applications of a generalization of 
Viehweg's weak positivity theorem in \cite{fujino-gongyo}. 

\begin{defn}[Weak positivity]\label{2929} 
Let $W$ be a smooth quasi-projective variety and 
let $\mathcal F$ be a locally free sheaf on $W$. 
Let $U$ be an open subvariety of $W$. 
Then, $\mathcal F$ is {\em{weakly positive over $U$}} 
if for every ample invertible sheaf $\mathcal H$ and 
every positive integer $\alpha$ there exists 
some positive integer $\beta$ such that 
$S^{\alpha\cdot \beta}(\mathcal F)\otimes \mathcal H^{\beta}$ is generated 
by global sections over $U$ where $S^k$ denotes the $k$-th symmetric product 
for every positive integer $k$. 
This means that the natural map 
$$
H^0(W, S^{\alpha\cdot\beta}(\mathcal F)\otimes \mathcal H^\beta)\otimes 
\mathcal O_W\longrightarrow S^{\alpha\cdot \beta}(\mathcal F)\otimes \mathcal H^\beta
$$ 
is surjective over $U$. 
\end{defn}

\begin{rem}
[{cf.~\cite[(1.3) Remark.~iii)]{v2}}] 
In Definition \ref{2929}, 
it is enough to check 
the condition for one invertible 
sheaf $\mathcal H$, 
not necessarily ample, and all $\alpha >0$. 
For details, see \cite[Lemma 2.14 a)]{vie-book}. 
\end{rem}

\begin{rem}
In \cite[Definition 3.1]{v3}, 
$S^{\alpha\cdot \beta}(\mathcal F)\otimes \mathcal H^{\otimes \beta}$ is 
only required to be generically generated. 
See also \cite[(5.1) Definition]{mori}. 
\end{rem}

We explicitly check the weak positivity 
for the elliptic fibration constructed in \ref{23} 
(cf.~\cite[Theorem 4.1 and Theorem III]{v2} and 
\cite[Theorem 2.41 and Corollary 2.45]{vie-book}). 

\begin{prop}\label{29}
Let $m$ be a positive integer.  
Let $f:X\longrightarrow Y$ be the elliptic fibration constructed in {\em{\ref{23}}}. 
Then $f_*\omega^{\otimes m}_{X/Y}$ is weakly positive 
over $Y_0=Y\setminus \sum _{i=1}^{16}C_i$. Let $U$ be 
a Zariski open set such that 
$U\not\subset Y_0$. Then 
$f_*\omega^{\otimes m}_{X/Y}$ is 
not weakly positive over $U$.  
\end{prop}
\begin{proof}
Let $H$ be a very ample Cartier divisor on $Y$ such that 
$L_{X/Y}+H$ is very ample. 
We put $\mathcal H=\mathcal O_Y(H)$. 
Let $\alpha$ be an arbitrary positive integer. 
Then 
$$
S^{\alpha}(f_*\omega^{\otimes m}_{X/Y})\otimes \mathcal H
\simeq \mathcal O_Y\left(\alpha \sum _{i=1}^{16}nC_i+H\right)  
$$ 
if $m=2n$. When $m=2n+1$, 
we have 
\begin{align*}
&S^{\alpha}(f_*\omega^{\otimes m}_{X/Y})\otimes \mathcal H
\\ 
&\simeq 
\begin{cases}
\mathcal O_Y(\alpha \sum _{i=1}^{16}nC_i+H+L_{X/Y}+\lfloor 
\frac{\alpha}{2}\rfloor\sum _{i=1}^{16}C_i) \ \ 
\text{if $\alpha$ is 
odd,}\\
\mathcal O_Y(\alpha \sum _{i=1}^{16}nC_i+H+ 
\frac{\alpha}{2}\sum _{i=1}^{16}C_i) 
\ \ 
\text{if $\alpha$ is even.} 
\end{cases}
\end{align*} Thus, 
$S^{\alpha}(f_*\omega^{\otimes m}_{X/Y})\otimes \mathcal H$ 
is generated by global sections over $Y_0$ for every $\alpha>0$. 
Therefore, $f_*\omega^{\otimes m}_{X/Y}$ is weakly positive 
over $Y_0$. 

Let $\mathcal A$ be an ample invertible sheaf on $Y$. 
We put $k=\underset{j}{\max}  (C_j\cdot \mathcal A)$. 
Let $\alpha$ be a positive integer with 
$\alpha >k/2$. 
We note that 
\begin{align*}
S^{2\alpha\cdot \beta}(f_*\omega^{\otimes m}_{X/Y})
\otimes \mathcal A^{\otimes 
\beta}\simeq \left(\mathcal O_Y(\alpha \sum _{i=1}^{16}mC_i)
\otimes \mathcal A\right)^{\otimes \beta}. 
\end{align*} 
If $H^0(Y, S^{2\alpha\cdot \beta}(f_*\omega^{\otimes m}_{X/Y})
\otimes \mathcal A^{\otimes \beta})\ne 0$, 
then we can take 
$$G\in \left|\left(\mathcal O_Y(\alpha \sum _{i=1}^{16}mC_i)
\otimes \mathcal A\right)^{\otimes \beta}\right|. $$ 
In this case, $G\cdot C_i<0$ for every $i$ because $\alpha >k/2$. 
Therefore, we obtain $G\geq \sum _{i=1}^{16}C_i$. 
Thus, $S^{2\alpha \cdot \beta}(f_*\omega^{\otimes m}_{X/Y})\otimes \mathcal 
A^{\otimes \beta}$ is not generated 
by global sections over $U$ for any $\beta\geq 1$. 
This means that $f_*\omega^{\otimes m}_{X/Y}$ 
is not weakly positive over 
$U$.  
\end{proof}

Proposition \ref{29} implies that \cite[Corollary 2.45]{vie-book} 
is the best result. 
\end{say}

\begin{ex}\label{214-a} 
Let $f:X\longrightarrow Y$ be the elliptic fibration 
constructed in \ref{23}. 
Let $Z:=C\times X$, where $C$ is a smooth projective curve with 
the genus $g(C)=r\geq 2$. 
Let $\pi_1:Z\longrightarrow C$ (resp.~$\pi_2:Z\longrightarrow X$) 
be the first (resp.~second) projection. 
We put $h:=f\circ \pi_2:Z\longrightarrow Y$. 
In this case, $K_Z=\pi^*_1K_C\otimes 
\pi^*_2K_X$. 
Therefore, we obtain 
$$h_*\omega ^{\otimes m}_{Z/Y}=
f_*\pi_{2*}(\pi^*_1\omega ^{\otimes m}_{C}
\otimes 
\pi^*_2\omega ^{\otimes m}_{X})\otimes 
\omega^{\otimes -m}_Y=(f_*\omega^{\otimes m}_{X/Y})^{\oplus 
l},$$  
where $l=\dim H^0(C, \mathcal O_C(mK_C))$. 
Thus, $l=
(2m-1)r-2m+1$ if $m\geq 2$ and 
$l=r$ if $m=1$. 
So, $h_*\omega_{Z/Y}$ is a rank $r\geq 2$ vector bundle on $Y$ 
such that $h_*\omega_{Z/Y}$ is not semipositive. 
We note that 
$h$ is smooth over $Y_0=Y\setminus \sum _{i=1}^{16}C_i$. 
We also note that 
$h_*\omega^{\otimes m}_{Z/Y}$ is weakly positive 
over $Y_0$ for every $m \geq 1$ by \cite[Theorem 2.41 and 
Corollary 2.45]{vie-book}. 
\end{ex}
Example \ref{214-a} shows that 
the assumption on the local monodromies 
around $\sum _{i=1}^{16}C_i$ is indispensable 
for the 
semipositivity theorem. 

We close this section with a comment on \cite{fm}. 

\begin{say}[Comment]
We give a remark on \cite[Section 4]{fm}. 
In \cite[4.4]{fm}, $g:Y\longrightarrow X$ is a log resolution of $(X, \Delta)$. 
However, it is better to assume that 
$g$ is a log resolution of $(X, \Delta -(1/b)B^{\Delta})$ for the 
proof of \cite[Theorem 4.8]{fm}. 
\end{say}


\end{document}